\documentclass[12pt]{article}
\usepackage[latin1]{inputenc}
\usepackage[british]{babel}
\usepackage{lmodern}
\usepackage[T1]{fontenc}
\usepackage[paper=a4paper, left=28mm, right=25mm, top=29mm, bottom=25mm]{geometry}
\usepackage{latexsym,amsfonts,amsmath,graphics}
\usepackage{epsfig}

\usepackage[active]{srcltx}
\usepackage{enumitem}
\usepackage{amssymb,mathrsfs}
\usepackage{verbatim}

\newtheorem{theorem}{Theorem}
\newtheorem{lemma}{Lemma}
\newtheorem{corollary}{Corollary}

\newtheorem{definition}{Definition}

\newtheorem{remark}{Remark}
\newenvironment{proof}{\begin{trivlist} \item[\hskip\labelsep{\it Proof.}]}{$\hfill\Box$\end{trivlist}}

\newcommand{\rd}{\,\mathrm{d}}

\newcommand{\bsx}{\boldsymbol{x}}
\newcommand{\bsy}{\boldsymbol{y}}
\newcommand{\bsz}{\boldsymbol{z}}

\newcommand{\bsb}{\boldsymbol{b}}

\newcommand{\bssigma}{\boldsymbol{\sigma}}

\newcommand{\bsone}{\boldsymbol{1}}

\newcommand{\RR}{\mathbb{R}}
\newcommand{\QQ}{\mathbb{Q}}

\newcommand{\FF}{\mathbb{F}}
\newcommand{\NN}{\mathbb{N}}

\newcommand{\ZZ}{\mathbb{Z}}

\newcommand{\sym}{{\rm sym}}
\newcommand{\cP}{\mathcal{P}}

\newcommand{\cH}{\mathcal{H}}

\newcommand{\id}{{\rm id}}

\newcommand{\Sy}{\mathfrak{S}}
\newcommand{\icomp}{\mathtt{i}}

\newcommand{\To}{\rightarrow}

\newcommand{\qed} {\hfill \Box \vspace{0.5cm}}
\allowdisplaybreaks

\title{From van der Corput to modern constructions of sequences for quasi-Monte Carlo rules}
\author{Henri Faure, Peter Kritzer, and Friedrich Pillichshammer\thanks{The last two authors are supported by the
Austrian Science Fund (FWF): Projects F5506-N26 (Kritzer)  and
F5509-N26 (Pillichshammer), respectively, which are part of the
Special Research Program ``Quasi-Monte Carlo Methods: Theory and
Applications''.}}
\date{}

\begin{document}

\maketitle

\centerline{\it Dedicated to the memory of Johannes G. van der
Corput}

\begin{abstract}
In 1935 J.G. van der Corput introduced a sequence which has
excellent uniform distribution properties modulo 1. This sequence
is based on a very simple digital construction scheme with respect
to the binary digit expansion. Nowadays the van der Corput
sequence, as it was named later, is the prototype of many
uniformly distributed sequences, also in the multi-dimensional
case. Such sequences are required as sample nodes in quasi-Monte
Carlo algorithms, which are deterministic variants of Monte Carlo
rules for numerical integration. Since its introduction many
people have studied the van der Corput sequence and
generalizations thereof. This led to a huge number of results.

On the occasion of the 125th birthday of J.G. van der Corput we
survey many interesting results on van der Corput sequences and
their generalizations. In this way we move from van der
Corput's ideas to the most modern constructions of sequences for
quasi-Monte Carlo rules, such as, e.g., generalized Halton
sequences or Niederreiter's $(t,s)$-sequences.
\end{abstract}

\centerline{\begin{minipage}[hc]{130mm}{
}
\end{minipage}}

\allowdisplaybreaks

\tableofcontents
\section{Introduction}

If one is given the task to sequentially distribute points in the unit
interval $[0,1)$ such that for every non-negative integer $N$ the
initial $N$ elements of the resulting sequence fill the interval
very ``uniformly'', then one might after some thinking come to an
answer as illustrated in Figure~\ref{f1}.
\begin{figure}[htp]
\begin{center}
\begin{picture}(0,0)%
\includegraphics{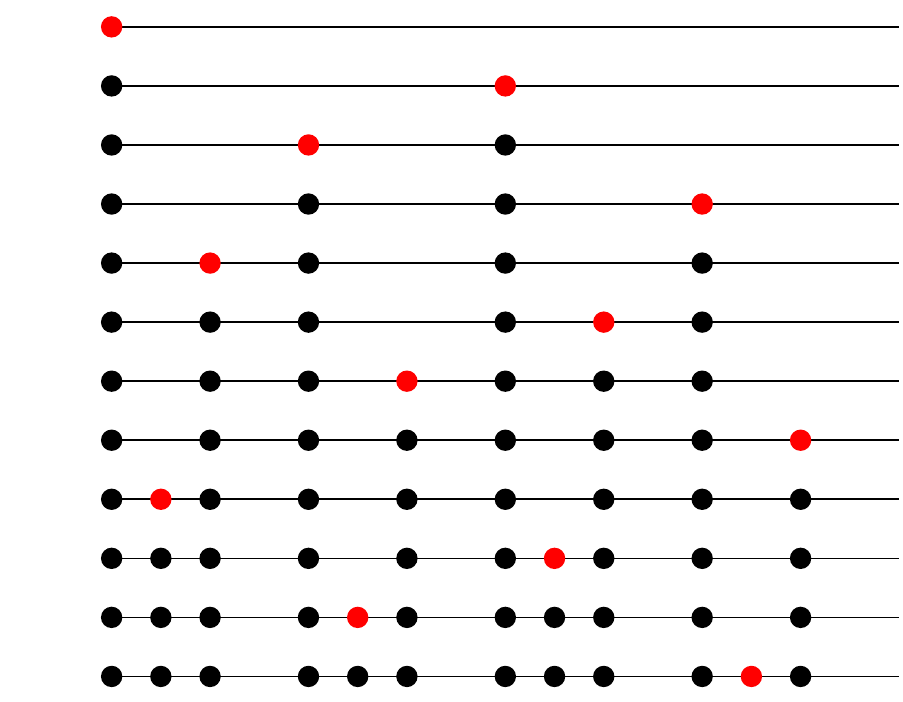}%
\end{picture}%
\setlength{\unitlength}{4144sp}%
\begingroup\makeatletter\ifx\SetFigFont\undefined%
\gdef\SetFigFont#1#2#3#4#5{%
  \reset@font\fontsize{#1}{#2pt}%
  \fontfamily{#3}\fontseries{#4}\fontshape{#5}%
  \selectfont}%
\fi\endgroup%
\begin{picture}(4122,3198)(1516,-3595)
\put(1531,-1366){\makebox(0,0)[lb]{\smash{{\SetFigFont{12}{14.4}{\rmdefault}{\mddefault}{\updefault}{$3:$}%
}}}}
\put(1531,-1096){\makebox(0,0)[lb]{\smash{{\SetFigFont{12}{14.4}{\rmdefault}{\mddefault}{\updefault}{$2:$}%
}}}}
\put(1531,-826){\makebox(0,0)[lb]{\smash{{\SetFigFont{12}{14.4}{\rmdefault}{\mddefault}{\updefault}{$1:$}%
}}}}
\put(1531,-556){\makebox(0,0)[lb]{\smash{{\SetFigFont{12}{14.4}{\rmdefault}{\mddefault}{\updefault}{$0:$}%
}}}}
\put(1531,-1906){\makebox(0,0)[lb]{\smash{{\SetFigFont{12}{14.4}{\rmdefault}{\mddefault}{\updefault}{$5:$}%
}}}}
\put(1531,-2176){\makebox(0,0)[lb]{\smash{{\SetFigFont{12}{14.4}{\rmdefault}{\mddefault}{\updefault}{$6:$}%
}}}}
\put(1531,-2446){\makebox(0,0)[lb]{\smash{{\SetFigFont{12}{14.4}{\rmdefault}{\mddefault}{\updefault}{$7:$}%
}}}}
\put(1531,-2716){\makebox(0,0)[lb]{\smash{{\SetFigFont{12}{14.4}{\rmdefault}{\mddefault}{\updefault}{$8:$}%
}}}}
\put(1531,-2986){\makebox(0,0)[lb]{\smash{{\SetFigFont{12}{14.4}{\rmdefault}{\mddefault}{\updefault}{$9:$}%
}}}}
\put(1531,-3256){\makebox(0,0)[lb]{\smash{{\SetFigFont{12}{14.4}{\rmdefault}{\mddefault}{\updefault}{$10:$}%
}}}}
\put(1531,-3526){\makebox(0,0)[lb]{\smash{{\SetFigFont{12}{14.4}{\rmdefault}{\mddefault}{\updefault}{$11:$}%
}}}}
\put(1531,-1636){\makebox(0,0)[lb]{\smash{{\SetFigFont{12}{14.4}{\rmdefault}{\mddefault}{\updefault}{$4:$}%
}}}}
\end{picture}%
\caption{Intuitive construction of a sequence whose initial elements fill the unit interval very uniformly.}
     \label{f1}
\end{center}
\end{figure}
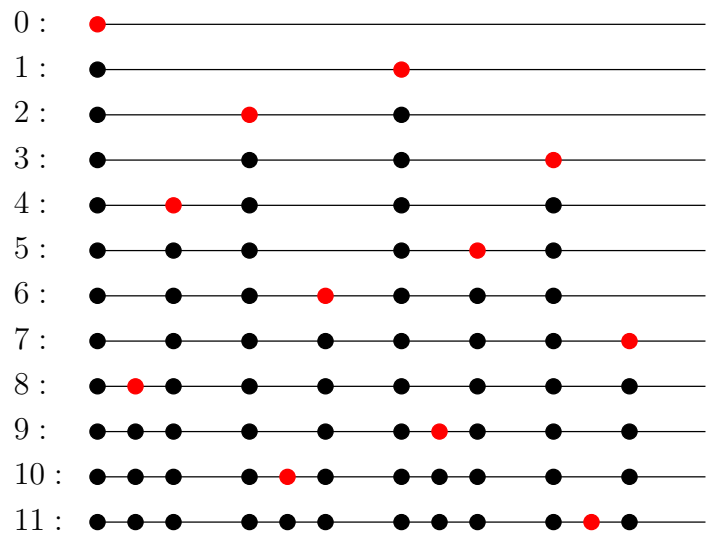

That is, one might discover the sequence
\begin{equation}\label{intuvdc}
0,\frac{1}{2},\frac{1}{4},\frac{3}{4},\frac{1}{8},\frac{5}{8},\frac{3}{8},
\frac{7}{8},\frac{1}{16},\frac{9}{16},\frac{5}{16},\frac{13}{16},\ldots.
\end{equation}

A similar task is to distribute points from a sequence very
uniformly on a circle. In this case, a possible idea may be to choose two
antipodal points on the circle. Then rotate the whole picture by
$\pi/2$ to obtain two further points. Next rotate the picture by
$\pi/4$ to obtain 4 new points. Next rotate by $\pi/8$, and so on
(see Figure~\ref{f1a}). In this way one constructs a sequence on the
torus which directly corresponds to the findings in
Figure~\ref{f1} when the unit interval is coiled up into a circle.

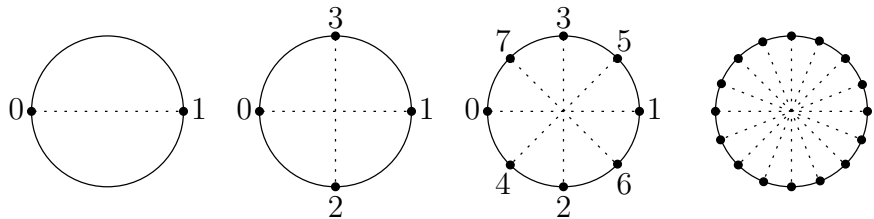
\begin{figure}[htp]
\begin{center}
\begin{picture}(0,0)%
\includegraphics{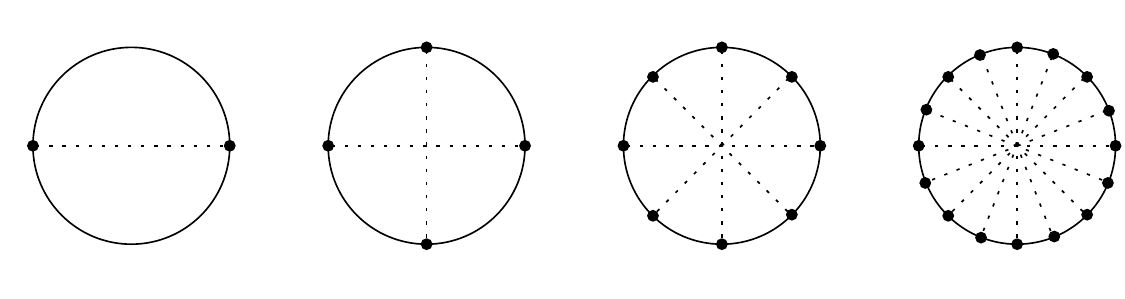}%
\end{picture}%
\setlength{\unitlength}{4144sp}%
\begingroup\makeatletter\ifx\SetFigFont\undefined%
\gdef\SetFigFont#1#2#3#4#5{%
  \reset@font\fontsize{#1}{#2pt}%
  \fontfamily{#3}\fontseries{#4}\fontshape{#5}%
  \selectfont}%
\fi\endgroup%
\begin{picture}(5131,1353)(1651,-1660)
\put(3556,-1591){\makebox(0,0)[lb]{\smash{{\SetFigFont{12}{14.4}{\rmdefault}{\mddefault}{\updefault}{$2$}%
}}}}
\put(3556,-466){\makebox(0,0)[lb]{\smash{{\SetFigFont{12}{14.4}{\rmdefault}{\mddefault}{\updefault}{$3$}%
}}}}
\put(4906,-1591){\makebox(0,0)[lb]{\smash{{\SetFigFont{12}{14.4}{\rmdefault}{\mddefault}{\updefault}{$2$}%
}}}}
\put(4906,-466){\makebox(0,0)[lb]{\smash{{\SetFigFont{12}{14.4}{\rmdefault}{\mddefault}{\updefault}{$3$}%
}}}}
\put(4546,-1456){\makebox(0,0)[lb]{\smash{{\SetFigFont{12}{14.4}{\rmdefault}{\mddefault}{\updefault}{$4$}%
}}}}
\put(5266,-601){\makebox(0,0)[lb]{\smash{{\SetFigFont{12}{14.4}{\rmdefault}{\mddefault}{\updefault}{$5$}%
}}}}
\put(5266,-1456){\makebox(0,0)[lb]{\smash{{\SetFigFont{12}{14.4}{\rmdefault}{\mddefault}{\updefault}{$6$}%
}}}}
\put(4546,-601){\makebox(0,0)[lb]{\smash{{\SetFigFont{12}{14.4}{\rmdefault}{\mddefault}{\updefault}{$7$}%
}}}}
\put(1666,-1006){\makebox(0,0)[lb]{\smash{{\SetFigFont{12}{14.4}{\rmdefault}{\mddefault}{\updefault}{$0$}%
}}}}
\put(2746,-1006){\makebox(0,0)[lb]{\smash{{\SetFigFont{12}{14.4}{\rmdefault}{\mddefault}{\updefault}{$1$}%
}}}}
\put(3016,-1006){\makebox(0,0)[lb]{\smash{{\SetFigFont{12}{14.4}{\rmdefault}{\mddefault}{\updefault}{$0$}%
}}}}
\put(4096,-1006){\makebox(0,0)[lb]{\smash{{\SetFigFont{12}{14.4}{\rmdefault}{\mddefault}{\updefault}{$1$}%
}}}}
\put(4366,-1006){\makebox(0,0)[lb]{\smash{{\SetFigFont{12}{14.4}{\rmdefault}{\mddefault}{\updefault}{$0$}%
}}}}
\put(5446,-1006){\makebox(0,0)[lb]{\smash{{\SetFigFont{12}{14.4}{\rmdefault}{\mddefault}{\updefault}{$1$}%
}}}}
\end{picture}%
\caption{Intuitive construction of a sequence whose initial elements are uniformly distributed on a circle.}
     \label{f1a}
\end{center}
\end{figure}

In the year 1935 Johannes G. van der Corput~\cite{vdc35}
introduced a general construction principle of a sequence in
$[0,1)$ which resembles exactly the sequence \eqref{intuvdc}. This construction
is based on the reflection of the binary digits of non-negative
integers $n$ in the comma. If $n$ has binary expansion $n=n_0+n_1
2+n_2 2^2+\cdots$ (which is of course finite) 
with $n_j \in \{0,1\}$ for $j=0,1,2,\ldots$
, then the $n$th element $y_n$ of the
sequence is given by
$$y_n=\frac{n_0}{2}+\frac{n_1}{2^2}+\frac{n_2}{2^3}+\cdots.$$ The
resulting infinite sequence $(y_n)_{n\ge 0}$ is nowadays called
the
{\it van der Corput sequence in base 2} or the {\it binary (or dyadic) van der Corput sequence}.\\

Since its introduction in 1935 a multitude of mathematicians
studied the van der Corput sequence with respect to its
distributional properties, but also many variations of the van der
Corput sequence have been introduced and studied. A search on
American Mathematical Society's
MathSciNet\footnote{\texttt{www.ams.org/mathscinet/}, last checked
on \today} using the keyword ``van der Corput sequence'' led to 93
hits, and presumably there are many more papers dealing with this
subject but do not explicitly mention van der Corput in
the title. The van der Corput sequence in base 2 can be seen as
the prototype of the most important sequences, even in the
multi-dimensional case, which are nowadays used in modern
quasi-Monte Carlo simulation algorithms.

On the occasion of the 125th anniversary of the birth of Johannes
G. van der Corput (Sept. 4, 1890 - Sept. 25, 1975) we give a
survey of many interesting results
on the van der Corput sequence and we present prominent generalizations thereof. \\

The paper is structured as follows: in Section~\ref{sec_udt} we
present results on uniform distribution modulo 1 of the van der
Corput sequence. Generalizations and variants of the van der
Corput sequence are presented in Section~\ref{sec_vari}. In
Section~\ref{sec_mult} we present some multi-dimensional
generalizations of the van der Corput sequence.\\

Before we proceed, we introduce some notation. By $\NN$ we denote
the set of positive integers $\{1,2,3,\ldots\}$ and we write $\NN_0=\NN
\cup \{0\}$. The natural logarithm is denoted by $\log$. By
$\{x\}$ we denote the fractional part of a real $x$ and $\|x\|$
denotes the distance of $x$ to the nearest integer, i.e.,
$\|x\|=\min(\{x\},1-\{x\})$. For $b\in \NN$, $b \ge 2$, let
$\ZZ_b=\{0,1,\ldots,b-1\}$. Furthermore, let $\Sy_b$ be the set of all
permutations of $\ZZ_b$.

For functions $f,g:\NN \rightarrow \RR^+$, we sometimes write
$g(N) \ll f(N)$ or $g(N)=O(f(N))$, if there exists a constant
$C>0$ such that $g(N) \le C f(N)$ for all $N \in \NN$, $N \ge 2$.
If we would like to stress that $C$ depends on some parameter, say
$p$, this will be indicated by writing $\ll_p$ or $O_p$.

\section{Uniform distribution modulo 1 and the van der Corput sequence}\label{sec_udt}

According to Weyl~\cite{weyl} a real sequence $(x_n)_{n\ge 0}$
is {\it uniformly distributed modulo 1} if for all $0 \le a < b
\le 1$ we have
\begin{equation}\label{def_udt}
\lim_{N \rightarrow \infty} \frac{\#\{n \in \{0,1,\ldots,N-1\}\ :
\ \{x_n\} \in [a,b)\}}{N}=b-a,
\end{equation}
i.e., the relative number of points in $[a,b)$ converges to the
length of the interval. See \cite{DT97,kuinie} for introductions
to the theory of uniform distribution modulo 1.

Two of the first examples of uniformly distributed sequences are
the so-called $(n \alpha)$-sequences for irrational $\alpha$, where
$x_n=n \alpha$ (see \cite{weyl}), and the van der Corput sequence.
Although van der Corput introduced his sequence only with respect
to binary digits, we introduce a first generalization of his idea
to $b$-adic digit expansions.

\begin{definition}\label{vdc1}\rm
Let $b \in \NN$, $b \ge 2$.
\begin{itemize}
\item The {\it $b$-adic radical inverse function}\index{radical
inverse function} is defined as $Y_b: \NN_0 \rightarrow [0,1)$,
$$Y_b(n)=\frac{n_0}{b}+\frac{n_1}{b^2}+\frac{n_2}{b^3}+\cdots,$$
for $n \in \NN_0$ with $b$-adic digit expansion $n=n_0+n_1 b+n_2
b^2+\cdots$, where $n_i \in \{0,1,\ldots,b-1\}$. In other words,
$Y_b(n)$ is the reflection of the $b$-adic digit expansion of $n$
in the comma. \item The {\it van der Corput sequence in base $b$}
is defined as $Y_b:=(y_n)_{n \ge 0}$ with $y_n=Y_b(n)$.
\end{itemize}
\end{definition}

For example, for $b=2$ we have
$$\begin{array}{l | ccccccccc}
n & 0 & 1 & 2 & 3 & 4 & 5 & 6 & 7 & 8\\
\hline
\mbox{binary $n$} & 0. & 1. & 10. & 11. & 100. & 101. & 110. & 111. & 1000.\\
\mbox{binary $Y_2(n)$} & 0.0  & 0.1 & 0.01 & 0.11 & 0.001 & 0.101 & 0.011 & 0.111 & 0.0001\\
Y_2(n) & 0 & \frac{1}{2} & \frac{1}{4} & \frac{3}{4} & \frac{1}{8}
& \frac{5}{8} & \frac{3}{8} & \frac{7}{8} & \frac{1}{16}
\end{array}$$

It can be shown by elementary counting methods that the van der
Corput sequence in base
$b$ is uniformly distributed modulo 1. We give a sketch of the proof here.\\

\noindent{\it Proof sketch.} For fixed $m \in \NN$ and $k \in
\{0,1,\ldots,b^m -1\}$ an element $y_n$ of the van der Corput
sequence in base $b$ belongs to
$J_{k,m}=[\tfrac{k}{b^m},\tfrac{k+1}{b^m})$ if and only if $n
\equiv A(k) \pmod{b^m}$ where $A(k)$ is a uniquely determined
integer from $\{0,1,\ldots,b^m-1\}$. Hence exactly one of $b^m$
consecutive elements of the van der Corput sequence belongs to
$J_{k,m}$. This implies that
\begin{equation}\label{anzbadicint}
\#\{n \in \{0,1,\ldots,N-1\}\ : \ y_n \in J_{k,m}\} = \left
\lfloor \frac{N}{b^m}\right\rfloor + \theta\ \ \ \mbox{ with }\
\theta \in \{0,1\},
\end{equation}
and hence
\begin{equation}\label{fairelint}
\lim_{N \rightarrow \infty} \frac{\#\{n \in \{0,1,\ldots,N-1\}\ :
\ y_n \in J_{k,m}\}}{N}=\frac{1}{b^m}={\rm length}(J_{k,m}).
\end{equation}
Arbitrary intervals $[a,b) \subseteq [0,1)$ are approximated from
the interior and the exterior by finite unions of intervals of the
form $J_{k,m}$. In this way the result \eqref{fairelint} carries
over to the general case. This means that the van der Corput
sequence in base $b$ is uniformly distributed modulo 1. $\qed$

For a detailed version of this proof we refer to \cite[Proof of
Proposition~2.10]{LP14}.

\subsection{Discrepancy}

Quantitative versions of \eqref{def_udt} are usually stated in
terms of discrepancy. For a real sequence $X=(x_n)_{n \ge 0}$ the
{\it local discrepancy} is defined, for $x\in [0,1]$, as
$$\Delta_{X,N}(x)=\frac{\#\{n \in \{0,1,\ldots,N-1\}\ : \ \{x_n\}
\in [0,x)\}}{N}-x.$$ Sometimes we will abbreviate the counting
part in the local discrepancy by $A(x;N;X):=\#\{n \in
\{0,1,\ldots,N-1\}\ : \ \{x_n\} \in [0,x)\}$.

Then the {\it extreme discrepancy} is defined as $$D_N(X)=\sup_{0
\le x < x'\le 1}|\Delta_{X,N}(x')-\Delta_{X,N}(x)|.$$ The {\it
star (extreme) discrepancy} is defined as $$D_N^{\ast}(X)=\sup_{0
\le x \le 1}|\Delta_{X,N}(x)|.$$ The star discrepancy of a
sequence $X$ can be interpreted as the $L_{\infty}$-norm of the
local discrepancy. In this vein one often also studies the
$L_p$-norm of the local discrepancy and then speaks of the {\it
$L_p$-discrepancy} of a sequence $X$, i.e.,
$$L_{p,N}(X)=\left(\int_0^1 | \Delta_{X,N}(x)|^p \rd
x\right)^{1/p}\ \ \ \mbox{ for }\ p \ge 1.$$ The {\it diaphony}
(introduced by Zinterhof) is defined as
$$
F_N(X) = \frac{1}{N}\left (2 \sum_{m=1}^{\infty} \frac{1}{m^2}
\bigg | \sum_{n=0}^{N-1} \exp(2\pi\icomp m x_n) \bigg |^2
\right)^{1/2}.
$$
The diaphony is linked to the $L_2$-discrepancy by the formula of Koksma:
\begin{equation}\label{FoKoks}
L_{2,N}^2(X)=\left (\sum_{n=0}^{N-1} \left (\frac{1}{2}-x_n
\right) \right )^2 + \frac{1}{4\pi^2}F_N^2(X).
\end{equation}

The notions of uniform distribution modulo 1 and discrepancy can
be extended to the multi-variate case in the obvious way. See, for
example, \cite{BC,DP10,DT97,kuinie,LP14,Mat99,niesiam}.

\begin{remark}[Some general facts on discrepancy]\label{genfactsdisc}\rm
A real sequence $X$ is uniformly distributed modulo 1 if and only
if $\lim_{N \rightarrow \infty} D_N(X)=0$. For any real sequence
$X$ and any $N$ we have $D_N^{\ast}(X) \le D_N(X) \le 2
D_N^{\ast}(X)$. Furthermore, for any $1 \le p \le q \le \infty$ we
have $L_{p,N}(X) \le L_{q,N}(X) \le D_N^{\ast}(X)$. It was first
shown by W.M. Schmidt~\cite{Schm72distrib} in 1972 (see also
\cite{be1982,lar2014}) that there exists a constant $c>0$ such
that for every infinite real sequence $X$ we have
\begin{equation}\label{lowschmid}
ND_N^{\ast}(X) \ge c \log N \ \ \mbox{ for infinitely many $N \in
\NN$}.
\end{equation}
In a recent paper Larcher~\cite{lar2014} showed that one can
choose $c=0.0646363\ldots$  which is the largest value for $c$
known so far (instead of $0.06011 \ldots$ in \cite{be1982},
resulting from the lower bound $0.12022 \ldots$ for $D$). The
star discrepancy also appears in the Koksma-Hlawka inequality for
estimating the absolute error of a quasi-Monte Carlo algorithm:
for all real functions defined on the unit interval with bounded
total variation $V(f)$, and all sequences $X$ in $[0,1)$ with star
discrepancy $D_N^*(X)$  we have $$\left|\int_0^1 f(x) \rd
x-\frac{1}{N} \sum_{n=0}^{N-1} f(x_n)\right| \le V(f) D_N^*(X).$$
In the multi-dimensional case $V(f)$ is the variation of $f$ in
the sense of Hardy and Krause (see \cite{kuinie,niesiam}).

It is also well known that for every $p\in[1,\infty)$ there exists
a positive number $c_p$ with the property that for every real
sequence $X$ we have
\begin{equation}\label{lowproinov}
NL_{p,N}(X) \ge c_p \sqrt{\log N} \ \ \ \mbox{ for infinitely many
$N \in \NN$}.
\end{equation}
For $p \in (1,\infty)$ this lower bound was shown by 
Proinov~\cite{pro86} based on results of Roth~\cite{Roth} and
W.M. Schmidt~\cite{schX} for finite point sets in dimension two.
Using the method of Proinov in conjunction with a result of
Hal\'{a}sz~\cite{hala} for finite point sets in dimension two the
lower bound follows also for the $L_1$-discrepancy. It follows
from a paper of Hinrichs and Larcher~\cite{HL15} that one
can choose $c_2=0.0515599\ldots$ and from a paper of
Vagharshakyan~\cite{V15} that one can choose $c_1=0.01138\ldots$
(see also \cite{lar15a} for a further discussion).

 Both lower bounds \eqref{lowschmid} and
\eqref{lowproinov} are optimal with respect to the order of magnitude in $N$.
\end{remark}

\subsection{The discrepancy in the binary case}

In this section we collect results on the discrepancy of the
classical van der Corput sequence in base 2. Many of these results
are valid in analogous form for arbitrary bases $b \ge 2$ and even
for generalized van der Corput sequences. These generalized van
der Corput sequences will be introduced in
Section~\ref{secdiscgenvdC} and then we will also collect results
on discrepancy in the general case.

\paragraph{The star discrepancy.}
A first estimate of the star discrepancy of the van der Corput
sequence in base 2 was already provided by van der Corput in his
1935 paper \cite{vdc35}. He showed that
$$
ND_N^*(Y_2) \le \frac{\log N}{\log 2} +1\ \ \ \mbox{ for every $N
\ge 2$}.
$$
Comparing this estimate with \eqref{lowschmid} we find that the
star discrepancy of the van der Corput sequence in base 2 is
optimal with respect to the order of magnitude in $N$. Later many authors
worked to improve the discrepancy bound and to find the smallest
possible constant in the $\log N$-term. The first who found this
optimal constant were Haber~\cite{hab1966} and Tijdeman
(unpublished according to \cite{kuinie}). Here we state the
slightly more explicit result of Tijdeman.
\begin{theorem}[Tijdeman]
For every $N \in \NN$ we have
\begin{equation}\label{discbd_tijde}
ND_N^*(Y_2) \le\frac{\log N}{3 \log 2} +1.
\end{equation}
\end{theorem}

Further very exact results were shown by B\'{e}jian and 
Faure \cite{befa77}. Their main finding is the formula
\begin{equation}\label{disc_henri}
ND_N^*(Y_2) = ND_N(Y_2)=\sum_{r=1}^{\infty}
\left\|\frac{N}{2^r}\right\|  = \sum_{r=1}^m
\left\|\frac{N}{2^r}\right\|+\frac{N}{2^m}\
 \ \ \mbox{ whenever $1 \le N \le 2^m$}.
\end{equation}
This formula gives rise to a very simple generating recursion for
$D_N^{\ast}(Y_2)$. The simplicity of this recursion becomes
especially apparent when we use the notation $D(N):=N
D_N^{\ast}(Y_2)$. Then we have $D(1)=1$  and for any $N\in \NN$
\begin{equation}\label{genrecY2}
D(2N)= D(N), \ \ \ \mbox{ and } \ \ \
D(2N+1)=\tfrac{1}{2}(D(N)+D(N+1)+1).
\end{equation}\
The formula \eqref{disc_henri}, or, even better, the formula
\eqref{genrecY2}, makes it possible to compute $D_N^*(Y_2)$ for every $N
\in \NN$ (see Figure~\ref{f2'}).

\begin{figure}[ht]
\begin{center}
\includegraphics[width=115mm]{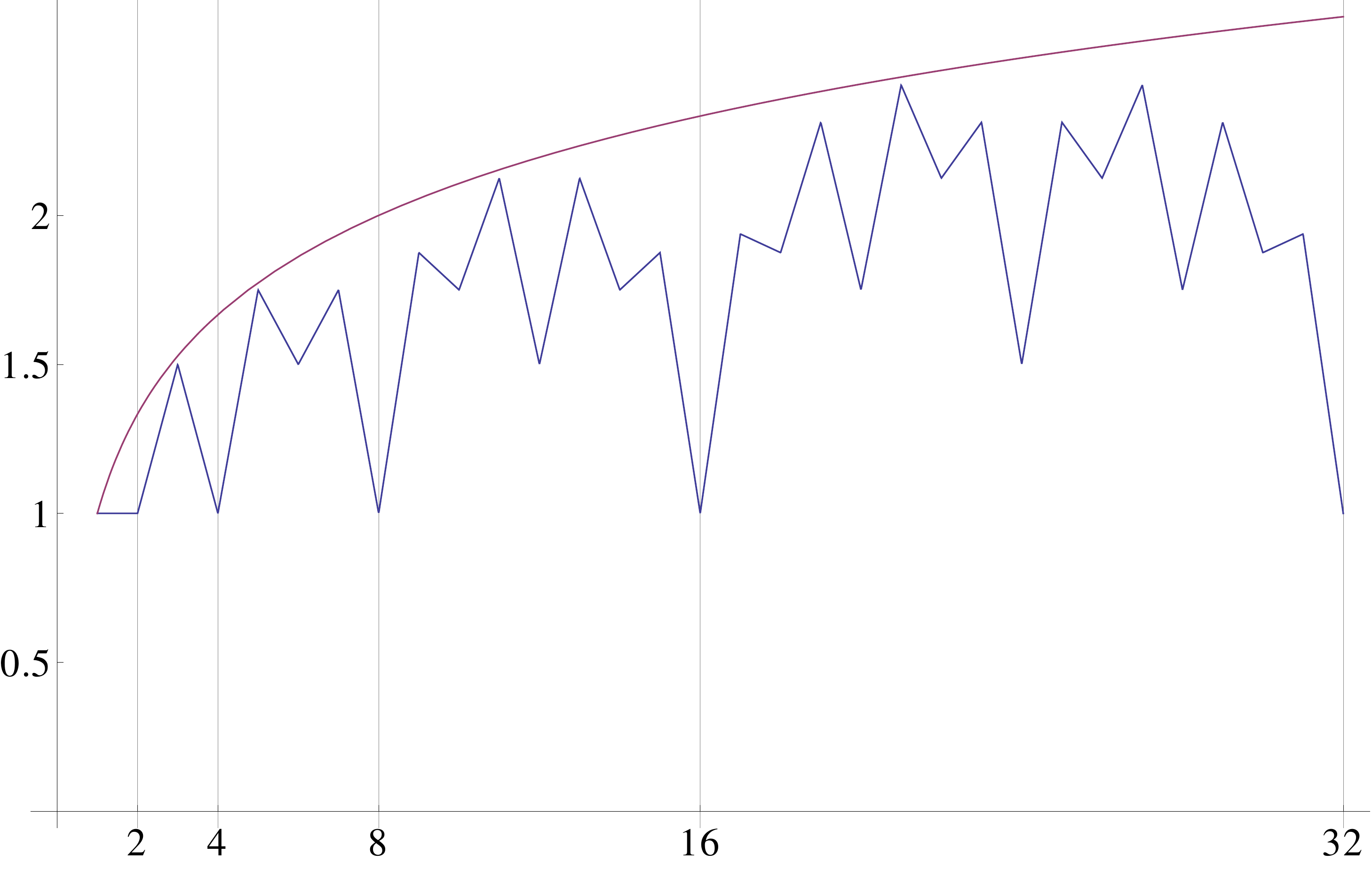}
\caption{$D(N)=N D_N^*(Y_2)$ of the van der Corput sequence in
base 2 for $N=1,2,\ldots,32$, compared with $N\mapsto
\tfrac{\log N}{3 \log 2} +1$. Observe how the
recursion~\eqref{genrecY2} is reflected in the graph, for example,
$D(1)=D(2)=D(4)=D(8)=D(16)=D(32)$.} \label{f2'}
\end{center}
\end{figure}

Based on formula \eqref{disc_henri} B\'{e}jian and Faure
reproved the discrepancy estimate \eqref{discbd_tijde}.
Furthermore they showed:
\begin{theorem}[B\'{e}jian and Faure]\label{thmbfstar}
We have
$$
\limsup_{N \rightarrow \infty} \left(N D_N^{\ast}(Y_2)-\frac{\log
N}{3 \log 2}\right)=\frac{4}{9}+\frac{\log 3}{3 \log 2}
$$
and
$$
\frac{1}{M} \sum_{N=1}^M ND_N^{\ast}(Y_2)=\frac{\log M}{4 \log
2}+O(1).
$$
\end{theorem}
Using (\ref{disc_henri}), B\'{e}jian and Faure also provided the
intervals for which the discrepancy is attained:
\begin{corollary}\label{corbefa}
The reals $\alpha \in [0,1)$ for which
$\Delta_{Y_2,N}(\alpha)=D_N(Y_2)$ are those for which there exists
a sequence $(\varepsilon_r)_{r \ge 1}$ of signs $\varepsilon_r \in
\{+, -\}$ such that
$$
\alpha= - \sum_{r=1}^{\infty} \frac{1}{2^r}
\left\|\frac{N}{2^r}\right\|'_{\varepsilon_r} \pmod{1}
$$
where $\left\| \cdot \right\|'_{+}$ and $\left\| \cdot
\right\|'_{-}$ are the right and the left derivatives
of $\left\| \cdot \right\|$, respectively.
\end{corollary}
Note that the number of the intervals given in Corollary \ref{corbefa} 
is the greatest power of 2 dividing $N$.

\medskip

A further result whose proof is based on \eqref{disc_henri} is the
following central limit theorem for the star discrepancy of $Y_2$.
\begin{theorem}[Drmota, Larcher, and Pillichshammer]\label{thD1}
For every real $y$ and for $M \rightarrow \infty$ we have
$$\frac{1}{M} \# \left\{ N < M : N D_N^*(Y_2) \le \frac{\log N}{4 \log 2} + \frac{y}{4} \sqrt{
\frac{\log N}{3 \log 2}} \right\} = \Phi(y)+ o(1),
$$
where
$$
\Phi(y) = \frac1{\sqrt{2\pi}} \int\limits_{-\infty}^y \exp(-t^2/2)
\rd t
$$
denotes the Gaussian cumulative distribution function. That is,
the star discrepancy of the van der Corput sequence in base $2$
satisfies a central limit theorem.
\end{theorem}

\noindent{\it Proof sketch.} It suffices to prove the equivalent
limit relation
\begin{equation}\label{eqlimitrelation}
\frac 1M \# \left\{ N < M : N D_N^*(Y_2) \le \frac{\log M}{4 \log
2} + \frac{y}{4} \sqrt{\frac{\log M}{3\log 2}} \right\} = \Phi(y)+
o(1).
\end{equation}

Assume first that $M=2^m$. Using formula \eqref{disc_henri} one
can show that $$N D_N^*(Y_2)= \sum_{r=1}^{m} \| \omega 2^r \| +
O(1),$$ where $\omega$ is uniformly distributed on $[0,1)$. Now,
by using well known limit theorems for lacunary sequences (compare
with \cite{gap66} and \cite{phst75}, alternatively we can use
\cite{sunklodas84}) it follows that
$$
S_m = \sum_{r=1}^m \| \omega 2^r \|
$$
satisfies a central limit theorem. Since ${\bf E} [S_m] = m/4$ and
${\bf Var} [S_m] = m/48$ we thus have
$$
\frac{S_m - m/4}{\sqrt{m/48}} \to N(0,1),
$$
and also all moments converge. Since $ND_N^*(Y_2) = S_m + O(1)$ we
get the same limit relation for $ND_N^*(Y_2)$ if $N$ is uniformly
distributed on $\{0,1,\ldots, 2^m-1\}$. This proves the theorem
for $M = 2^m$. The general case, where $M$ is not a power of $2$,
can be reconstructed from the case of $M = 2^m$ by considering
moments. A detailed version of this proof can be found in
\cite[Proof of Theorem~2]{DLP05}. $\qed$

\begin{remark}\rm
It is known from \cite[Th\'{e}or\`{e}me 1]{befa} that
$D_N^*(Y_2)=D_N(Y_2)$. Hence the presented results for the star
discrepancy of $Y_2$ remain valid if we replace the star
discrepancy by the extreme discrepancy.
\end{remark}

Central limit theorems for the star discrepancy of van der Corput
sequences in arbitrary bases $b$ can be found in the PhD thesis
of Wohlfarter~\cite{wohldiss}.

\paragraph{The $L_p$-discrepancy.} For the $L_p$-discrepancy of
the van der Corput sequence in base $2$ the situation is a bit
different than for the star discrepancy. The most studied case is
the case $p=2$. Similar to \eqref{disc_henri} Faure~\cite{fau1990}
proved the following explicit formula:
\begin{equation}\label{discfoL2}
4 (N L_{2,N}(Y_2))^2
=\left(\sum_{j=1}^{\infty}\left\|\frac{N}{2^j}\right\|\right)^2 +
\sum_{j=1}^{\infty}\left\|\frac{N}{2^j}\right\|^2.
\end{equation}
Based on this formula he was able to show the following results
for the $L_2$-discrepancy of the van der Corput sequence in base
$2$:

\begin{theorem}[Faure]\label{thmFL2}
For all $N \in \NN$ we have $$(N L_{2,N}(Y_2))^2 \le
\left(\frac{\log N}{6 \log 2}\right)^2+\left(\frac{11}{3}+\frac{2
\log 3}{\log 2}\right)\frac{\log N}{36 \log 2}+\frac{1}{3}.$$
Moreover
\begin{eqnarray*}
\limsup_{N \rightarrow \infty} \left( (N
L_{2,N}(Y_2))^2-\left(\frac{\log N}{6 \log
2}\right)^2-\left(\frac{11}{3}-\frac{2 \log 3}{\log 2}\right)
\frac{\log N}{36 \log 2}\right)\\ = \frac{7}{81}+\frac{11 \log
3}{108 \log 2}+\left(\frac{\log 3}{6 \log 2}\right)^2.
\end{eqnarray*}
\end{theorem}

A proof of these results can be found in \cite{fau1990}, where
a similar study is also performed for the symmetrized version of
the van der Corput sequence (see Section \ref{sec_sym} for
details).

\medskip

For $p\ge 1$, the following estimates are known.
\begin{theorem}[Pillichshammer]\label{discLpb2}
Let $p\ge 1$. For all $N \in \NN$ we have
\begin{eqnarray}
(N  L_{p,N}(Y_2))^p \le \left(\frac{\log N}{6 \log 2}\right)^p
+O((\log N)^{p-1}),\nonumber
\end{eqnarray}
where the constant in the $O$-notation only depends on $p$. If $N$
is of the form
$$N=\frac{2^{m+1}}{3}\left(1-\left(-\frac{1}{2}\right)^{m+1}\right)$$
we have $$(N L_{p,N}(Y_2))^p \ge
\frac{1}{2^p}\left(\frac{m}{3}+\frac{1}{9}-(-1)^m\frac{1}{9 \cdot
2^m}\right)^p.$$ In particular,
\begin{equation}\label{limsupLp}
\limsup_{N \rightarrow \infty}\frac{N L_{p,N}(Y_2)}{\log
N}=\frac{1}{6 \log 2}=0.2404\ldots.
\end{equation}
\end{theorem}
A proof of these results can be found in \cite{pil04}. Equation
\eqref{limsupLp} for $1 \le p \le 2$ can also be found in
\cite{chafa,proat} and for $p=2$ in \cite{hab1966}.

A comparison of Theorem~\ref{discLpb2} with \eqref{lowproinov}
shows that the $L_p$-discrepancy of the van der Corput sequence in
base 2 is not of optimal order of magnitude in $N$. This problem
can be overcome with a symmetrization technique which will be
explained in Section~\ref{sec_sym}.

On average the $L_p$-discrepancy of the van der Corput sequence is
large but behaves very regularly, as can be seen from the following
theorem, which is \cite[Theorem~3]{DLP05}.

\begin{theorem}[Drmota, Larcher, and Pillichshammer]
For every $p\ge 1$, for every real $y$, and for $M \rightarrow
\infty$ we have
$$
\frac 1M \# \left\{ N < M : N L_{p,N}(Y_2) \le \frac{\log N}{8
\log 2} + \frac{y}{8} \sqrt{\frac{\log N}{3 \log 2}} \right\} =
\Phi(y) + o(1),
$$
that is, the $L_p$-discrepancy satisfies a central limit theorem.
\end{theorem}


\subsection{The discrepancy of the generalized van der Corput sequence}\label{secdiscgenvdC}


\paragraph{Definition of the generalized van der Corput sequence.}

This type of sequence was introduced by Faure~\cite{Fau1981} with
the aim of improving the constant term in the asymptotic behavior
of low-discrepancy sequences.

\begin{definition}\label{vdc2}\rm
Let $b \in \NN$, $b \ge 2$, and let $\Sigma=(\sigma_r)_{r\geq 0}$
be a sequence of permutations of $\{0,1,\ldots,b-1\}=:\ZZ_b$.
\begin{itemize}
\item The {\it $b$-adic radical inverse function with respect to
$\Sigma$} is defined as $Y_b^{\Sigma}: \NN_0 \rightarrow [0,1)$,
$$Y_b^{\Sigma}(n)=\frac{\sigma_0(n_0)}{b}+\frac{\sigma_1(n_1)}{b^2}+\frac{\sigma_2(n_2)}{b^3}+\cdots,$$
for $n \in \NN_0$ with $b$-adic digit expansion $n=n_0+n_1 b+n_2
b^2+\cdots$, where $n_i \in \ZZ_b$. \item The {\it generalized van
der Corput sequence in base $b$ associated with $\Sigma$} is
defined as $Y_b^{\Sigma}:=(y_n)_{n \ge 0}$ with
$y_n=Y_b^{\Sigma}(n)$.
\end{itemize}
If $\Sigma=(\sigma)_{r \ge 0}$ is constant, we simply write
$Y_b^\Sigma=Y_b^\sigma$. The original van der Corput sequence
$Y_b$ in base $b$ is obtained with the identical permutation,
i.e., $Y_b=Y_b^{\id}$.
\end{definition}

It is an easy extension of the classical case of $Y_b$ that for
every $\Sigma$ the generalized van der Corput sequence in base $b$
associated with $\Sigma$ is uniformly distributed modulo 1. See
also \cite[Propri\'et\'e 3.1.1]{Fau1981} for a short proof using a
result on $b$-additive functions.

\paragraph{Functions related to a pair $(b, \sigma)$.}
In \cite{Fau1981} and \cite{chafa} the formulas \eqref{disc_henri}
and \eqref{discfoL2} were fully generalized. To recall these
formulas we first need to introduce some definitions used in the
study of generalized van der Corput sequences. For $J \subseteq
[0,1)$ and $X=(x_n)_{n \ge 0}$ let $A(J;N;X)$ be the number of
indices $n \in \{0,1,\ldots,N-1\}$ such that $x_n \in J$.

Let $\Sy_b$ be the set of all permutations of $\ZZ_b$. For any
$\sigma \in \Sy_b$ set
$$\mathcal{Z}_b^{\sigma}=\left(\tfrac{\sigma(0)}{b},\tfrac{\sigma(1)}{b},\ldots
,\tfrac{\sigma(b-1)}{b}\right).$$ For $h \in \ZZ_b$ and $x \in
[(k-1)/b,k/b)$, where $k \in \{1,\ldots ,b\}$, define
$$\varphi_{b,h}^{\sigma}(x)=\left\{
\begin{array}{ll}
A([0,h/b);k;\mathcal{Z}_b^{\sigma})-h x & \mbox{ if } 0 \le h \le
\sigma(k-1),\\(b-h)x-A([h/b,1);k;\mathcal{Z}_b^{\sigma}) & \mbox{
if }\sigma(k-1)< h < b.
\end{array}\right.
$$
Note that $\varphi_{b,0}^{\sigma}=0$ for any $\sigma \in \Sy_b$.
Further we observe that if $b=2$, then we only have two
permutations which give either $\varphi_{2,1}^{\sigma}=\| \cdot
\|$ if $\sigma$ is the identity or $\varphi_{2,1}^{\sigma}=-\|
\cdot \|$ if $\sigma$ is the transposition $(0\,1)$.

Moreover, the functions $\varphi_{b,h}^{\sigma}$ are extended to
the reals by periodicity. Based on $\varphi_{b,h}^{\sigma}$ we
define further functions which then appear in the formulas for
different notions of discrepancy. Put $$\psi_b^{\sigma,+}=\max_{0
\le h < b}
\varphi_{b,h}^{\sigma}\;,\;\;\;\psi_b^{\sigma,-}=\max_{0 \le h <
b} (-\varphi_{b,h}^{\sigma}), \;\;
\psi_b^{\sigma}=\psi_b^{\sigma,+}+\psi_b^{\sigma,-},$$
$$\varphi_b^\sigma=\sum_{h=0}^{b-1} \varphi_{b,h}^{\sigma} \ ,\;\; \phi_b^\sigma=\sum_{h=0}^{b-1} \left (\varphi_{b,h}^{\sigma} \right )^2 \ \mbox{ and } \chi_b^\sigma=b\phi_b^\sigma-(\varphi_b^\sigma)^2.$$
Note that $\psi_b^{\sigma}=\max_{0 \le h < h' < b}
|\varphi_{b,h}^{\sigma}-\varphi_{b,h'}^{\sigma}|$ \;and
\;$\chi_b^\sigma=\sum_{0 \le h < h' < b}
(\varphi_{b,h}^{\sigma}-\varphi_{b,h'}^{\sigma})^2$.

\paragraph{Exact formulas.}
For a sequence $X=(x_n)_{n \ge 0}$ in $[0,1)$ let
$$D_N^+(X)=\sup_{0 \le t \le 1} \Delta_{X,N}(t)\ \ \ \mbox{ and }\ \ \ D_N^-(X)=\sup_{0 \le t \le 1} (-\Delta_{X,N}(t)).$$

The following theorems link, with exact formulas, the functions
defined above to the discrepancy of generalized van der Corput
sequences. Note that the infinite series in these formulas can in fact 
be represented as the sum of finite terms and geometric series, and hence 
can be computed exactly (see \cite[3.3.6 Corollaire 1]{Fau1981} for $D_N,
D_N^+, D_N^-$; analogous proofs hold for $L_{2,N}$ and $F_N$). We
first give the formulas for the extreme discrepancies obtained in
\cite[Th\'eor\`eme 1]{Fau1981}:

\begin{theorem}[Faure]\label{thmF81D}
For any $N \in \NN$ we have
$$
ND^+_N(Y_b^{\Sigma})=\sum_{j=1}^\infty \psi_b^{\sigma_{j-1},+}
\left ( \frac{N}{b^j} \right ),
$$
$$
ND^-_N(Y_b^{\Sigma})=\sum_{j=1}^\infty
\psi_b^{\sigma_{j-1},-}\left ( \frac{N}{ b^j} \right ),
$$
$$
ND_N(Y_b^{\Sigma})=\sum_{j=1}^\infty \psi_b^{\sigma_{j-1}} \left (
\frac{N}{b^j} \right ) \mbox{ and }
$$
$$
D_N^{\ast}(Y_b^{\Sigma})=\max(D^+_N(Y_b^{\Sigma}),D^-_N(Y_b^{\Sigma})).
$$
\end{theorem}
Next, we state the formulas for $L_{2,N}, F_N$, and $L_{1,N}$ obtained
in \cite[Th\'eor\`emes 4.1, 4.2, 4.3]{chafa}, but in a fixed base
$b$ instead of in variable bases as in \cite{chafa}:

\begin{theorem}[Chaix and Faure]\label{thmCF93}
For any $N \in \NN$, we have
$$
(N L_{2,N}(Y_b^{\Sigma}))^2=\\
\frac{1}{b}\sum_{j=1}^\infty \phi_b^{\sigma_{j-1}} \left
(\frac{N}{b^j} \right )+ \frac{1}{b^2} \sum_{i \neq j}
\varphi_b^{\sigma_{i-1}}\left (\frac{N}{b^i} \right )
\varphi_b^{\sigma_{j-1}}\left (\frac{N}{b^j} \right ),
$$
$$
(NF_N(Y_b^{\Sigma}))^2=\frac{4 \pi^2}{b^2}\sum_{j=1}^\infty
\chi_b^{\sigma_{j-1}} \left ( \frac{N}{b^j} \right ) \ \ \ \mbox{
and }
$$
$$
N L_{1,N}(Y_b)=\frac{1}{b}\sum_{j=1}^\infty \varphi_b^{\id} \left
(\frac{N}{b^j} \right ).
$$
\end{theorem}

\paragraph{Asymptotic behavior for the extreme and star discrepancies.}
Concerning the asymptotic behavior of the extreme discrepancy
$D_N$, the following generalization of \eqref{discbd_tijde} to
arbitrary bases with fixed permutations was  proved in
\cite[Th\'{e}or\`{e}me 2]{Fau1981}:
\begin{theorem}[Faure]\label{asymptD}
For $\sigma \in \Sy_b$ and for all $N \in \NN$ we have
$$D_N(Y_b^{\sigma}) \le \frac{1}{N}\left(\frac{\alpha_{b,\sigma}}{\log b} \log
N +\max\left(2,1+\frac{1}{b}+\alpha_{b,\sigma}\right)\right),$$
where $$\alpha_{b,\sigma}=\inf_{m \ge 1} \sup_{x \in
\RR}\left(\frac{1}{m} \sum_{j=1}^m
\psi_b^{\sigma}\left(\frac{x}{b^j}\right)\right).$$ Moreover,
\begin{equation}\label{Fdef_d}
d(Y_b^{\sigma}):=\limsup_{N \rightarrow \infty} \frac{N
D_N(Y_b^{\sigma})}{\log N} =\frac{\alpha_{b,\sigma}}{\log b}.
\end{equation}\
\end{theorem}
\begin{remark}\label{alphabid}\rm
In particular, see \cite[Th\'{e}or\`{e}me 6]{Fau1981}, for
$\sigma=\id$ we have (the best constant being obtained for $b=3$
with $\alpha_{3,\id}=1/2 \log 3$)
\begin{equation}\label{defalphabid}
\alpha_{b,\id}=\left\{
\begin{array}{ll}\displaystyle
\frac{b^2}{4(b+1)} & \mbox{ if $b$ is even,}\\\displaystyle
\frac{b-1}{4} & \mbox{ if $b$ is odd.}
\end{array}\right.
\end{equation}
It is worth noting that $\psi_b^{\id,-}=0$, so that
$\psi_b^{\id}=\psi_b^{\id,+}$ and hence $D_N(Y_b)=D_N^*(Y_b)$ for
the van der Corput sequence $Y_b$. Equation (\ref{defalphabid})
shows that $D_N(Y_b)$ increases and tends to infinity with $b$.
The question of the existence of a permutation $\sigma_b$ such
that $d(Y_b^{\sigma_b})$ is bounded (by an absolute constant)
can be answered positively, see Section \ref{cantor}.
\end{remark}

Theorem \ref{asymptD} is the starting point for research on
best possible upper bounds for the extreme discrepancy $D_N$, see
the end of the present Section \ref{secdiscgenvdC}.
In the next theorem, we give an analogous result for the star
discrepancy $D_N^*$ in a slightly different form than the original
version in \cite[Th\'{e}or\`{e}me 3]{Fau1981}:
\begin{theorem}[Faure]\label{asymptD*}
Let $\tau_b \in \Sy_b$ be the permutation defined by
$\tau_b(k)=b-k-1$ for all $k \in \ZZ_b$ and let $\mathcal{A}$ be
the subset of $\NN_0$ defined by $\mathcal{A}=\bigcup_{H=1}^\infty
\mathcal{A}_H$ with $\mathcal{A}_H=\{H(H-1),\ldots,H^2-1\}$.
For any permutation $\sigma \in \Sy_b$, let
$\overline{\sigma}:=\tau_b \circ \sigma$ and let
$$\Sigma^\sigma_\mathcal{A}=(\sigma_r)_{r \ge 0} := (\sigma,\overline{\sigma},\sigma,\sigma,\overline{\sigma},
\overline{\sigma},\sigma,\sigma,\sigma,\overline{\sigma},\overline{\sigma},\overline{\sigma},\ldots)$$
be the sequence of permutations defined by $\sigma_r=\sigma$ if $r
\in \mathcal{A}$ and $\sigma_r=\overline{\sigma}$ if $r \notin
\mathcal{A}$. Then
\begin{equation}\label{Fdef_d*}
d^*(Y_b^{\Sigma^\sigma_\mathcal{A}}):=\limsup_{N \rightarrow
\infty} \frac{N D_N^*(Y_b^{\Sigma^\sigma_\mathcal{A}})}{\log
N}=\frac{\alpha_{b,\sigma}^{+}+\alpha_{b,\sigma}^{-}}{2 \log b},
\end{equation}
where
$$\alpha_{b,\sigma}^{+}=\inf_{m\geq 1} \frac{1}{m}\sup_{x\in [0,1]} \sum_{j=1}^m
\psi_b^{\sigma,+}\left (\frac{x}{b^j}\right ) \quad \mbox{ and }
\quad \alpha_{b,\sigma}^{-}=\inf_{m\geq 1} \frac{1}{m} \sup_{x\in
[0,1]} \sum_{j=1}^m \psi_b^{\sigma,-}\left (\frac{x}{b^j}\right ).
$$
\end{theorem}

\begin{remark}[On the ``swapping'' permutation $\tau_b$]\rm
Let $\mathcal{S}$ be a subset of $\NN_0$ and let $\sigma \in
\Sy_b$. Define the sequence
${\Sigma^\sigma_\mathcal{S}}=(\sigma_r)_{r\geq 0}$ by
$\sigma_r=\sigma$ if $r\in \mathcal{S}$ and
$\sigma_r=\overline{\sigma}=\tau_b\circ\sigma$ if $r \notin
\mathcal{S}$. Then, see \cite[Lemme~4.4.1]{Fau1981}, we have
\begin{eqnarray*}
N D^+_N(Y_b^{\Sigma^\sigma_\mathcal{S}}) & = & \sum_{j=1 \atop j\in \mathcal{S}}^\infty  \psi_b^{\sigma,+} 
\left (\frac{N}{b^j} \right )+\sum_{j=1 \atop j\notin \mathcal{S}}^\infty \psi_b^{\sigma,-} \left (\frac{N}{b^j} \right ) \ \mbox{ and }\\
ND^-_N(Y_b^{{\Sigma^\sigma_\mathcal{S}}}) & = & \sum_{j=1 \atop
j\in \mathcal{S}}^\infty \psi_b^{\sigma,-} \left ( \frac{N}{b^j}
\right )+\sum_{j=1 \atop j\notin \mathcal{S}}^\infty
\psi_b^{\sigma,+} \left ( \frac{N}{b^j} \right ).
\end{eqnarray*}
The permutation $\tau_b$ {\it swaps} the functions
$\psi_b^{\sigma,+}$ and $\psi_b^{\sigma,-}$ and hence, in order to
minimize $D_N^*=\max(D_N^+,D_N^-)$, one has to find a set
$\mathcal{S}$ for which the sums with $\psi_b^{\sigma,+}$ and
$\psi_b^{\sigma,-}$ asymptotically split into two equal parts.
This is achieved, for example, by the set
$\mathcal{A}=\{0,2,3,6,7,8,12,13,14,15,\ldots\}$ introduced in
Theorem \ref{asymptD*} above. Further, since
$\psi_b^{\sigma,+}+\psi_b^{\sigma,-}=\psi_b^\sigma$, for any
$\mathcal{S} \subseteq \NN_0$, we have
$$
D_N(Y_b^{\Sigma^\sigma_\mathcal{S}})=D^+_N(Y_b^{\Sigma^\sigma_\mathcal{S}})+D^-_N(Y_b^{{\Sigma^\sigma_\mathcal{S}}})=\sum_{j=1}^\infty
\left(\psi_b^{\sigma,+} \left (\frac{N}{b^j} \right )+
\psi_b^{\sigma,-} \left (\frac{N}{b^j} \right
)\right)=D_N(Y_b^\sigma).
$$
This property is the starting point for research on best
possible lower bounds for the star discrepancy for the family of
$Y_b^{\Sigma}$ sequences, and for a new generalization of this
family, see Section \ref{secdig01seq}.
\end{remark}

\begin{remark}\label{re5bejKrLaPi}\rm
For reasonably small $b$, the constants $\alpha_{b,\sigma}^{+}$
and $\alpha_{b,\sigma}^{-}$ are not difficult to compute, and for
the identical permutation (for which $\psi_b^{\id,-}=0$) it is
even possible to find them explicitly: from (\ref{defalphabid}),
for arbitrary $b$ we obtain
\begin{equation}\label{d*alter}
d^*(Y_b^{\Sigma^\id_\mathcal{A}})=\frac{\alpha_{b, \id}^{+}}{2
\log b}=\left\{
\begin{array}{ll}\displaystyle
\frac{b^2}{8 (b+1)\log b} & \mbox{ if } b \mbox{ is
even,}\\\displaystyle \frac{b-1}{8 \log b} & \mbox{ if } b \mbox{
is odd.}
\end{array}\right.
\end{equation}
For $b=2$, this result was first obtained by B\'ejian
\cite{be1978} and then rediscovered in \cite{KLP07} in the context
of shifted Hammersley and van der Corput point sets. The result in
\cite{KLP07} states that
\begin{equation}\label{stdiscestasm}
ND_N^\ast (Y_2^{\Sigma^\id_\mathcal{A}})\le \frac{\log N}{6 \log
2}+c\sqrt{\log N},
\end{equation}
for all $N\in\NN$, where $c>0$ is a constant and where the
constant $1/(6 \log 2)$ is best possible. We thus see that there
exists an explicit generalized version of $Y_2$ which reduces the
constant in \eqref{discbd_tijde} by a factor of 2, and this
reduction is best possible.
\end{remark}

It is also worth mentioning the following result from \cite{KLP07}
for general sequences $\Sigma$ and $b=2$ which reflects the
properties of $\Sigma$ in order to obtain a low star discrepancy.
For $m\in\NN$ let
$$S_m(\Sigma):=\max\left(\#\{0\le r\le m-1:\  \sigma_r=\tau_2\},\#\{0\le r\le m-1:\  \sigma_r=\mathrm{id}\}\right) $$
and
$$T_m(\Sigma):=\#\{1\le r\le m-1:\ \sigma_{r-1}=\tau_2\ \mbox{and}\ \sigma_r=\id\}.$$
The quantities $S_m (\Sigma)$ and $T_m (\Sigma)$ can be used to
give very precise discrepancy estimates for a generalized van der
Corput sequence based on $\Sigma$. The following theorem was shown
in \cite{KLP07}.
\begin{theorem}[Kritzer, Larcher, and Pillichshammer]\label{thmKLPstar}
Let $\Sigma= (\sigma_r)_{r\ge 0}$ be a sequence of permutations of
$\ZZ_2$, and let $Y_2^\Sigma$ be the corresponding generalized van
der Corput sequence. Then it is true that
$$\frac{S_m (\Sigma)}{3}+\frac{T_m (\Sigma)}{48}-4\le \max_{1\le N\le 2^m} N D_N^\ast (Y_2^\Sigma)\le
\frac{S_m (\Sigma)}{3}+\frac{2T_m (\Sigma)}{9}+\frac{56}{9}.$$
\end{theorem}

Theorem \ref{thmKLPstar} implies that two properties of $\Sigma$
influence the star discrepancy of the shifted van der Corput
sequence. On the one hand, it is the number of identical
permutations $\id$ compared to the number of transpositions
$\tau_2$ in $\Sigma$, and on the other hand, it is also how these
are distributed. Obviously, it is favorable to have a situation
where, for as many $m\in \NN$ as possible, approximately half of
the permutations among the first $m$ elements of $\Sigma$ is the
identity, and simultaneously there should not be too many changes
between $\id$ and $\tau_2$. This observation is perfectly
reflected in the star discrepancy estimate for
$Y_2^{\Sigma^\id_\mathcal{A}}$ presented above in
\eqref{stdiscestasm}.

\paragraph{Asymptotic behavior of the $L_2$-discrepancy and the diaphony.}
As for Theorem \ref{asymptD}, Theorem \ref{asymptD*} is the
starting point for research on best possible upper bounds for
the star discrepancy (see the end of Section \ref{secdiscgenvdC}).
These theorems have analogous formulations, resulting from Theorem
\ref{thmCF93}, for the $L_p$-discrepancy and for the diaphony
$F_N$. We only give an excerpt of a series of results in fixed and
variable bases that can be found in \cite[Section 4.2]{chafa}:
\begin{theorem}[Chaix and Faure]\label{asymptL2_F}
For $\sigma \in \Sy_b$ and for all $N \in \NN$ we have
\begin{equation*}\label{Fdef_t}
t_1(Y_b^{\sigma}):=\limsup_{N \rightarrow \infty} \frac{N
L_{2,N}(Y_b^{\sigma})}{\log N} =\frac{\beta_{b,\sigma}}{b\log b}
\ \mbox{ with } \ \beta_{b,\sigma}=\inf_{m \ge 1} \sup_{x
\in\RR}\left(\frac{1}{m} \sum_{j=1}^m
\varphi_b^{\sigma}\left(\frac{x}{b^j}\right)\right),
\end{equation*}
and
\begin{equation*}\label{Fdef_f}
f(Y_b^{\sigma}):=\limsup_{N \rightarrow \infty} \frac{(N
F_N(Y_b^{\sigma}))^2}{\log N} =\frac{4 \pi^2
\gamma_{b,\sigma}}{b^2\log b} \ \mbox{ with } \ 
\gamma_{b,\sigma}=\inf_{m \ge 1} \sup_{x \in\RR}\left(\frac{1}{m}
\sum_{j=1}^m \chi_b^{\sigma}\left(\frac{x}{b^j}\right)\right).
\end{equation*}
\end{theorem}

Theorem~\ref{asymptL2_F} confirms that the generalized van der
Corput sequences $Y_b^{\sigma}$ do not have optimal
$L_2$-discrepancy with respect to the order of magnitude in $N$, which would be
$O(\sqrt{\log N}/N)$.

\begin{remark}\rm
In particular, see \cite[Th\'{e}or\`{e}mes 4.12 and 4.13]{chafa},
for $\sigma=\id$ we have:
\begin{equation*}\label{betaid}
t_1(Y_b^\id)=\left\{
\begin{array}{ll}\displaystyle
\frac{b^2+b-2}{8(b+1)\log b} & \mbox{ if $b$ is
even,}\\\displaystyle \frac{b^2-1}{8b \log b} & \mbox{ if $b$ is
odd,}
\end{array}\right.
\end{equation*}
(cf. Theorem~\ref{thmFL2} for $b=2$) and
\begin{equation*}\label{gammaid}
f(Y_b^\id)=\left\{
\begin{array}{ll}\displaystyle
\frac{\pi^2(b^3+b^2+4)}{48(b+1)\log b} & \mbox{ if $b$ is
even,}\\\displaystyle \frac{\pi^2(b^4+2b^2-3)}{48 b^2 \log b} &
\mbox{ if $b$ is odd.}
\end{array}\right.
\end{equation*}
These exact values for $f(Y_b^\id)$ are about four times smaller than
the upper bound obtained by Pag\`es with the von Neumann-Kakutani
transformation, see \cite[Proposition 2.5]{page}.
\end{remark}

\medskip

We now summarize the results from \cite{KP05}, where the
$L_2$-discrepancy of $Y_2^\Sigma$ is considered. From Theorem
\ref{thmFL2}, we know that
\begin{equation}\label{eqonesixth}
N L_{2,N} (Y_2) \sim \frac{\log N}{6 \log 2}.
\end{equation}
Again, the question arises how the $L_{2}$-discrepancy of
$Y_2^\Sigma$ is influenced by the properties of $\Sigma$. Opposed
to the star discrepancy, here one may ask not only whether the
constant in the $\log N$-term can be improved, but, in view of
\eqref{lowproinov}, also whether one could possibly obtain a
better order of magnitude. To give precise answers, we proceed in
a way that resembles the results for the star discrepancy. Let
$\Sigma= (\sigma_r)_{r\ge 0}$ be a sequence of permutations of
$\ZZ_2$, and define, for $m\in\NN$,
$$Z_m (\Sigma):=\#\{0\le r\le m-1:\ \sigma_r=\mathrm{id}\}.$$
Furthermore, let $z_m (\Sigma):=Z_m (\Sigma)-\frac{m}{2}$. Using
the quantity $z_m$ the following average-type result was shown in
\cite{KP05}.
\begin{theorem}[Kritzer and Pillichshammer]\label{thmKPL2}
 Let $\Sigma= (\sigma_r)_{r\ge 0}$ be a
sequence of permutations of $\ZZ_2$ for which
$A:=\lim_{m\To\infty} \frac{(z_m(\Sigma))^2}{m}$ exists. Then the
following two statements are true:
\begin{itemize}
 \item $$\limsup_{m\To\infty} \frac{1}{2^m}\sum_{N=1}^{2^m}\frac{(N L_{2,N}(Y_2^\Sigma))^2}{m}< \frac{A}{16}+\frac{4}{3}.$$
 \item For any $\varepsilon>0$ we have
$$\lim_{m\To\infty}\frac{\#\{1\le N\le 2^m:\ NL_{2,N}(Y_2^\Sigma)< (\log N)^{1/2+\varepsilon}\}}{2^m} =1.$$
\end{itemize}
\end{theorem}

Considering the result in Theorem \ref{thmKPL2}, one might hope
that there is a sequence $\Sigma$ such that the order of
magnitude in $N$ of $L_{2,N}(Y_2^\Sigma)$ is lower than that of
$L_{2,N}(Y_2)$. Unfortunately, these hopes are crushed by the
following result that was also shown in \cite{KP05}.
\begin{theorem}[Kritzer and Pillichshammer]
 There exists a positive constant $c$ such that for all sequences $\Sigma= (\sigma_r)_{r\ge 0}$
of permutations of $\ZZ_2$ we have
$$N L_{2,N} (Y_2^\Sigma)\ge c\log N \ \ \ \mbox{for infinitely many $N\in\NN$.} $$
\end{theorem}

As we see from this theorem, we cannot improve the order of the
$L_2$-discrepancy of $Y_2$ by permutations, but we can again ask
whether one can at least improve the constants in
\eqref{eqonesixth}. For this question, there is a positive answer
which implies that permutations never worsen the constant in
\eqref{eqonesixth} and can even improve the constant considerably.
The following result is also due to \cite{KP05}.
\begin{theorem}[Kritzer and Pillichshammer]
It is true that
$$\sup_{\Sigma} \limsup_{N\To\infty}\frac{NL_{2,N} (Y_2^\Sigma)}{\log N}=\frac{1}{6 \log 2},$$
$$\inf_{\Sigma} \limsup_{N\To\infty}\frac{NL_{2,N} (Y_2^\Sigma)}{\log N} \le \frac{1}{20 \log 2},$$
where the supremum and the infimum, respectively, are extended
over all sequences $\Sigma$ of permutations of $\ZZ_2$.
\end{theorem}

Regarding the latter result, one can show that the value $\tfrac{1}{20 \log 2}$ is
attained for the sequence $\Sigma$ in which the identical
permutation $\mathrm{id}$ and the transposition $\tau_2$
alternate. It is even conjectured that this result is best
possible with respect to all choices of $\Sigma$. Note that, if
the conjecture is correct, this would mark a difference to the
star discrepancy, where an alternating sequence $\Sigma$ has
disadvantages due to the role of the quantity $T_m (\Sigma)$ in
Theorem \ref{thmKLPstar}. For the $L_2$-discrepancy it seems to be
the case that, while $S_m$ also plays a crucial role, the value of
$T_m$ has no influence on the results.

\paragraph{The state of the art for asymptotic constants related to $D_N, D_N^*, L_{2,N}$, and $F_N$.}
Generalizations of van der Corput sequences give the smallest
discrepancies and diaphony currently known within the class of
infinite one-dimensional sequences $X$ in $[0,1)$. We need the
following quality parameters to present these current ``records''.
Let
$$
d:=\inf_X(d(X))=\inf_X\left (\limsup_{N \rightarrow \infty}
\frac{N D_N(X)}{\log N}\right), \ \ \
d^*:=\inf_X(d^*(X))=\inf_X\left (\limsup_{N \rightarrow \infty}
\frac{N D^*_N(X)}{\log N}\right),
$$
$$
t_1:=\inf_X(t_1(X))=\inf_X\left (\limsup_{N \rightarrow \infty}
\frac{N L_{2,N}(X)}{\log N}\right), \ \ \ t_2(X):=\limsup_{N
\rightarrow \infty} \frac{(N L_{2,N}(X))^2}{\log N},
$$
$$
t_2:=\inf_X(t_2(X))\ \ \mbox{ and } \ \
f:=\inf_X(f(X))=\inf_X\left (\limsup_{N \rightarrow \infty}
\frac{(N F_N(X))^2}{\log N}\right),
$$
where the infima are taken over all infinite one-dimensional
sequences $X$ in $[0,1)$.

\medskip

For the extreme discrepancy $D_N$, it was shown, in chronological order, that $d<0.38$,
in base 12 by Faure \cite[Th\'{e}or\`{e}me 4]{Fau1981} in 1981,
$d<0.367$, in base 36 by Faure \cite[Th\'{e}or\`{e}me 1.2]{Fau92}
in 1992, and $d<0.354$ in base 84 by Ostromoukhov \cite[Theorem
4]{Ost09} in 2009, in each case for a specific permutation of
digits in the specified base. Hence with the lower bound from 
B\'ejian \cite{be1982} in 1982, one has the current estimate
$$0.12 < d < 0.354.$$

For the star discrepancy $D_N^*$, it was shown, in chronological order, that
$d^*<0.2236$, in base 12 by Faure \cite[Th\'{e}or\`{e}me
4]{Fau1981} in 1981, and $d^*<0.2223$  in base 60 by Ostromoukhov
\cite[Theorem 5]{Ost09} in 2009, in each case for a specific
permutation of digits in the specified base. Hence with the lower
bound from Larcher \cite{lar2014} in 2014 one has the current
estimate $$0.0646 < d ^*< 0.2223.$$

Concerning the $L_2$-discrepancy, the best permutations for the
set of $Y_b^\sigma$ sequences, denoted by $\sigma_b$, satisfy
$\lim_{b\rightarrow \infty} t_1(Y_b^{\sigma_b})=0 $, see
\cite[Th\'eor\`eme 4.14 and its proof]{chafa}. However, this result is
not surprising since $t_1=0$ as can be seen from the study of
symmetrized sequences (see Section \ref{sec_sym}) where estimates
for $t_2$ will be given.

Finally, for the diaphony, besides earlier results by Proinov
and Grozdanov \cite{progro88} with weak constants (around
15), it was shown, in chronological order, that $f<1.59$, in base 2 by Proinov and
Atanassov \cite{proat} in 1988, $f<1.316$, in base 19 by Chaix and
Faure \cite[Th\'eor\`eme 4.16]{chafa} in 1993 and $f<1.137$, in
base 57 by Pausinger and Schmid \cite[Theorem
3.2]{pausch10} in 2010. As to the lower bound, the only result  we
know is that of Proinov  \cite{pro86} from the year 1986 which
states that $f>0.0002$.

\paragraph{Comparison with $(n \alpha)$-sequences.} 
Concerning the extreme discrepancy of $(n\alpha)$-sequences, 
Ramshaw~\cite{Ram} obtained 
$$d((n \widetilde{\alpha}))=\frac{1}{5 \log((\sqrt{5}+1)/2)}=0.4156\ldots$$ 
with the golden ration $\widetilde{\alpha}=\tfrac{\sqrt{5}+1}{2}$ and 
it is a widely held belief that $\inf_{\alpha} d((n\alpha)) = d((n\widetilde{\alpha}))$.
Concerning the star discrepancy of $(n\alpha)$-sequences,
Dupain and S\'{o}s~\cite{DS84} proved that \label{dupSos}
\begin{equation*}\label{best_nalpha}
\inf_{\alpha} d^*((n\alpha)) = d^*((n \sqrt{2})) = \frac{1}{4
\log(\sqrt{2}+1)}=0.2836\ldots.
\end{equation*}


\subsection{The elementary descent method}\label{sec_Faure_meth}
\paragraph{Context, statement and comments.}
In this section, we are going to give an illustration---on the
simple example of the triadic van der Corput  sequence $Y_3$---of
the method yielding the results given in Sections
\ref{secdiscgenvdC}, \ref{bdremaindersets},  \ref{sec_sym}, and
\ref{cantor} for the generalized van der Corput sequences
$Y_b^\Sigma$. A more involved variant of this method also applies
to NUT digital $(0,1)$-sequences in prime base and to a larger
family, the so-called NUT $(0,1)$-sequences over $\ZZ_b$, that
will be considered in Section \ref{secdig01seq}. Before this illustration, it is necessary to state the fundamental
lemma resulting from the descent method and leading to the
results. However, we first need some notation and a discretization
lemma to prepare this statement.

For convenience, in the following we define the error (the deviation from uniform 
distribution) as
$$E(x, N, X):=N\Delta_{X,N}(x) \mbox{ and then } E([x,y), N, X):=E(y, N, X)-E(x, N, X),$$
where $\Delta_{X,N}(x)$ is the local discrepancy defined in
Section \ref{sec_udt}. Further, we call a rational number in
$[0,1)$ an {\it $m$-bit number}\label{nbit} if it belongs to the
set $\QQ(b^m)=\{0,\tfrac{1}{b^m},\tfrac{2}{b^m},\ldots ,
\tfrac{b^m -1}{b^m}\}$.

\begin{lemma}[discretization lemma]\label{DisLem}
Let $\alpha \in [0,1]$ and let $m, N$ be integers with $1\leq
N\leq b^m$. Furthermore, let $u,v$ be $m$-bit numbers such that $u
\le \alpha <v$ and $v=u+b^{-m}$, and let $x$ be the unique element
in $\{Y_b^\Sigma(0),\ldots, Y_b^\Sigma(b^m-1)\}$ such that $u \le
x < v$. Then, with $y(\alpha):= u$ if $\alpha \le x$ and
$y(\alpha):= v$ if $\alpha > x$, we have
$$E(\alpha,N,Y_b^\Sigma)=E(y(\alpha),N,Y_b^\Sigma)+(y(\alpha)-\alpha)N.$$
\end{lemma}

This is a classical discretization property resulting from the
definition of $y(\alpha)$ which implies the equality of counting
functions $A(\alpha;N;Y_b^\Sigma)=A(y(\alpha);N;Y_b^\Sigma)$. From
Lemma~\ref{DisLem} it follows that it suffices to know the values
of the error $E(\alpha,N,Y_b^{\Sigma})$ for $\alpha \in \QQ(b^m)$.
Exactly these values are given by the following lemma which is the
key property in the proof of Theorem~\ref{thmF81D}.

\begin{lemma}[fundamental descent lemma]\label{DesLem}
Let $m,N$, and $\lambda$ be integers with $1\leq N\leq b^m$ and
$1\leq \lambda < b^m$, and let $\lambda=\lambda_1
b^{m-1}+\cdots+\lambda_{m-1}b+\lambda_m$ be the b-adic expansion
of $\lambda$. Then
\begin{equation}\label{DesLemForm}
E \left (\frac{\lambda}{ b^m},N,Y_b^\Sigma\right )
=\sum_{j=1}^m\varphi_{b,\varepsilon_j}^{\sigma_{j-1}}\left
(\frac{N}{ b^j} \right ),
\end{equation}
where the functions $\varepsilon_j=\varepsilon_j(\lambda,m,N)$ are
defined inductively by $\varepsilon_m=\eta_m=\lambda_m$ and, for
$1\leq j< m$,
$$
\eta_j=\lambda_j+\frac{\eta_{j+1}}{b}+\frac{1}{b} \left
(\varphi_{b,\varepsilon_{j+1}}^{\sigma_j} \right )' \left
(\frac{N}{b^{j+1}} \right) \ \  \mbox{ where $g'$ means the right
derivative of $g$,}
$$
$$
\varepsilon_j=\eta_j \quad {\rm if} \quad 0\leq \eta_j<b, \quad
{\rm and} \quad \varepsilon_j=0  \quad {\rm if} \quad \eta_j=b.
$$
\end{lemma}

The proof is a descending recursion for the $b$-adic resolution of
the argument $y=\lambda/b^m$ from $m$ to 1. At each step, the
$b$-adic resolution of $y$ is decreased by 1. The differences
between the discrepancy functions in a reduction step is controlled 
by means of the functions $\varphi_{b,h}^\sigma$
while the relation between the intervals depends on the right
derivatives of these functions.

With Lemmas \ref{DisLem} and \ref{DesLem}, it is easy to obtain
Theorem \ref{thmF81D}. For instance, consider $D_N^+$: First, from
Lemma \ref{DisLem}, we get
$$
\lim_{m\rightarrow\infty}\ \sup_{y \in \QQ(b^m)}
E(y,N,Y_b^\Sigma)=D_N^+(Y_b^\Sigma).
$$
Then, since $\psi_b^{\sigma,+}=\max_{0 \le h < b}
\varphi_{b,h}^{\sigma}$, for any $m$-bit number $y \in \QQ(b^m)$
we get from Lemma \ref{DesLem} the upper bound
$$E(y,N,Y_b^\Sigma)\leq \sum_{j=1}^m \psi_b^{\sigma_{j-1},+}\left (\frac{N}{b^j}\right ).$$
Now, using the algorithm for the construction of the
$\varepsilon_j$'s from the $\lambda_j$'s (at the end of Lemma
\ref{DesLem}) in reversed direction, it is easy to construct a
$\widetilde{\lambda} \in \{0,1,\ldots,b^m-1\}$ for which the above
upper bound is achieved by $E(\widetilde{\lambda}
b^{-m},N,Y_b^\Sigma)$. Finally, letting $m$ tend to $\infty$
gives the formula for $D_N^+$.

Theorem \ref{thmCF93} is not so easy and needs a further lemma
(\cite[Lemme 5.3]{chafa}) to explicitly compute the functions
$\varepsilon_j$; this means that recovering the functions
$\varphi_b^\sigma$ and $\phi_b^\sigma$ is more involved in this
case.

There are four versions of Lemma \ref{DesLem}: The first one was 
worked out for the study of the $L_2$-discrepancy of the
original van der Corput sequence in base 2 and its symmetrized
version \cite[Lemme 2.2]{fau1990}. The second one is the
generalization of the preceding one to arbitrary variable bases
with arbitrary sequences of permutations (the so-called
$Y_B^\Sigma$ sequences) in \cite[Lemme 5.2]{chafa} for the study
of both the $L_p$-discrepancy and the diaphony. The third and
fourth ones were necessary to extend the method to digital
$(0,1)$-sequences \cite[Lemma 6.2]{Fau05} and a generalization of
these sequences that also includes $Y_b^\Sigma$ sequences
\cite[Lemma 2]{FauPi13}. In Theorems \ref{thmbfstar} and
\ref{thmF81D} the descent method is only present in filigree
since, at that time, the relation between the intervals at each
step of the descent was not found yet, and it was not necessary
to get formulas for the extreme discrepancies. For non-French
speaking readers who would like to go thoroughly into the details
of the descent method we recommend \cite[Lemma 6.2]{Fau05} where
$X_b^C$ has to be replaced by $Y_b^\Sigma$ and the formula there
by Equation \eqref{DesLemForm}.

\paragraph{Illustration of the descent method.} The first 
$b^m$ points of $Y_b^\Sigma$, $y_0, \ldots, y_{b^m-1}$,
 can be represented on a checkerboard $B_m$ of $b^{2m}$ 
squares as black squares according to their position  $b^m y_k$ on the
abscissa (where $[0,1)$ is identified with $[0, b^m)$) and their index 
$k$ on the ordinate for $1 \le k \le b^m$. This is possible because Equation \eqref{DesLemForm} 
deals with $m$-bit numbers $\lambda / b^m$, the error $E(\lambda / b^m,N,Y_b^\Sigma)$ 
involves the $m$-bit part of $y_k=Y_b^\Sigma(k)$ only. Then, in this framework, 
the error is $b^m$ times the number $A$ of black squares minus the area of 
$[0,\lambda] \times [0,N]$ divided by $b^m$, i.e.,  
$$\widetilde{E}(\lambda, N, B_m):=b^m E(\lambda / b^m,N,Y_b^\Sigma)=b^m A-\lambda N.$$

One step of the descent method can be described, and the idea of
$\varphi$-functions pointed out, by means of this representation
as explained in the following special case of $Y_3$, for the first
step (from $m$ to $m-1$) and with $m=3$, i.e., $b^m=3^3=27$, see
Figure~\ref{fboard}.

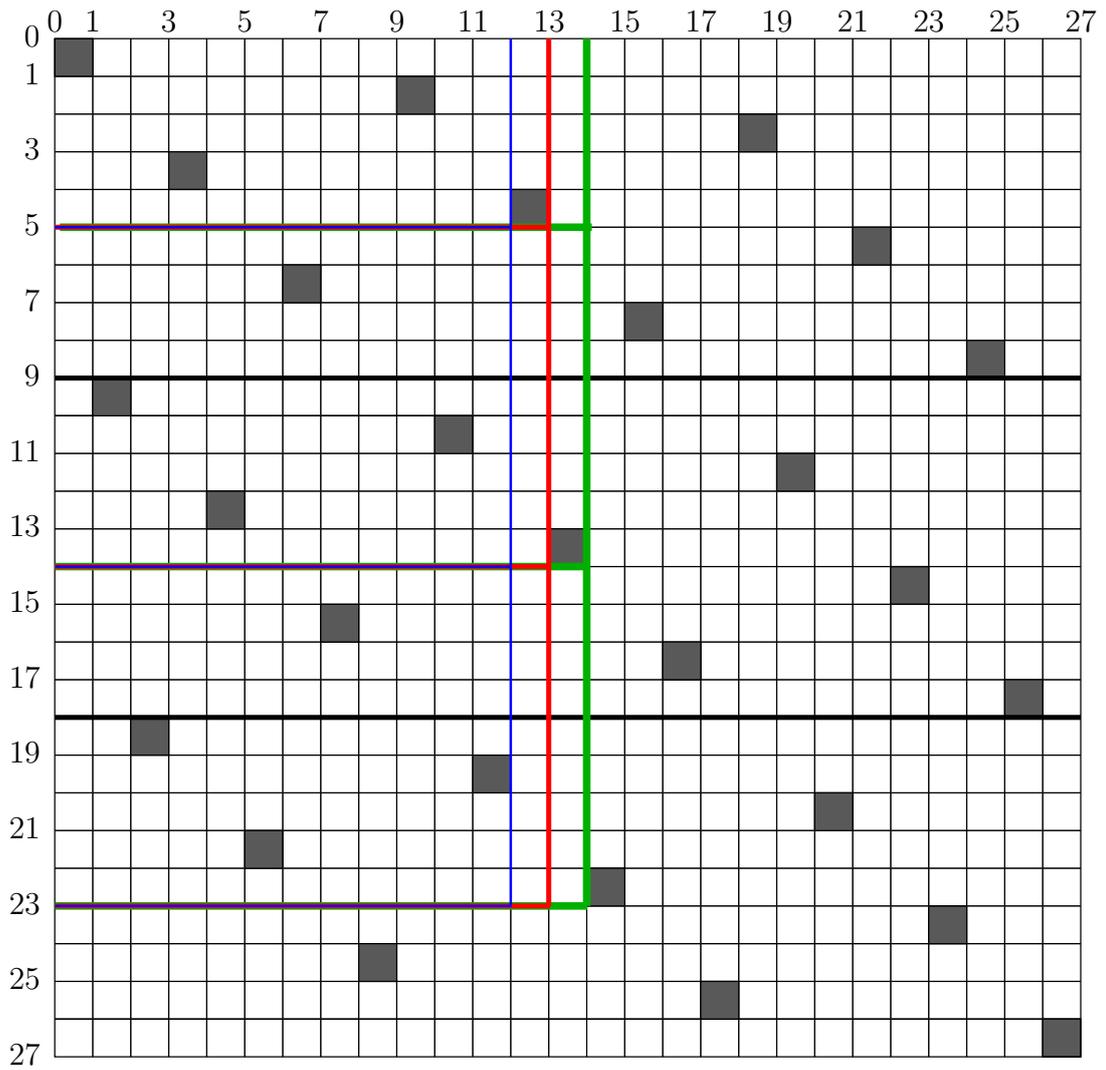
\begin{figure}[htp]
\begin{center}
 \begin{picture}(0,0)%
\includegraphics{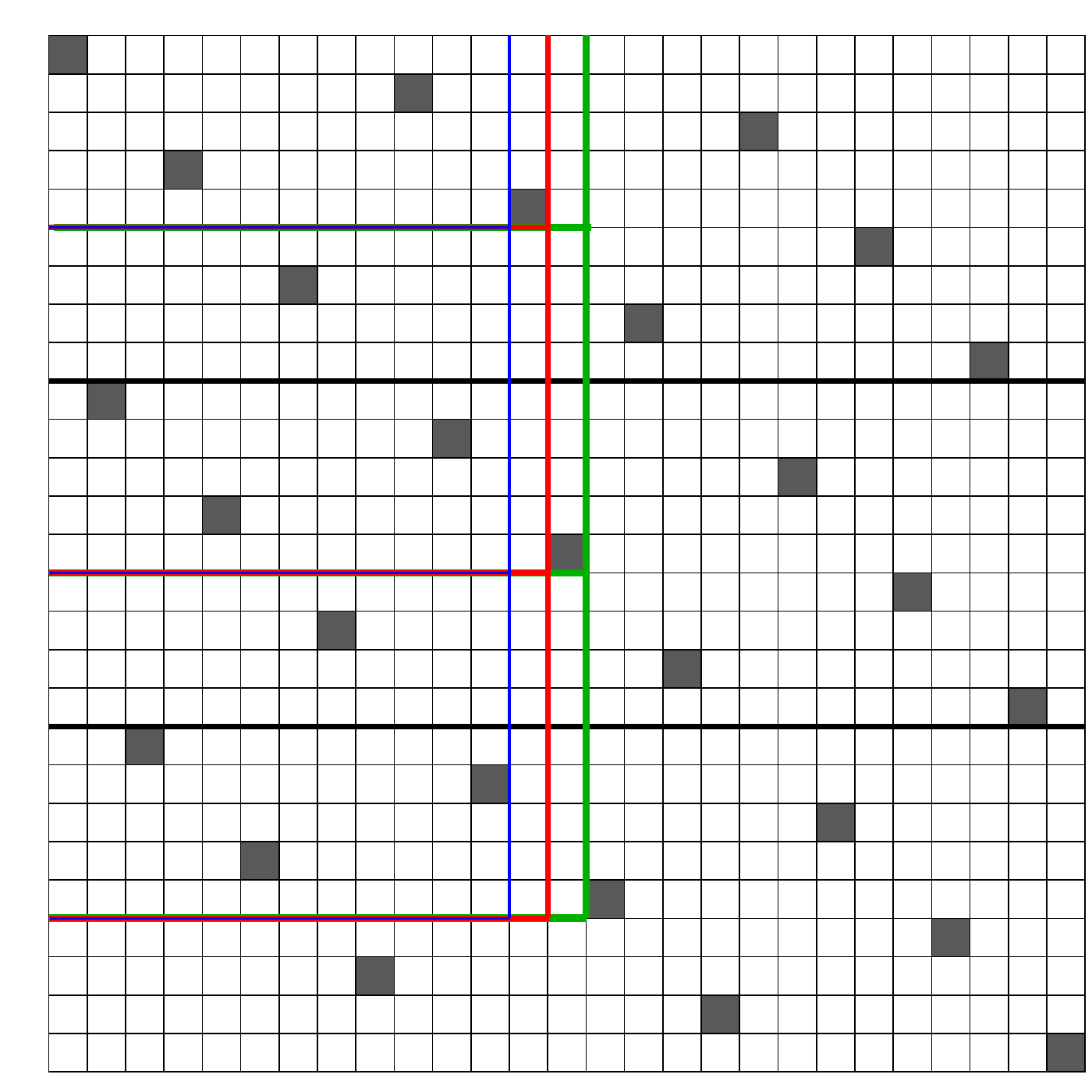}%
\end{picture}%
\setlength{\unitlength}{4144sp}%
\begingroup\makeatletter\ifx\SetFigFont\undefined%
\gdef\SetFigFont#1#2#3#4#5{%
  \reset@font\fontsize{#1}{#2pt}%
  \fontfamily{#3}\fontseries{#4}\fontshape{#5}%
  \selectfont}%
\fi\endgroup%
\begin{picture}(6393,6393)(-284,-5350)
\put(2386,884){\makebox(0,0)[lb]{\smash{{\SetFigFont{12}{14.4}{\rmdefault}{\mddefault}{\updefault}{$11$}%
}}}}
\put(2836,884){\makebox(0,0)[lb]{\smash{{\SetFigFont{12}{14.4}{\rmdefault}{\mddefault}{\updefault}{$13$}%
}}}}
\put(3286,884){\makebox(0,0)[lb]{\smash{{\SetFigFont{12}{14.4}{\rmdefault}{\mddefault}{\updefault}{$15$}%
}}}}
\put(3736,884){\makebox(0,0)[lb]{\smash{{\SetFigFont{12}{14.4}{\rmdefault}{\mddefault}{\updefault}{$17$}%
}}}}
\put(4186,884){\makebox(0,0)[lb]{\smash{{\SetFigFont{12}{14.4}{\rmdefault}{\mddefault}{\updefault}{$19$}%
}}}}
\put(4636,884){\makebox(0,0)[lb]{\smash{{\SetFigFont{12}{14.4}{\rmdefault}{\mddefault}{\updefault}{$21$}%
}}}}
\put(5086,884){\makebox(0,0)[lb]{\smash{{\SetFigFont{12}{14.4}{\rmdefault}{\mddefault}{\updefault}{$23$}%
}}}}
\put(5536,884){\makebox(0,0)[lb]{\smash{{\SetFigFont{12}{14.4}{\rmdefault}{\mddefault}{\updefault}{$25$}%
}}}}
\put(5986,884){\makebox(0,0)[lb]{\smash{{\SetFigFont{12}{14.4}{\rmdefault}{\mddefault}{\updefault}{$27$}%
}}}}
\put(-269,-1681){\makebox(0,0)[lb]{\smash{{\SetFigFont{12}{14.4}{\rmdefault}{\mddefault}{\updefault}{$11$}%
}}}}
\put(-269,-2131){\makebox(0,0)[lb]{\smash{{\SetFigFont{12}{14.4}{\rmdefault}{\mddefault}{\updefault}{$13$}%
}}}}
\put(-269,-2581){\makebox(0,0)[lb]{\smash{{\SetFigFont{12}{14.4}{\rmdefault}{\mddefault}{\updefault}{$15$}%
}}}}
\put(-269,-3031){\makebox(0,0)[lb]{\smash{{\SetFigFont{12}{14.4}{\rmdefault}{\mddefault}{\updefault}{$17$}%
}}}}
\put(-269,-3481){\makebox(0,0)[lb]{\smash{{\SetFigFont{12}{14.4}{\rmdefault}{\mddefault}{\updefault}{$19$}%
}}}}
\put(-269,-3931){\makebox(0,0)[lb]{\smash{{\SetFigFont{12}{14.4}{\rmdefault}{\mddefault}{\updefault}{$21$}%
}}}}
\put(-269,-4381){\makebox(0,0)[lb]{\smash{{\SetFigFont{12}{14.4}{\rmdefault}{\mddefault}{\updefault}{$23$}%
}}}}
\put(-269,-4831){\makebox(0,0)[lb]{\smash{{\SetFigFont{12}{14.4}{\rmdefault}{\mddefault}{\updefault}{$25$}%
}}}}
\put(-269,-5281){\makebox(0,0)[lb]{\smash{{\SetFigFont{12}{14.4}{\rmdefault}{\mddefault}{\updefault}{$27$}%
}}}}
\put(181,884){\makebox(0,0)[lb]{\smash{{\SetFigFont{12}{14.4}{\rmdefault}{\mddefault}{\updefault}{$1$}%
}}}}
\put(631,884){\makebox(0,0)[lb]{\smash{{\SetFigFont{12}{14.4}{\rmdefault}{\mddefault}{\updefault}{$3$}%
}}}}
\put(1081,884){\makebox(0,0)[lb]{\smash{{\SetFigFont{12}{14.4}{\rmdefault}{\mddefault}{\updefault}{$5$}%
}}}}
\put(1531,884){\makebox(0,0)[lb]{\smash{{\SetFigFont{12}{14.4}{\rmdefault}{\mddefault}{\updefault}{$7$}%
}}}}
\put(1981,884){\makebox(0,0)[lb]{\smash{{\SetFigFont{12}{14.4}{\rmdefault}{\mddefault}{\updefault}{$9$}%
}}}}
\put(-44,884){\makebox(0,0)[lb]{\smash{{\SetFigFont{12}{14.4}{\rmdefault}{\mddefault}{\updefault}{$0$}%
}}}}
\put(-179,794){\makebox(0,0)[lb]{\smash{{\SetFigFont{12}{14.4}{\rmdefault}{\mddefault}{\updefault}{$0$}%
}}}}
\put(-179,569){\makebox(0,0)[lb]{\smash{{\SetFigFont{12}{14.4}{\rmdefault}{\mddefault}{\updefault}{$1$}%
}}}}
\put(-179,119){\makebox(0,0)[lb]{\smash{{\SetFigFont{12}{14.4}{\rmdefault}{\mddefault}{\updefault}{$3$}%
}}}}
\put(-179,-331){\makebox(0,0)[lb]{\smash{{\SetFigFont{12}{14.4}{\rmdefault}{\mddefault}{\updefault}{$5$}%
}}}}
\put(-179,-781){\makebox(0,0)[lb]{\smash{{\SetFigFont{12}{14.4}{\rmdefault}{\mddefault}{\updefault}{$7$}%
}}}}
\put(-179,-1231){\makebox(0,0)[lb]{\smash{{\SetFigFont{12}{14.4}{\rmdefault}{\mddefault}{\updefault}{$9$}%
}}}}
\end{picture}%
\caption{Checkerboard $B_3$ corresponding to the first $3^3$ elements of $Y_3$.}
     \label{fboard}
\end{center}
\end{figure}

We distinguish three cases corresponding to $9(i-1) < N \le  9i$,
where $i \in \{1,2,3\}$, and for each of them we distinguish the
abscissas $\lambda$ of the form $3\mu$ (blue lines),
$3\mu+1$ (red lines), and $3\mu+2$ (green lines).
Concerning the three cases for $N$, the idea is to express the
error by moving the abscissa $\lambda$ to the nearest multiple of
3, $3\mu$ or $3(\mu+1)$, say $3\nu$, in such a way that the black
squares $(27y_k,k)$ with $9(i-1) < k \le9i$ still remain in the
new rectangle $[0,3\nu]\times [0,N]$ in the three cases. In this way,
the process is the same for any $N$ within a given case, and it
will give a single formula for the error. Further,
$\widetilde{E}(3\nu, N, B_3)=\widetilde{E}(3\nu,N-9(i-1), B_3)$
since $\widetilde{E}(3\nu,9(i-1), B_3)=0$, and hence (cf.
Figure~\ref{fboard2}) $$\widetilde{E}(3\nu, N,
B_3)=\widetilde{E}(\nu,N-9(i-1), B_2).$$

\begin{figure}[htp]
\begin{center}
\begin{picture}(0,0)%
\includegraphics{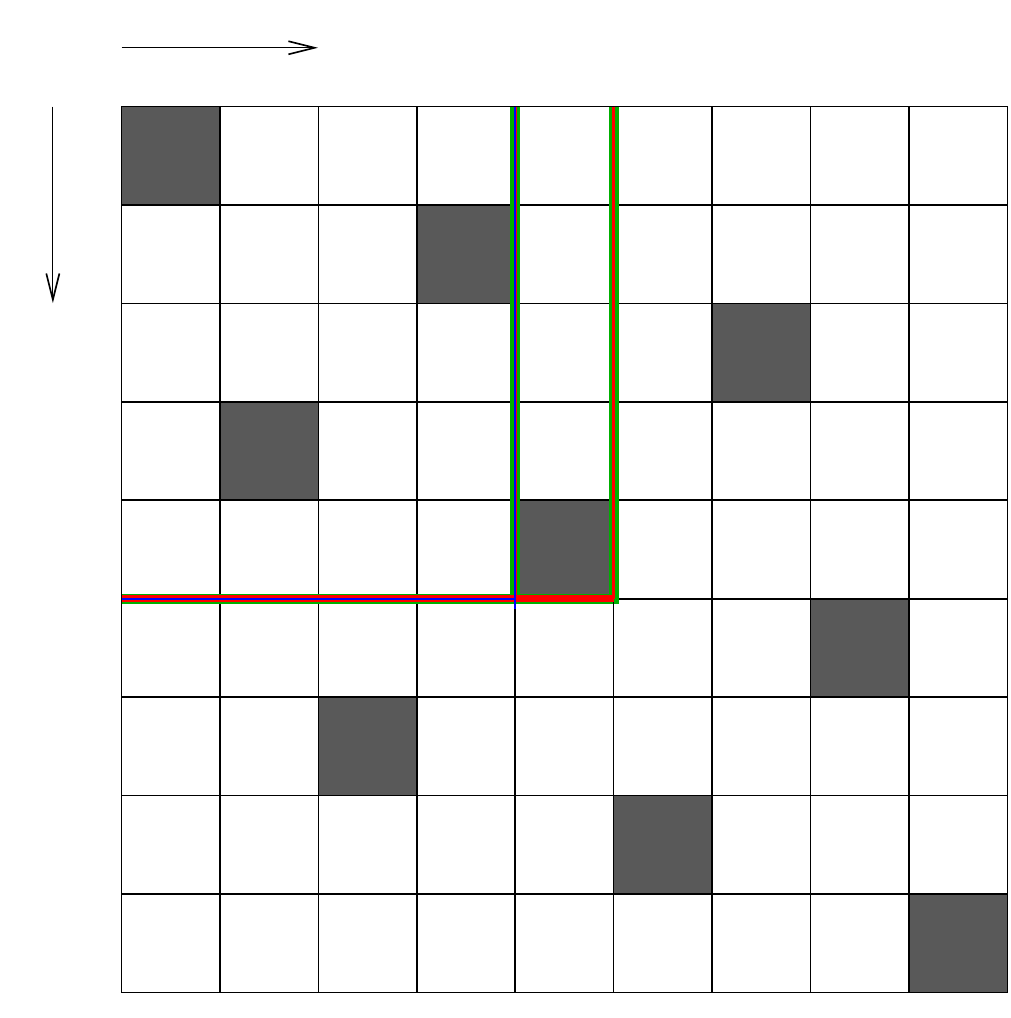}%
\end{picture}%
\setlength{\unitlength}{4144sp}%
\begingroup\makeatletter\ifx\SetFigFont\undefined%
\gdef\SetFigFont#1#2#3#4#5{%
  \reset@font\fontsize{#1}{#2pt}%
  \fontfamily{#3}\fontseries{#4}\fontshape{#5}%
  \selectfont}%
\fi\endgroup%
\begin{picture}(4617,4638)(-554,-3325)
\put(361,884){\makebox(0,0)[lb]{\smash{{\SetFigFont{12}{14.4}{\rmdefault}{\mddefault}{\updefault}{$1$}%
}}}}
\put(811,884){\makebox(0,0)[lb]{\smash{{\SetFigFont{12}{14.4}{\rmdefault}{\mddefault}{\updefault}{$2$}%
}}}}
\put(1261,884){\makebox(0,0)[lb]{\smash{{\SetFigFont{12}{14.4}{\rmdefault}{\mddefault}{\updefault}{$3$}%
}}}}
\put(1711,884){\makebox(0,0)[lb]{\smash{{\SetFigFont{12}{14.4}{\rmdefault}{\mddefault}{\updefault}{$4$}%
}}}}
\put(2116,884){\makebox(0,0)[lb]{\smash{{\SetFigFont{12}{14.4}{\rmdefault}{\mddefault}{\updefault}{$5$}%
}}}}
\put(2566,884){\makebox(0,0)[lb]{\smash{{\SetFigFont{12}{14.4}{\rmdefault}{\mddefault}{\updefault}{$6$}%
}}}}
\put(3061,884){\makebox(0,0)[lb]{\smash{{\SetFigFont{12}{14.4}{\rmdefault}{\mddefault}{\updefault}{$7$}%
}}}}
\put(3466,884){\makebox(0,0)[lb]{\smash{{\SetFigFont{12}{14.4}{\rmdefault}{\mddefault}{\updefault}{$8$}%
}}}}
\put(3961,884){\makebox(0,0)[lb]{\smash{{\SetFigFont{12}{14.4}{\rmdefault}{\mddefault}{\updefault}{$9$}%
}}}}
\put(-224,344){\makebox(0,0)[lb]{\smash{{\SetFigFont{12}{14.4}{\rmdefault}{\mddefault}{\updefault}{$1$}%
}}}}
\put(-224,-106){\makebox(0,0)[lb]{\smash{{\SetFigFont{12}{14.4}{\rmdefault}{\mddefault}{\updefault}{$2$}%
}}}}
\put(-224,-556){\makebox(0,0)[lb]{\smash{{\SetFigFont{12}{14.4}{\rmdefault}{\mddefault}{\updefault}{$3$}%
}}}}
\put(-224,-1006){\makebox(0,0)[lb]{\smash{{\SetFigFont{12}{14.4}{\rmdefault}{\mddefault}{\updefault}{$4$}%
}}}}
\put(-224,-1456){\makebox(0,0)[lb]{\smash{{\SetFigFont{12}{14.4}{\rmdefault}{\mddefault}{\updefault}{$5$}%
}}}}
\put(-224,-1906){\makebox(0,0)[lb]{\smash{{\SetFigFont{12}{14.4}{\rmdefault}{\mddefault}{\updefault}{$6$}%
}}}}
\put(-224,-2356){\makebox(0,0)[lb]{\smash{{\SetFigFont{12}{14.4}{\rmdefault}{\mddefault}{\updefault}{$7$}%
}}}}
\put(-224,-2806){\makebox(0,0)[lb]{\smash{{\SetFigFont{12}{14.4}{\rmdefault}{\mddefault}{\updefault}{$8$}%
}}}}
\put(-224,-3256){\makebox(0,0)[lb]{\smash{{\SetFigFont{12}{14.4}{\rmdefault}{\mddefault}{\updefault}{$9$}%
}}}}
\put(  1,884){\makebox(0,0)[lb]{\smash{{\SetFigFont{12}{14.4}{\rmdefault}{\mddefault}{\updefault}{$0$}%
}}}}
\put(-224,794){\makebox(0,0)[lb]{\smash{{\SetFigFont{12}{14.4}{\rmdefault}{\mddefault}{\updefault}{$0$}%
}}}}
\put(-44,1154){\makebox(0,0)[lb]{\smash{{\SetFigFont{12}{14.4}{\rmdefault}{\mddefault}{\updefault}{$\lambda$}%
}}}}
\put(-539,704){\makebox(0,0)[lb]{\smash{{\SetFigFont{12}{14.4}{\rmdefault}{\mddefault}{\updefault}{$N$}%
}}}}
\end{picture}%
\caption{Checkerboard $B_2$ corresponding to the first $3^2$
elements of $Y_3$.}
     \label{fboard2}
\end{center}
\end{figure}

\begin{figure}[htp]
\begin{center}
\begin{picture}(0,0)%
\includegraphics{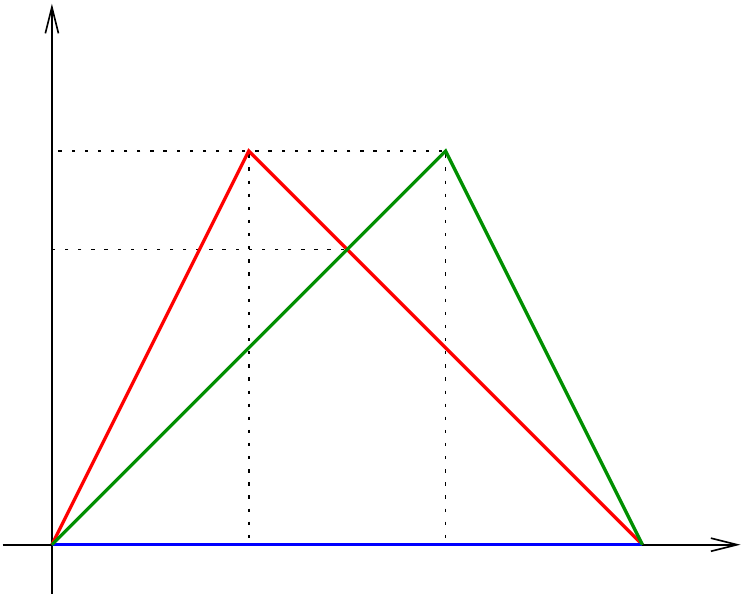}%
\end{picture}%
\setlength{\unitlength}{4144sp}%
\begingroup\makeatletter\ifx\SetFigFont\undefined%
\gdef\SetFigFont#1#2#3#4#5{%
  \reset@font\fontsize{#1}{#2pt}%
  \fontfamily{#3}\fontseries{#4}\fontshape{#5}%
  \selectfont}%
\fi\endgroup%
\begin{picture}(3399,2790)(2689,-4639)
\put(3781,-4561){\makebox(0,0)[lb]{\smash{{\SetFigFont{12}{14.4}{\rmdefault}{\mddefault}{\updefault}{$\frac{1}{3}$}%
}}}}
\put(4681,-4561){\makebox(0,0)[lb]{\smash{{\SetFigFont{12}{14.4}{\rmdefault}{\mddefault}{\updefault}{$\frac{2}{3}$}%
}}}}
\put(5581,-4561){\makebox(0,0)[lb]{\smash{{\SetFigFont{12}{14.4}{\rmdefault}{\mddefault}{\updefault}{$1$}%
}}}}
\put(2746,-2581){\makebox(0,0)[lb]{\smash{{\SetFigFont{12}{14.4}{\rmdefault}{\mddefault}{\updefault}{$\frac{2}{3}$}%
}}}}
\put(2746,-3031){\makebox(0,0)[lb]{\smash{{\SetFigFont{12}{14.4}{\rmdefault}{\mddefault}{\updefault}{$\frac{1}{2}$}%
}}}}
\put(2746,-4561){\makebox(0,0)[lb]{\smash{{\SetFigFont{12}{14.4}{\rmdefault}{\mddefault}{\updefault}{$0$}%
}}}}
\put(4141,-2761){\makebox(0,0)[lb]{\smash{{\SetFigFont{12}{14.4}{\rmdefault}{\mddefault}{\updefault}{$\psi_3^{\id}$}%
}}}}
\put(3331,-4111){\makebox(0,0)[lb]{\smash{{\SetFigFont{12}{14.4}{\rmdefault}{\mddefault}{\updefault}{$\varphi_{3,2}^{\id}$}%
}}}}
\put(4996,-4111){\makebox(0,0)[lb]{\smash{{\SetFigFont{12}{14.4}{\rmdefault}{\mddefault}{\updefault}{$\varphi_{3,1}^{\id}$}%
}}}}
\end{picture}%
\caption{Graph of $\varphi_{3,0}^\id$ (blue), $\varphi_{3,1}^\id$
(red) and $\varphi_{3,2}^\id$ (green);
$\psi_3^{\id}=\max(\varphi_{3,1}^\id,\varphi_{3,1}^\id)$.}
     \label{graphpsi}
\end{center}
\end{figure}
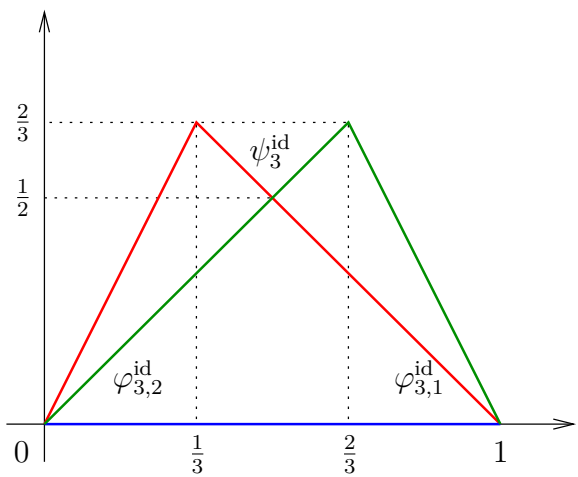

Now, it remains to compute the difference
$\delta=\widetilde{E}(\lambda, N, B_3)-\widetilde{E}(3\nu, N,
B_3)$ for each of the nine possibilities to make the step from
$B_3$ to $B_2$:
\begin{itemize}
\item If $\lambda=3\mu$ (blue lines) then for any $N$ we have $\delta=0$ and $\nu=\mu$.\\
\item If $\lambda=3\mu+1$ (red lines) then for $0<N\le 9$
we have $\delta=2N$ and $\nu=\mu+1$, and for $9<N\le 27$
we have $\delta=27-N$ and $\nu=\mu$.\\
\item If $\lambda=3\mu+2$ (green lines) then for $0<N\le
18$ we have $\delta=N$ and $\nu=\mu+1$, and for $18<N\le 27$ we
have $\delta=2(27-N)$ and $\nu=\mu$.
\end{itemize}

Coming back to the remainder and setting $\varphi_{3,0}^\id=0$,
$\varphi_{3,1}^\id(x)=\min(2x, 1-x)$, and
$\varphi_{3,2}^\id(x)=\min(x, 2(1-x))$ (see Figure~\ref{graphpsi})
it is easy to check that the three cases above lead to the first
step of Equation \eqref{DesLemForm}, where $1< i \le  3$ and
$\varepsilon=\lambda \bmod 3$:
$$
E\left (\frac{\lambda}{27},N,Y_b^\Sigma\right )=E\left
(\frac{\nu}{9}, N-9(i-1), Y_b^\Sigma\right
)+\varphi_{3,\varepsilon}^\id\left(\frac{N}{27}\right ).
$$
The next steps follow in the same way, but it is necessary to
control the intervals that appear inductively in the descending
recursion which explains the sophisticated definition of the
$\varepsilon_j$'s in Equation \eqref{DesLemForm} and the length of
the proof of this formula (5 pages in \cite{Fau05}).

\medskip

Concluding this illustration for base 3, we give the analog
of Theorem \ref{thmbfstar} for the triadic van der Corput sequence
$Y_3$ and the corresponding graph (Figure~\ref{f3'}) showing the
structure of the discrepancy by means of the function
$\psi_3^\id$, see \cite{fau1978}.
\begin{theorem}[Faure]\label{thmfauY3}
For $N \in \NN$ we have
\begin{equation}
ND_N^*(Y_3) = ND_N(Y_3)=\sum_{r=1}^{\infty} \psi_3^\id\left
(\frac{N}{3^r}\right )  = \sum_{r=1}^m
\psi_3^\id\left(\frac{N}{3^r}\right )+\frac{N}{3^m}\ \ \mbox{
whenever $1 \le N \le 3^m$},
\end{equation}
$$
ND_N^*(Y_3) \le \frac{\log N}{2 \log 3} +\frac{3}{4}+\frac{\log
2}{2 \log 3},
$$
and
$$
\limsup_{N \rightarrow \infty} \left(ND_N^*(Y_3)-\frac{\log N}{2
\log 3}\right)=\frac{3}{4}+ \frac{\log 2}{2 \log 3}.
$$
\end{theorem}

It is even possible to obtain generating formulas like for $Y_2$
(cf. \eqref{genrecY2}), but, due to the form of the function
$\psi_3^\id$, they are not as simple as for $Y_2$ where
$\psi_2^\id= \| \cdot \|$ (the ``hat-function''). Again the
recursion becomes particularly apparent when we use the notation
$D(N):=N D_N^{\ast}(Y_3)$. Set  $D(0)=0$, then for any $N \in
\NN_0$
$$
D(3N)=D(N), \ \
D(3N+\tfrac{1}{2})=\tfrac{2}{3}D(N)+\tfrac{1}{3}D(N+
\tfrac{1}{2})+\tfrac{1}{3},
$$
$$
D(3N+1)=\tfrac{1}{3}D(N)+\tfrac{2}{3}D(N+\tfrac{1}{2})+\tfrac{2}{3}
\ \ , \ \  D(3N+\tfrac{3}{2})=D(N+\tfrac{1}{2})+\tfrac{1}{2},
$$
$$
D(3N+2)=\tfrac{1}{3}D(N+1)+\tfrac{2}{3}D(N+\tfrac{1}{2})+\tfrac{2}{3},
$$
$$
D(3N+\tfrac{5}{2})=\tfrac{2}{3}D(N+1)+\tfrac{1}{3}D(N+\tfrac{1}{2})+\tfrac{2}{3}.
$$

\begin{figure}[ht]
\begin{center}
\includegraphics[width=115mm]{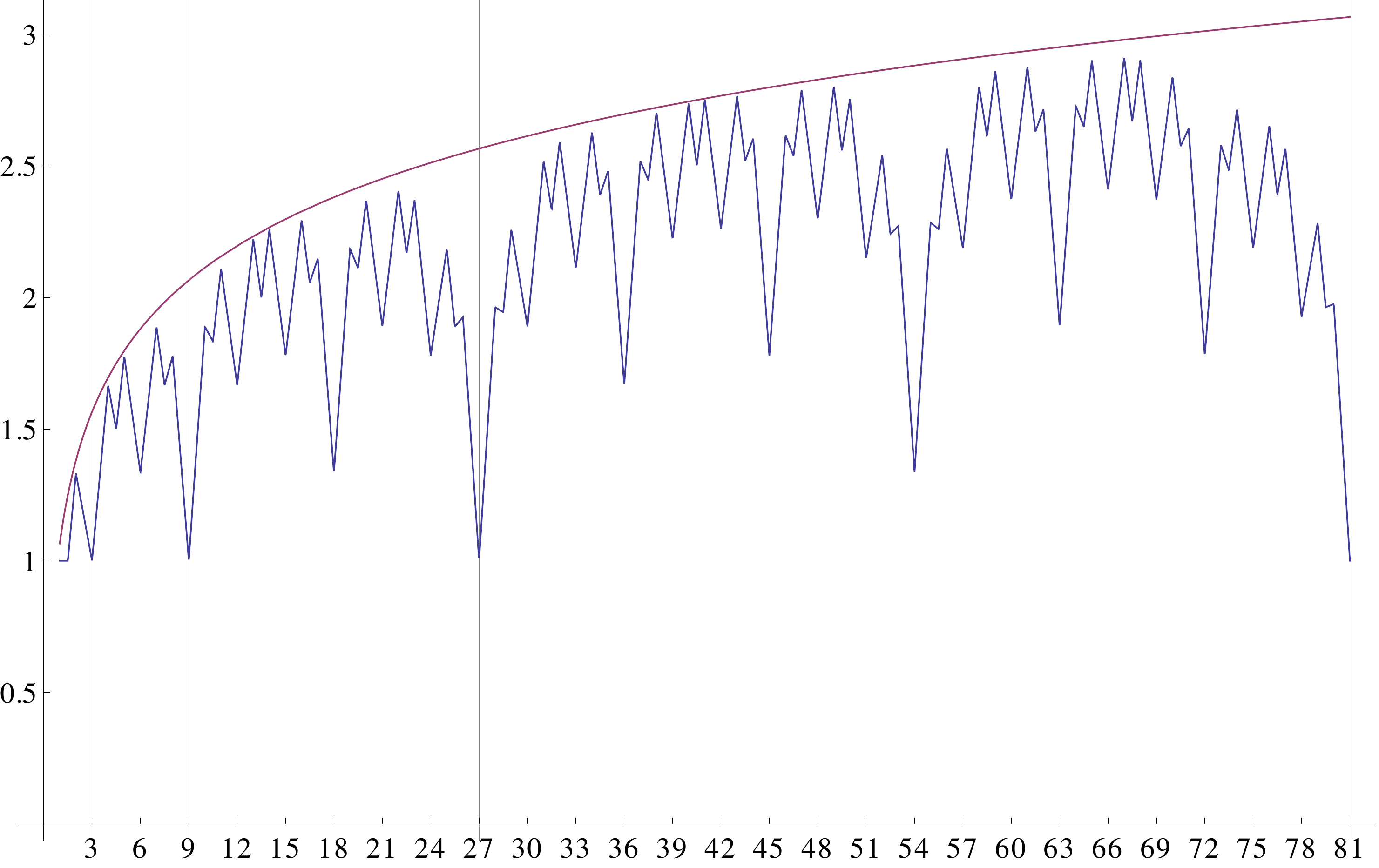}
\caption{$D(N)=N D_N^*(Y_3)$ of the van der Corput sequence in
base 3 for $N=1,2,\ldots,81$, in comparison with $N\mapsto
\tfrac{\log N}{2 \log 3}+\tfrac{3}{4}+\tfrac{\log 2}{2 \log 3}$.}
\label{f3'}
\end{center}
\end{figure}

For larger bases, the complexity of such generating formulas
increases and reduces the interest in searching for them.


\subsection{The von Neumann-Kakutani transformation}\label{NKtransform}

The van der Corput sequence can also be obtained from the von
Neumann-Kakutani transformation. This was first observed by
Lambert \cite{lamb1,lamb2}.

The {\it $b$-adic von Neumann-Kakutani transformation}
$T_b:[0,1)\rightarrow [0,1)$ is a piecewise translation map given
by
$$T_b(x) = x - 1 + \frac{1}{b^k} + \frac{1}{b^{k+1}}
\ \ \ \mbox{ for $x \in \left[1 - \frac{1}{b^k}, 1 -
\frac{1}{b^{k+1}}\right)$ with $k = 0, 1, 2, \ldots$.}
$$ 
A plot of
$T_b$ for $b=2$ is presented in Figure~\ref{f3}. The von
Neumann-Kakutani transformation is an ergodic and measure
preserving transformation $T_b:[0,1) \rightarrow [0,1)$, that is,
for every Borel set $A$, $T_bA=A$ only if $\lambda(A)=0$ or $\lambda(A)=1$,
and $\lambda(T_bA)=\lambda(A)$, where $\lambda$ denotes the
Lebesgue measure.
\begin{figure}[htp]
\begin{center}
\begin{picture}(0,0)%
\includegraphics{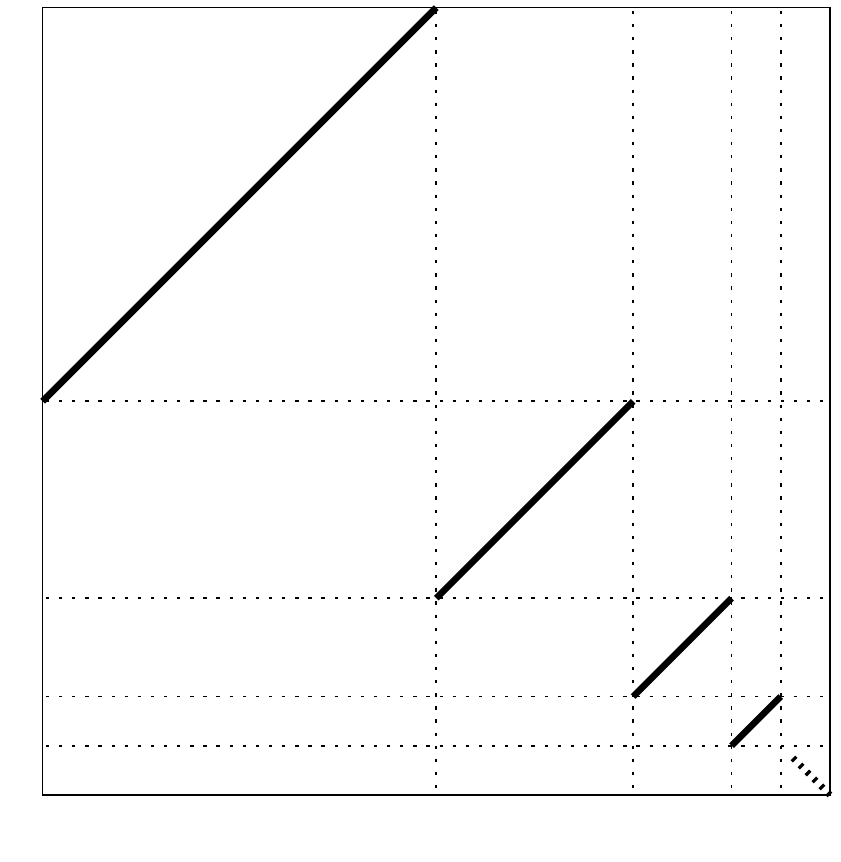}%
\end{picture}%
\setlength{\unitlength}{4144sp}%
\begingroup\makeatletter\ifx\SetFigFont\undefined%
\gdef\SetFigFont#1#2#3#4#5{%
  \reset@font\fontsize{#1}{#2pt}%
  \fontfamily{#3}\fontseries{#4}\fontshape{#5}%
  \selectfont}%
\fi\endgroup%
\begin{picture}(3828,3936)(2506,-4864)
\put(4456,-4786){\makebox(0,0)[lb]{\smash{{\SetFigFont{12}{14.4}{\rmdefault}{\mddefault}{\updefault}{$\frac{1}{2}$}%
}}}}
\put(2656,-4786){\makebox(0,0)[lb]{\smash{{\SetFigFont{12}{14.4}{\rmdefault}{\mddefault}{\updefault}{$0$}%
}}}}
\put(5806,-4786){\makebox(0,0)[lb]{\smash{{\SetFigFont{12}{14.4}{\rmdefault}{\mddefault}{\updefault}{$\frac{7}{8}$}%
}}}}
\put(5356,-4786){\makebox(0,0)[lb]{\smash{{\SetFigFont{12}{14.4}{\rmdefault}{\mddefault}{\updefault}{$\frac{3}{4}$}%
}}}}
\put(2521,-4156){\makebox(0,0)[lb]{\smash{{\SetFigFont{12}{14.4}{\rmdefault}{\mddefault}{\updefault}{$\frac{1}{8}$}%
}}}}
\put(2521,-4606){\makebox(0,0)[lb]{\smash{{\SetFigFont{12}{14.4}{\rmdefault}{\mddefault}{\updefault}{$0$}%
}}}}
\put(2521,-3706){\makebox(0,0)[lb]{\smash{{\SetFigFont{12}{14.4}{\rmdefault}{\mddefault}{\updefault}{$\frac{1}{4}$}%
}}}}
\put(2521,-2806){\makebox(0,0)[lb]{\smash{{\SetFigFont{12}{14.4}{\rmdefault}{\mddefault}{\updefault}{$\frac{1}{2}$}%
}}}}
\put(6031,-4786){\makebox(0,0)[lb]{\smash{{\SetFigFont{12}{14.4}{\rmdefault}{\mddefault}{\updefault}{$\frac{15}{16}$}%
}}}}
\end{picture}%
\caption{The von Neumann-Kakutani transformation for $b=2$.}
     \label{f3}
\end{center}
\end{figure}
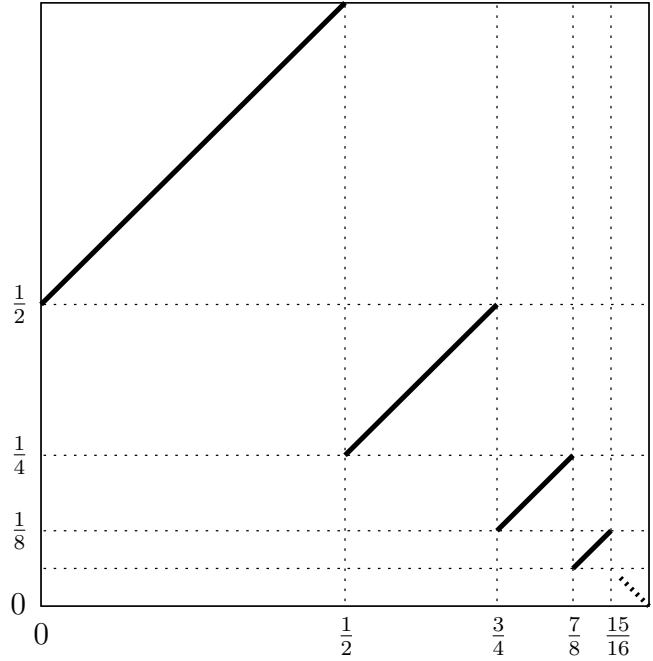

Likewise, $T_b$ can also be constructed inductively with a
so-called {\it cutting and stacking construction} which is often
used in ergodic theory: at first cut $[0,1)$ into $b$ intervals
$I_j^{(1)}=[\tfrac{j}{b},\tfrac{j+1}{b})$ for $j =
0,1,\ldots,b-1$. Define the transformation
$T_{b,1}:[0,\tfrac{b-1}{b}) \rightarrow [\tfrac{1}{b},1)$ as the
translation of $I_j^{(1)}$ onto $I_{j+1}^{(1)}$ for
$j=0,1,\ldots,b-1$. In the second step cut all intervals
$I_j^{(1)}$ into $b$ subintervals of the form
$I_j^{(2)}=[\tfrac{j}{b^2},\tfrac{j+1}{b^2})$ for
$j=0,1,\ldots,b^2 -1$. Then the transformation
$T_{b,2}:[0,\tfrac{b^2-1}{b^2}) \rightarrow [\tfrac{1}{b^2},1)$ is
the extension of $T_{b,1}$ which translates $I_{b^2-b+j}^{(2)}$
onto $I_{b^2-b+j+1}^{(2)}$ for $j=0,1,\ldots,b-1$. In this way one
can inductively construct a sequence of mappings
$T_{b,1},T_{b,2},T_{b,3},\ldots$. The first two steps for $b=2$
are illustrated in Figure~\ref{f4}. Finally the von
Neumann-Kakutani
transformation is defined as $T_b:=\lim_{n \rightarrow \infty} T_{b,n}$. See also \cite{oekt}.\\
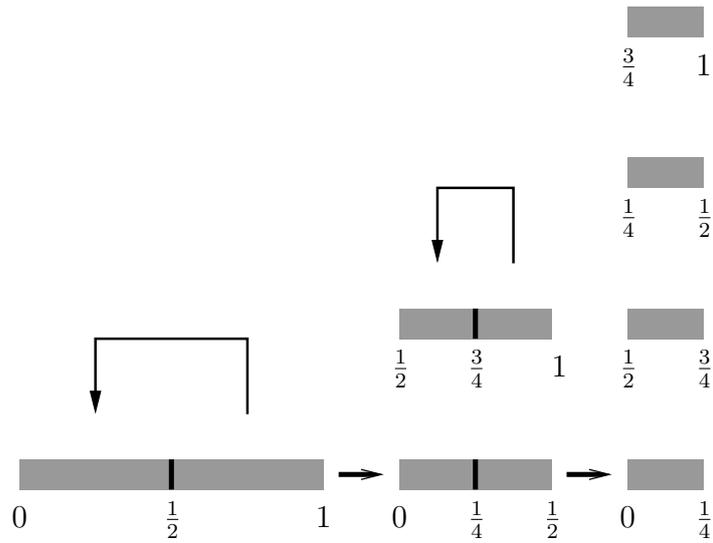
\begin{figure}[htp]
\begin{center}
\begin{picture}(0,0)%
\includegraphics{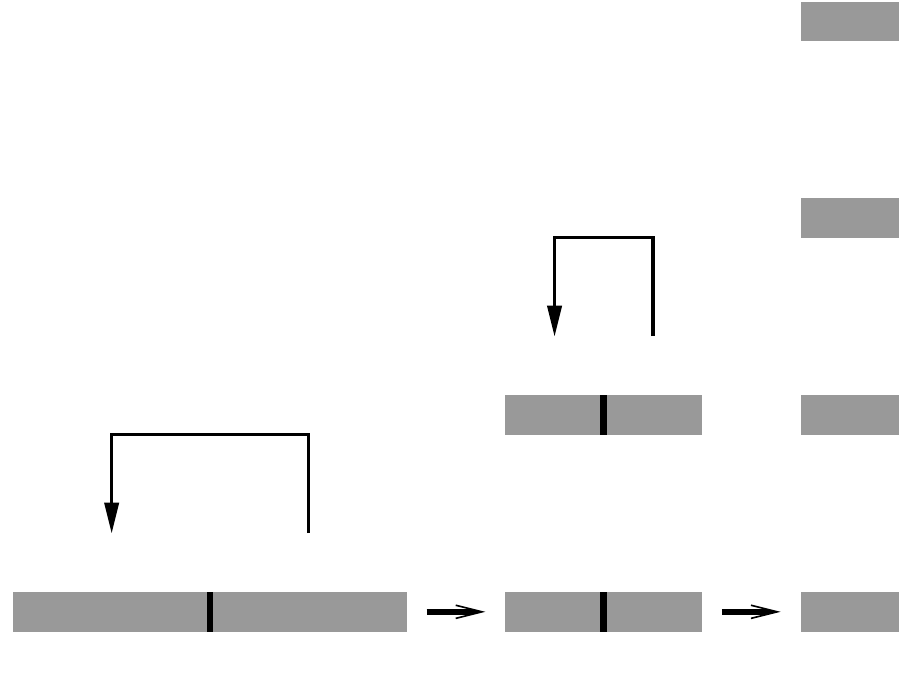}%
\end{picture}%
\setlength{\unitlength}{4144sp}%
\begingroup\makeatletter\ifx\SetFigFont\undefined%
\gdef\SetFigFont#1#2#3#4#5{%
  \reset@font\fontsize{#1}{#2pt}%
  \fontfamily{#3}\fontseries{#4}\fontshape{#5}%
  \selectfont}%
\fi\endgroup%
\begin{picture}(4111,3184)(3991,-2164)
\put(4006,-2086){\makebox(0,0)[lb]{\smash{{\SetFigFont{12}{14.4}{\rmdefault}{\mddefault}{\updefault}{$0$}%
}}}}
\put(5806,-2086){\makebox(0,0)[lb]{\smash{{\SetFigFont{12}{14.4}{\rmdefault}{\mddefault}{\updefault}{$1$}%
}}}}
\put(4906,-2086){\makebox(0,0)[lb]{\smash{{\SetFigFont{12}{14.4}{\rmdefault}{\mddefault}{\updefault}{$\frac{1}{2}$}%
}}}}
\put(6256,-2086){\makebox(0,0)[lb]{\smash{{\SetFigFont{12}{14.4}{\rmdefault}{\mddefault}{\updefault}{$0$}%
}}}}
\put(6706,-2086){\makebox(0,0)[lb]{\smash{{\SetFigFont{12}{14.4}{\rmdefault}{\mddefault}{\updefault}{$\frac{1}{4}$}%
}}}}
\put(7156,-2086){\makebox(0,0)[lb]{\smash{{\SetFigFont{12}{14.4}{\rmdefault}{\mddefault}{\updefault}{$\frac{1}{2}$}%
}}}}
\put(6256,-1186){\makebox(0,0)[lb]{\smash{{\SetFigFont{12}{14.4}{\rmdefault}{\mddefault}{\updefault}{$\frac{1}{2}$}%
}}}}
\put(7201,-1186){\makebox(0,0)[lb]{\smash{{\SetFigFont{12}{14.4}{\rmdefault}{\mddefault}{\updefault}{$1$}%
}}}}
\put(6706,-1186){\makebox(0,0)[lb]{\smash{{\SetFigFont{12}{14.4}{\rmdefault}{\mddefault}{\updefault}{$\frac{3}{4}$}%
}}}}
\put(7606,-2086){\makebox(0,0)[lb]{\smash{{\SetFigFont{12}{14.4}{\rmdefault}{\mddefault}{\updefault}{$0$}%
}}}}
\put(8056,-2086){\makebox(0,0)[lb]{\smash{{\SetFigFont{12}{14.4}{\rmdefault}{\mddefault}{\updefault}{$\frac{1}{4}$}%
}}}}
\put(7606,-286){\makebox(0,0)[lb]{\smash{{\SetFigFont{12}{14.4}{\rmdefault}{\mddefault}{\updefault}{$\frac{1}{4}$}%
}}}}
\put(8056,-286){\makebox(0,0)[lb]{\smash{{\SetFigFont{12}{14.4}{\rmdefault}{\mddefault}{\updefault}{$\frac{1}{2}$}%
}}}}
\put(7606,-1186){\makebox(0,0)[lb]{\smash{{\SetFigFont{12}{14.4}{\rmdefault}{\mddefault}{\updefault}{$\frac{1}{2}$}%
}}}}
\put(8056,-1186){\makebox(0,0)[lb]{\smash{{\SetFigFont{12}{14.4}{\rmdefault}{\mddefault}{\updefault}{$\frac{3}{4}$}%
}}}}
\put(7606,614){\makebox(0,0)[lb]{\smash{{\SetFigFont{12}{14.4}{\rmdefault}{\mddefault}{\updefault}{$\frac{3}{4}$}%
}}}}
\put(8056,614){\makebox(0,0)[lb]{\smash{{\SetFigFont{12}{14.4}{\rmdefault}{\mddefault}{\updefault}{$1$}%
}}}}
\end{picture}%
\caption{The cutting and stacking construction of $T_b$ for $b=2$.}
     \label{f4}
\end{center}
\end{figure}

Define the iterates of $T_b$ by $$T_b^{n} x =T_b T_b^{n-1}x \ \ \
\mbox{ for $n \in \NN$},$$ where we formally set $T_b^0x=x$. Then
it is easily seen that the orbit of the iterates of $T_b$ with
starting point $0$ is exactly the van der Corput sequence. For
example for $b=2$ we obtain (cf. Figure~\ref{f3}) $$0 \rightarrow
1/2 \rightarrow 1/4 \rightarrow 3/4 \rightarrow 1/8 \rightarrow
\cdots.$$

It is known that $T_b$ is not only ergodic, but it is even
uniquely ergodic. From this it follows that the orbit of the
iterates of $T_b$ is for every starting point $x \in [0,1)$
uniformly distributed modulo 1. Hence the sequence $$(T_b^n(x))_{n
\ge 0}$$ can be seen as a natural generalization of the van der
Corput sequence in base $b$, which is obtained for the choice
$x=0$. A proof of this fact and more information in this direction
can be found in the excellent survey \cite{GHL} by Grabner,
Hellekalek, and Liardet.

Pag\`{e}s \cite{page} showed that the orbit of the von
Neumann-Kakutani transformation with arbitrary starting point $x
\in [0,1)$ has star discrepancy of order of magnitude $O((\log
N)/N)$ which is optimal according to Schmidt's lower bound
\eqref{lowschmid}.


\subsection{Bounded remainder intervals for the van der Corput sequence}\label{bdremaindersets}

It follows from the lower bound \eqref{lowschmid} that for every
sequence $X$ in $[0,1)$ the sequence $(N D_N^*(X))_{N \in \NN}$ is
unbounded as $N \rightarrow \infty$. On the other hand the
question arises whether there exist $t \in [0,1)$ for which $N
\Delta_{X,N}(t)$ is bounded in $N$. If yes, then an interval
$[0,t)$ with this property is called a {\it bounded remainder
interval}. Bounded remainder intervals or, more generally, bounded
remainder sets have been investigated in a series of papers by
W.M. Schmidt, see, in particular, \cite{schcompo,sch1974}.

For the van der Corput sequence in base $b$ one can exactly
characterize those $t \in [0,1)$ which yield bounded remainder
intervals. This was first done by Hellekalek~\cite{hel1980}.

\begin{theorem}[Hellekalek]\label{bdremainder}
An interval $[0,t)$ is a bounded remainder interval for the van
der Corput sequence in base $b$ if and only if $t$ is a $b$-adic
rational, i.e., $t \in \QQ(b^m)$ for some $m \in \NN_0$.
\end{theorem}

The sufficiency of the condition in Theorem \ref{bdremainder} is almost trivial. Writing
the interval $[0,a b^{-m})$ with $a \in \{0,1,\ldots,b^m -1\}$ as
the disjoint union $\bigcup_{k=0}^{a-1}[k b^{-m},(k+1)b^{-m})$ it
follows from \eqref{anzbadicint} that, for suitable $\theta_k \in
\{0,1\}$,
$$N \Delta_{Y_b,N}(a b^{-m})=\sum_{k=0}^{a-1}
\left(\left\lfloor\frac{N}{b^m}\right\rfloor +\theta_k\right) - N
\frac{a}{b^m} \ll 1.$$ The necessity of the condition needs some
investigation for which we refer to \cite{hel1980}. Proofs of the
result can also be found in \cite{fau1983,fau1986,GHL}.
Quantitative versions of Theorem~\ref{bdremainder} are given in
\cite{DLP05,fau1983,fau1986}.

We begin with the van der Corput sequence in base $2$ and present
the result of \cite{DLP05}. For $t \in [0,1)$ with base 2
representation $t=\tau_1/2+\tau_2/2^2+\cdots$ and $m \in \NN$
define
$$f_{t}(m):=\#\{j<m\, : \, \tau_j \not= \tau_{j+1}\}\ \ \ \mbox{ and }\ \ \
g_{t}(m):=\frac{1}{2} \sum_{u=0}^{m-1}\|2^u t\|.$$
We have the following result.
\begin{theorem}[Drmota, Larcher, and Pillichshammer]\label{th7}
Let $t \in [0,1)$ with an infinite base 2 representation. Then for
$\varepsilon >0$ we have
\begin{eqnarray*}
\lim_{m \rightarrow \infty}\frac{\#\left\{N \le 2^m :
(1-\varepsilon)g_t(m) < \Delta_{Y_2,N}(t) <
(1+\varepsilon)g_t(m)\right\}}{2^m}=1.
\end{eqnarray*}
\end{theorem}

The functions $f_{t}$ and $g_{t}$ are intimately related by the
inequality
$$\frac{1}{8}(f_t(m)+1) \le g_t(m) \le \frac{1}{2}(f_t(m)+1) \ \ \
\mbox{ for all $t \in [0,1)$ and $m \in \NN$,}$$ see
\cite[Lemma~5]{DLP05}. So we obtain the following corollary.

\begin{corollary}\label{cor1}
Let $t \in [0,1)$ with an infinite base 2 representation. Then for
$\varepsilon >0$ we have
\begin{eqnarray*}
\lim_{m \rightarrow \infty}\frac{\#\left\{N \le 2^m :
(1-\varepsilon)f_t(m) < \Delta_{Y_2,N}(t) \right\}}{2^m}=1.
\end{eqnarray*}
\end{corollary}

Note that Corollary \ref{cor1} is a strong quantitative version of
Theorem~\ref{bdremainder} (for $b=2$). It follows from
\cite[Theorem~9]{DLP05} that for every $N \in \NN$ we have
$|\Delta_{Y_2,N}(t)| \le f_t(\lfloor \log_2 N\rfloor)+4$ for every
$t \in [0,1)$. This result even holds for more general digital
$(0,1)$-sequences over $\ZZ_2$, the definition of which can be
found in
Section~\ref{secdig01seq}.\\

Next, we present another quantitative version of Theorem
\ref{bdremainder} for generalized van der Corput sequences in base
$b$. We first need some notation: For $a \in [0,1)$ and a real
$y>0$, let $u_a(y,k)$ be the number of occurrences of $k\in \ZZ_b$ among the 
first $\lfloor y \rfloor +1$ digits in the
$b$-adic expansion of $a$,and
set $U_a(y)=\sum_{k=1}^{b-2} u_a(y,k)$. Further we call $V_a(y)$
the number of pairs $(0,b-1)$ or $(b-1,0)$ among the 
first $\lfloor y \rfloor +1$ digits in the
$b$-adic expansion of $a$. Let also $\log_b$ denote the logarithm to base $b$, i.e., $\log_b x=(\log x)/\log b$. The following
result is \cite[Th\'eor\`eme 1]{fau1983}:
\begin{theorem}[Faure]\label{th1F83}
For any sequence $Y_b^\Sigma$ we have the following estimates.
\begin{enumerate}
\item For every $N \in \NN$ we have 
$$
\ N \vert \Delta_{Y_b^\Sigma,N}(a) \vert \le
\frac{b+4}{4}U_a(\log_b N)+\frac{(b-1)^2}{2 b^2} V_a(\log_b
N)+\frac{b}{4}+3.
$$
\item \ There exist infinitely many $N \in \NN$ such that
$$
N \vert \Delta_{Y_b^\Sigma,N}(a) \vert \ge \frac{1}{4}U_a(\log_b
N)+\frac{1}{2} V_a(\log_b N)-1.
$$
\end{enumerate}
\end{theorem}

As an immediate consequence, we obtain the generalization of Theorem
\ref{bdremainder} to generalized van der Corput sequences in base
$b$. A simpler version exists for $Y_2$, see \cite[Corollaire
2]{fau1983}. The proof heavily relies on the properties of
$\psi$-functions introduced in Section~\ref{secdiscgenvdC} for the
study of $D_N(Y_b^\Sigma$).

There exists yet another formula for the remainder of van der Corput
sequences $Y_b$ that draws an interesting parallel between these
sequences and $(n\alpha)$-sequences. It was proved (see
\cite[Th\'eor\`eme 2]{fau1983}) using the famous property of
elementary intervals (see Definition~\ref{def01seq} in Section
\ref{secdig01seq}) for the van der Corput sequence.
\begin{theorem}[Faure]\label{th2F83}
For given $a \in [0,1)$ and $N \in \NN$, let their $b$-adic
expansions be $a=\sum_{j=0}^\infty a_j b^{-j-1}$ and
$N=\sum_{i=0}^\infty N_i b^i$. Put $I(N,a)=\sum_{i=0}^\infty
\min(N_i,a_i)$ and let $W(N,a)$ be the number of pairs $(-,+)$ in
the sequence of signs of $(a_i-N_i)_{i \ge 0}$ (disregarding
zeros). Then we have
$$
N \Delta_{Y_b,N}(a)=-\sum_{j=0}^\infty \sum_{i=0}^j N_i a_j
b^{i-j-1} + I(N,a) + W(N,a).
$$
\end{theorem}

The structure of this formula for $Y_b$ is exactly the same as
that for $(n\alpha)$-sequences obtained in \cite[Equation
2.5]{Gil68}. For the sake of completeness, we recall this formula:

\medskip

\noindent{\it For given $x \in [0,1)$ and $N \in \NN$, let
$x=\sum_{j=1}^\infty x_jm_j$ and $N=\sum_{i=1}^\infty y_i q_i$
(where $y_i=0$ for sufficiently large $i$) be their canonical expansions
in semi-regular continued fractions (the $(p_i/q_i)$'s approach
$\alpha$ and $m_i=q_i\alpha-p_i$). Put $s_{i,j}=q_im_j-p_i^j$,
$I(N,x)=\sum_{i=1}^\infty \min(x_i,y_i)$ and let $W(N,x)$ be the
number of pairs $(-,+)$ in the sequence of signs of $(x_i-y_i)_{i
\ge 0}$ (disregarding zeros). Then we have
\begin{equation*}
N \Delta_{(n\alpha),N}(a)=-\sum_{j=0}^\infty \sum_{i=0}^j x_i y_j
s_{i,j} + I(N,x) + W(N,x).
\end{equation*}}

These formulas underline the relation between van der Corput
sequences and $(n\alpha)$-sequences through systems of numeration
that permit the precise study of their irregularities of
distribution. The simplicity of $b$-adic systems underlying van
der Corput sequences as compared to continued fraction systems for
$(n\alpha)$-sequences, appears to be the reason for a better
understanding, and also a better asymptotic behavior of the former. In
the same vein, these two families of sequences are included in a
much larger family of sequences of minimal discrepancy, the
so-called ``self similar sequences'' introduced by Borel in
\cite{borel}.

\medskip

From Theorem \ref{th2F83} it is possible to obtain an improvement
of  \cite[Th\'eor\`eme 1 and Corollaire 2]{fau1983} for the van
der Corput sequence $Y_2$, see \cite[Corollaire 3 and Section
4.5]{fau1983} .

\medskip

The paper \cite{fau1986} is a variant of Bohl's lemma for
$(n\alpha)$-sequences applied to generalized van der Corput
sequences. Bohl's lemma for $(n\alpha)$-sequences asserts that the
remainder on an interval $J\subseteq [0,1)$ only depends the length of
$J$. Its very short proof is based on the density of  $(n\alpha)$
and on the fact that the counting function does not change under
translations of $J$ by $\alpha \pmod 1$. In contrast to
systems of numeration, the variant for $Y_b^\Sigma$ sequences is
much more involved (see \cite[Th\'eor\`eme 1]{fau1986}), but it
leads to the same property for bounded remainder intervals, i.e.,
for any $Y_b^\Sigma$ sequence, $J$ is a bounded remainder interval
if and only if ${\rm length}(J)$ is a $b$-adic rational.

\subsection{Functions with bounded remainder}

The question dealt with in Section~\ref{bdremaindersets} can be
generalized. Let $f:[0,1]\rightarrow \RR$ be a Riemann-integrable
function with $\int_0^1 f(x) \rd x=0$ and let $X=(x_n)_{n \ge 0}$
be a sequence in $[0,1)$. Under which condition on $f$ do we have
$$\sum_{n=0}^{N-1}f(x_n) \ll 1?$$ The classical example studied in
Section~\ref{bdremaindersets} corresponds to $f(x)=1_{[0,t)}(x)-t$
for $t \in [0,1)$.

In \cite{hellar} Hellekalek and Larcher studied this
question for sequences induced by the von Neumann-Kakutani
transformation $T_b$ (and also sequences induced by irrational
rotations). For $x \in [0,1)$ define
$$G_N(x):=\sum_{n=0}^{N-1}f(T_b^n x).$$ Recall that $(T_b^n(0))_{n
\ge 0}$ is the van der Corput sequence in base $b$ (see
Section~\ref{NKtransform}).

\begin{theorem}[Hellekalek and Larcher]\label{thmHelLar}
Let $f \in C^1([0,1])$ with $\int_0^1 f(x) \rd x=0$ and Lipschitz
continuous first derivative $f'$. Then we have:
\begin{enumerate}
 \item If $f(0)=f(1)$, then $\sup_N|G_N(x)|< \infty$ for all $x \in [0,1)$.
 \item If $\sup_N|G_N(x)|< \infty$ for some $x \in [0,1)$, then $f(0)=f(1)$.
 \item If $f(1)< f(0)$, then $-\infty < \liminf_{N \rightarrow \infty} G_N(0)$ and $\limsup_{N \rightarrow \infty} G_N(0) = \infty$.
 \item If $f(1)> f(0)$, then $-\infty = \liminf_{N \rightarrow \infty} G_N(0)$ and $\limsup_{N \rightarrow \infty} G_N(0) < \infty$.
\end{enumerate}
\end{theorem}

Recall that the Theorem of Rademacher states that every Lipschitz
continuous function on an interval is continuously differentiable
in almost every point of the interval (in the sense of Lebesgue
measure). Hence for the functions considered in
Theorem~\ref{thmHelLar} we obtain (by using partial integration)
that for $h\not=0$ we have
\begin{eqnarray*}
\widehat{f}(h) = \int_0^1f(x) {\rm e}^{-2 \pi \icomp h x}\rd x =
\frac{1}{2 \pi \icomp h} \int_0^1f'(x) {\rm e}^{-2 \pi \icomp h
x}\rd x = \frac{1}{(2 \pi \icomp h)^2} \int_0^1f''(x) {\rm e}^{-2
\pi \icomp h x}\rd x,
\end{eqnarray*}
and therefore $|\widehat{f}(h)| \ll |h|^{-2}$.

This motivates the following definition: let
$\mathcal{F}_{\alpha}$ be the class of all 1-periodic functions
$f:\RR \rightarrow \RR$ with an absolutely convergent Fourier
series expansion with $|\widehat{f}(h)| \ll |h|^{-\alpha}$ for all
$h \not=0$ and $\int_0^1f(x) \rd x=0$. In \cite{DP06} Dick and Pillichshammer proved the following result
for the classical van der Corput sequence in base $b$:

\begin{theorem}[Dick and Pillichshammer]
Let $Y_b=(y_n)_{n \ge 0}$ be the van der Corput sequence in base
$b$. For $\alpha>1$ we have
\begin{eqnarray}\label{result1}
\left|\sum_{n=0}^{N-1}f(y_n)\right|\ll 1\quad \ \mbox{ for all $f
\in \mathcal{F}_{\alpha}$}.
\end{eqnarray}
Furthermore, there exists a function $g \in \mathcal{F}_1$ such
that $$\limsup_{N \rightarrow \infty}
\left|\sum_{n=0}^{N-1}g(y_n)\right|=\infty.$$
\end{theorem}

More accurate results in this direction can be found in
\cite{DP06}.

\subsection{Subsequences of the van der Corput sequence}

In recent years also distribution properties of subsequences (or
of index transformed versions) of the van der Corput sequence have
been studied. These properties are intimately related to uniform
distribution of integer sequences. For $m \in \NN$ a sequence
$(k_n)_{n \ge 0}$ of integers is uniformly distributed modulo $m$
if
$$\lim_{N\rightarrow \infty} \frac{\#\{n \in \{0,1,\ldots,N-1\}\ :
\ k_n\equiv j \pmod{m}\}}{N}=\frac{1}{m}\ \ \ \mbox{ for all $j
\in\{0,1,\ldots,m-1\}$}.
$$ 
For information on uniform distribution
of integer sequences we refer to \cite[Chapter~5,
Section~1]{kuinie}.

Let $Y_b=(y_n)_{n \ge 0}$ be the van der Corput sequence in base
$b$ and let $(k_n)_{n \ge 0}$ be a sequence in $\NN_0$. Then it is
well known (see, e.g., \cite[Remark~4.4]{HLKP}) that the sequence
$(y_{k_n})_{n \ge 0}$ is uniformly distributed modulo 1 if and
only if the sequence $(k_n)_{n \ge 0}$ is uniformly distributed
modulo $b^d$ for all $d \in \NN_0$.

From this insight together with known results on distribution
properties of integer sequences (see, e.g., \cite{kuinie} and the
references therein) one can obtain the following results which are
taken from \cite{HLKP}.

Let $Y_b=(y_n)_{n \ge 0}$ denote the van der Corput sequence in
base $b$. Then the following statements hold true:
\begin{enumerate}
\item The sequence $Y_b^{u,v}=(y_{u n+v})_{n \ge 0}$ for $(u,v)
\in \NN \times \NN_0$ is uniformly distributed modulo 1 if and
only if $\gcd(u,b)=1$. 
\item For every $\alpha \in \NN$, $\alpha
\ge 2$, the sequence $(y_{n^\alpha})_{n \ge 0}$ is not uniformly
distributed modulo~1. 
\item Let $p_n$, $n \in \NN_0$, denote the
$(n+1)$-st prime number, i.e., $p_0=2, p_1=3, p_2=5, \ldots$. Then
the sequence $(y_{p_n})_{n \ge 0}$ is not uniformly distributed
modulo~1. 
\item Let $F_n$, $n \in \NN_0$, denote the $n$-th
Fibonacci number, i.e., $F_0=0, F_1=1$, and $F_n=F_{n-1}+F_{n-2}$
for $n \ge 2$. Then the sequence $Y_b^{{\rm Fib}}=(y_{F_n})_{n \ge
0}$ in base $b$ is uniformly distributed modulo 1 if and only if
$b=5^k$ for some $k \in \NN$. 
\item Let $F_n$ be as above. Then,
in any base $b$, the sequence $(y_{\lfloor\log F_n \rfloor})_{n
\ge 1}$ is uniformly distributed modulo 1. 
\item Let $\xi \in \RR
\setminus \QQ$ or $\xi=1/d$ for some nonzero integer $d$. Then, in
any base $b$, the sequence $(y_{\lfloor n \xi \rfloor})_{n \ge 0}$
is uniformly distributed modulo 1. 
\item For an integer $q \ge 2$
and $n \in \NN_0$ let $s_{q}(n)$ denote the $q$-ary sum-of-digits
function. Then, in any base $b$, the sequence
$Y_b^{s_q}=(y_{s_{q}(n)})_{n \ge 0}$ is uniformly distributed
modulo~1.
\end{enumerate}

In some cases we also have estimates for the star discrepancy of
index transformed van der Corput sequences which we collect in the
following theorem.

\begin{theorem}
\begin{enumerate}
\item Let $(u,v) \in \NN \times \NN_0$ with $\gcd(u,b)=1$ and let
$Y_b^{u,v}$ be the van der Corput sequence in base $b$ indexed by
the arithmetic progression $u n+v$. Then we have
$$D_N^{\ast}(Y_b^{u,v})\le D_N^{\ast}(Y_b).$$
\item Let $b=5^{\ell}$ and let $Y_b^{{\rm Fib}}$ be the van der
Corput sequence in base $b$ indexed by the Fibonacci numbers. Then
we have $$D_N^{\ast}(Y_b^{{\rm Fib}}) \ll_b \frac{1}{\sqrt{N}}.$$
\item Let $b,q \ge 2$ be integers and let $Y_b^{s_q}$ be the van
der Corput sequence in base $b$ indexed by the $q$-ary
sum-of-digits function $s_q$. Then
$$\frac{1}{\sqrt{\log N}} \ll_q D_N^{\ast}(Y_b^{s_q}) \ll_{q,b} \frac{\log \log N}{\sqrt{\log N}}.$$
If $q<14$, then the upper bound can be improved to
$$D_N^{\ast}(Y_b^{s_q}) \ll_{q,b} \frac{1}{\sqrt{\log N}}.$$ \item
Let $\alpha\in (0,1)$ and let $Y_b^{\alpha}$ be the van der Corput
sequence in base $b$ indexed by $\lfloor n^{\alpha}\rfloor$. Then
$$\frac{1}{N^{\alpha}}\ll_{\alpha} D_N^*(Y_b^{\alpha})\ll_{\alpha} \frac{\log N}{N^{\alpha}}.$$
\end{enumerate}
\end{theorem}

The first point is \cite[Theorem~6.1]{HLKP}, the second one is
\cite[Theorem~2.1]{P12Fib}. The third and the fourth points are
taken from \cite{KrLaPi}.

\subsection{Distribution properties of consecutive elements of the van der Corput sequence}

One test for the statistical independence of successive
pseudorandom numbers is the so-called {\it serial test}. This test
can be viewed as a way of measuring the deviation from the
property of complete uniform distribution. For a given sequence
$X=(x_n)_{n \ge 0}$ and a fixed dimension $s \ge 2$, consider the
overlapping tuples
$$\bsx_n=(x_n,x_{n+1},\ldots,x_{n+s-1}) \ \ \ \mbox{ for $n \in \NN_0$}$$
and consider the discrepancy of the resulting $s$-dimensional
sequence $X^{(s)}=(\bsx_n)_{n \ge 0}$. For a random sequence
$Z=(\bsz_n)_{n \ge 0}$ in $[0,1)^s$ the law of the iterated
logarithm states that
$$\limsup_{N \rightarrow \infty}\frac{\sqrt{2N} D_N^*(Z)}{\sqrt{\log\log N}}=1 \ \ \ \lambda_s^{\infty}{\rm -a.e.},$$
where $\lambda_s$ is the $s$-dimensional Lebesgue measure. This
serves as a benchmark for the order of magnitude of
$D_N^*(X^{(s)})$. More information on the serial test and further
statistical tests for the goodness of fit for the empirical distribution 
of uniform pseudorandom numbers can be found in the book of Niederreiter~\cite[Section~7.2]{niesiam}.\\

We now return to the van der Corput sequence. Let
$Y_b^{(s)}=(\bsy_n)_{n \ge 0}$ where
$\bsy_n=(Y_b(n),Y_b(n+1),\ldots,Y_b(n+s-1))$. In \cite{blaze} 
Bla\v{z}ekov\'{a} showed the following result for $s=2$:

\begin{theorem}[Bla\v{z}ekov\'{a}]
For $N \in \NN$ we have
$$D_N^*(Y_b^{(2)})= \max\left(\frac{1}{b}\left(1-\frac{1}{b}\right),\frac{1}{4}  \left(1-\frac{1}{b}\right)^2\right)+O(D_N^*(Y_b)).$$
\end{theorem}

In particular, $\lim_{N \rightarrow \infty} D_N^*(Y_b^{(2)})=
\max(\frac{1}{b}(1-\frac{1}{b}),\frac{1}{4} (1-\frac{1}{b})^2)>0$.
Thus the van der Corput sequence $Y_b$ does not pass the serial
test for $s=2$ and is therefore not a pseudorandom sequence. In
\cite{fialstr} Fialova and Strauch showed that every point
of $Y_b^{(2)}$ lies on the line segment described by the formula
$$y=x-1+\frac{1}{b^k}+\frac{1}{b^{k+1}} \ \ \
\mbox{ for $x \in\left[1-\frac{1}{b^k},1-\frac{1}{b^{k+1}}\right]$
with $k \in\NN_0$}$$ and they calculated the limit distribution of
the sequence $Y_b^{(2)}$.

These findings have been generalized by Aistleitner and
Hofer~\cite{aisthof} to general $s \ge 2$: let $T_b$ denote the
$b$-adic von Neumann-Kakutani transformation. Define the map
$\gamma:[0,1) \rightarrow [0,1)^s$ by setting
$$\gamma(t):=(t,T_bt,T_b^2 t,\ldots,T_b^{s-1} t)$$ and set
$$\Gamma:=\{\gamma(t) \ : \ t \in [0,1)\}.$$ The Lebesgue measure
$\lambda$ on $[0,1)$ induces a measure $\nu$ on $\Gamma$ by
setting $$\nu(A)=\lambda(\{t \ : \ \gamma(t)\in A\})\ \ \mbox{ for
$A \subseteq \Gamma$.}$$  Furthermore, $\nu$ induces a measure
$\mu$ on $[0,1)^s$ by embedding $\Gamma$ into $[0,1)^s$, i.e.,
$$\mu(B)=\nu(B \cap \Gamma) \ \ \mbox{ for $B \subseteq [0,1)^s$.}$$

\begin{theorem}[Aistleitner and Hofer]\label{aishofthm}
The limit measure of $Y_b^{(s)}$ is $\mu$.
\end{theorem}

Note that the measure $\mu$ is a copula on $[0,1]^s$ since every
distribution function of a multi-dimensional sequence  is a copula
if all uni-variate projections of the sequence are uniformly
distributed modulo 1.

The proof of Theorem~\ref{aishofthm} in \cite{aisthof} is based on
the fact that the van der Corput sequence can be viewed as the
orbit of the origin under the ergodic von
Neumann-Kakutani transformation (see Section~\ref{NKtransform}).\\

Recently Lertchoosakul and Nair \cite{LertRad} studied the
same question for sequences of the form
$$(Y_b(k_n+n_1),\ldots,Y_b(k_n+n_s))_{n \ge 0},$$ where
$(n_1,\ldots,n_s)$ is a fixed $s$-tuple of non-negative integers
and where the integer sequence $(k_n)_{n \ge 0}$ is Hartmann
uniformly distributed on $\ZZ$. It is shown in \cite{LertRad} that
the asymptotic distribution function exists and is a copula.

\section{Generalizations and variants of the van der Corput sequence}\label{sec_vari}


Several authors have studied different ways of generalizing or
modifying the van der Corput sequence. In this section, we
describe some of the most important variants.


\subsection{Symmetrized generalized van der Corput sequences}\label{sec_sym}

One way to overcome the defect that, compared to the lower bound
in \eqref{lowproinov}, the van der Corput sequence does not have
optimal order of $L_p$-discrepancy is based on a process
which was initially introduced by Davenport for
$(n\alpha)$-sequences (see \cite{Dav56} or
\cite[Theorem~1.75]{DT97}). This method is also known as {\it
Davenport's reflection principle}. Later
Proinov~\cite{pro1988a}, who named the process {\it symmetrization
of a sequence}, obtained the same result with generalized van der
Corput sequences.

\begin{definition}\rm
A sequence $Y=(y_n)_{n \ge 0}$ is called {\it symmetric} if for
all $n \in \NN_0$, one has $ y_{2n}+y_{2n+1}=1$. A symmetric
sequence $Y=(y_n)_{n \ge 0}$ is said to be {\it produced} by a
sequence $X=(x_n)_{n \ge 0}$ if for all $n\in \NN_0$ one has $
y_{2n}=x_n$ or $y_{2n+1}=x_n$. In this case, the sequence $Y$ is
called a {\it symmetrized version of} $X$ and is denoted as $Y=
\widetilde{X}$.
\end{definition}

Two years after Proinov, Faure \cite[Th\'eor\`eme 3]{fau1990}
obtained an exact formula for the $L_2$-discrepancy of the
symmetrized  sequence $\widetilde{Y}_2$ of the van der Corput
sequence $Y_2$ (for short,  the {\it symmetrized  van der Corput
sequence}).  The proof is based on the descent method (see Section
\ref{sec_Faure_meth}) worked out for the first time for that study
of $Y_2$ and $\widetilde{Y}_2$.

We remark that the symmetrized van der Corput sequence in base 2
can also be obtained with the help of the tent-transformation
$\Psi:[0,1]\rightarrow [0,1]$, given by $\Psi(x)=1-|2x-1|$, which
is a Lebesgue measure preserving function. If $Y_2=(y_n)_{n \ge
0}$ denotes the van der Corput sequence in base 2, then
$$\widetilde{Y}_2=\Psi(Y_2)=(\Psi(y_n))_{n \ge 0}.$$

\begin{theorem}[Faure]\label{L2symvdc2} For any positive integers $m$ and $N$ such that $1\le N \le2^m$, we have
\[
(NL_{2,N}(\widetilde{Y}_2))^2=\sum_{j=1}^m (1-2\|2^jY_2(N+1)\|)\
\left\|\frac{N}{2^j}\right\|^2+\frac{N^2}{3\cdot 4^m}\cdot
\]
\end{theorem}

\begin{corollary}\label{corL2symvdc2}
The preceding theorem implies the following formula and estimate:
\[
t_2(\widetilde{Y}_2)=\inf_{m \ge 1} \max_{1\le N \le 2^m} \left
(\frac{1}{m \log 2}\sum_{j=1}^m (1-2\|2^jY_2(N+1)\|)\ \left
\|\frac{N}{2^j}\right\|^2\right )
\]
and
\[
  0.089\ldots=\frac{421}{6750 \log 2} \le t_2(\widetilde{Y}_2) <0.103.
\]
\end{corollary}

Numerical computations suggest that the lower bound in Corollary \ref{corL2symvdc2} is the
exact value of $t_2(\widetilde{Y}_2)$.

\medskip

The proof of the preceding result is based on the descent method
which permits a very precise study of the $L_2$-discrepancy of
$\widetilde{Y}_2$, but which (until now) is ineffective for the
$L_p$-discrepancy. The optimal order of $L_{2,N}(\widetilde{Y}_2)$
was also shown in \cite{lp} by means of a Walsh series
expansion of the local discrepancy and Parseval's identity. On
the other hand, quite recently, Kritzinger and
Pillichshammer~\cite{KriPi2015} obtained a general optimal
estimation for the $L_p$-discrepancy of $\widetilde{Y}_2$ (with
non-explicit constants) by means of strong analytic tools. This
is a new illustration of (at least) two complementary approaches in
discrepancy theory that are source a of emulation and
progress (see also Sections \ref{NKtransform} and
\ref{bdremaindersets}).

\begin{theorem}[Kritzinger and Pillichshammer]\label{Lp_sym_vdc}
For every $p \in [1,\infty)$ we have
 \[L_{p,N}(\widetilde{Y}_2)\ll_p \frac{\sqrt{\log{N}}}{N}.\]
\end{theorem}

We conjecture that an analogous result to Theorem \ref{Lp_sym_vdc} also holds for general bases $b$ (this is
work in progress). The proof of Theorem~\ref{Lp_sym_vdc} is based
on dyadic Haar functions and on Littlewood-Paley theory. We sketch
the basic ideas:

A {\it dyadic interval} of length $2^{-j}$, $j\in {\mathbb N}_0,$
in $[0,1)$ is an interval of the form
$$ 
I=I_{j,m}:=\left[\frac{m}{2^j},\frac{m+1}{2^j}\right) \ \ \mbox{for } \  m=0,1,\ldots,2^j-1.
$$

We also define $I_{-1,0}=[0,1)$. The left and right halves of
$I=I_{j,m}$ are the dyadic intervals $I^+ = I_{j,m}^+ =I_{j+1,2m}$
and $I^- = I_{j,m}^- =I_{j+1,2m+1}$, respectively. The {\it Haar
function} $h_I = h_{j,m}$ with support $I$ is the function on
$[0,1)$ which is  $+1$ on the left half of $I$, $-1$ on the right
half of $I$ and 0 outside of $I$. The $L_\infty$-normalized {\it
Haar system} consists of all Haar functions $h_{j,m}$ with
$j\in{\mathbb N}_0$ and $m=0,1,\ldots,2^j-1$ together with the
indicator function $h_{-1,0}$ of $[0,1)$. Normalized in
$L_2([0,1))$ we obtain the {\it orthonormal Haar basis} of
$L_2([0,1))$.

The {\it Haar coefficients} of a function $f \in L_p([0,1))$ are
defined as 
$$
\mu_{j,m}(f):=\langle f,h_{j,m} \rangle=\int_{0}^{1}
f(t)h_{j,m}(t)\rd t\ \ \ \mbox{for $j\in \NN_{-1}$ and $m
\in\mathbb{D}_j$},
$$ 
where we use the abbreviations
$\NN_{-1}:=\NN_{0}\cup\{-1\}$,
$\mathbb{D}_{j}:=\{0,1,\dots,2^j-1\}$ for $j\in\NN_0$, and
$\mathbb{D}_{-1}:=\{0\}$.

In order to prove Theorem~\ref{Lp_sym_vdc} we apply the {\it
Littlewood-Paley inequality} which involves the square function
$$ 
S(f) := \left( \sum_{j \in \NN_{-1}} \sum_{m \in \mathbb{D}_{j}} 2^{2\max\{0,j\}} \,
\langle f , h_{j,m} \rangle^2 \, {\mathbf 1}_{I_{j,m}}
\right)^{1/2} \ \ \ \mbox{ for $f\in L_p([0,1))$.}
$$

\begin{lemma}[Littlewood-Paley inequality]\label{lpi}
Let $p \in (1,\infty)$. Then there exist $c_p,C_p>0$ such that for
every $f\in L_p([0,1))$ we have
 $$c_p \| f \|_{L_p} \le \| S(f) \|_{L_p} \le C_p \| f \|_{L_p}.$$
\end{lemma}

So it remains to find good estimates for the Haar coefficients
$\langle f , h_{j,m} \rangle$ when $f$ is the local discrepancy of
the first $N$ terms of $\widetilde{Y}_2$. Let
$\mu_{j,m}^{N,\sym}:=\langle \Delta_{\widetilde{Y}_2,N},
h_{j,m}\rangle$. The following estimate was shown in
\cite{KriPi2015}:

\begin{lemma}\label{ex_HK_sym}
For $j \in \NN_0$ and $m \in \mathbb{D}_j$ we have
$$|\mu_{j,m}^{N,\sym}|\left\{
\begin{array}{ll}
\leq \frac{1}{2^{j-1}}\frac{1}{N} & \mbox{ if } j < \lceil \log_{2}{N}\rceil,\\ \\
= 2^{-2j-2} & \mbox{ if } j\geq \lceil \log_{2}{N}\rceil.
\end{array}\right.$$ Furthermore we also have
\begin{equation}\label{cuc_part}
|\mu_{-1,0}^{N,\sym}|= \begin{cases} 0 &\mbox{if } N=2M, \\
                       \left|\frac{1}{2N}-\frac{Y_2(M)}{N}\right|\leq \frac{1}{2N}\ & \mbox{if } N=2M+1. \end{cases}
\end{equation}
\end{lemma}

Equation \eqref{cuc_part} is the most crucial part in the proof of
Theorem~\ref{Lp_sym_vdc}. It shows how the symmetrization trick
keeps the Haar coefficient $\mu_{-1,0}^{N,\sym}$ small enough in
order to achieve the optimal $L_p$-discrepancy rate. For the
non-symmetrized van der Corput sequence the corresponding Haar
coefficient is at least of order $(\log N)/N$ for infinitely many
$N \in \NN$ which causes the large $L_p$-discrepancy of $Y_2$.

With the ingredients of Lemmas~\ref{lpi} and \ref{ex_HK_sym} one can derive the result in Theorem~\ref{Lp_sym_vdc} (cf. \cite{KriPi2015}):\\

\noindent{\it Proof of Theorem~\ref{Lp_sym_vdc}.} Using
Lemma~\ref{lpi} with $f=\Delta_{\widetilde{Y}_2,N}$ we have
\begin{eqnarray*}
L_{p,N}(\widetilde{Y}_2) & \ll_p &
\|S(\Delta_{\widetilde{Y}_2,N})\|_{L_p}
\\
& = & \left\| \left( \sum_{j \in \NN_{-1}}
\sum_{m \in \mathbb{D}_{j}} 2^{2\max\{0,j\}} \, (\mu_{j,m}^{N,\sym})^2 \, {\mathbf 1}_{I_{j,m}} \right)^{1/2} \right\|_{L_p}\\
& = & \left\|\sum_{j \in \NN_{-1}} 2^{2\max\{0,j\}}
\sum_{m \in \mathbb{D}_{j}} \, (\mu_{j,m}^{N,\sym})^2 \, {\mathbf 1}_{I_{j,m}} \right\|_{L_{p/2}}^{1/2}\\
& \le & \left( \sum_{j \in \NN_{-1}} 2^{2\max\{0,j\}} \left\|
\sum_{m \in \mathbb{D}_{j}} \, (\mu_{j,m}^{N,\sym})^2 \, {\mathbf
1}_{I_{j,m}} \right\|_{L_{p/2}} \right)^{1/2},
\end{eqnarray*}
where we used Minkowski's inequality for the $L_{p/2}$-norm. Now
Lemma~\ref{ex_HK_sym} gives
\begin{align*}
\sum_{j \in \NN_{-1}} 2^{2\max\{0,j\}} \left\| \sum_{m \in
\mathbb{D}_{j}}  \mu_{j,m}^2 \, {\mathbf 1}_{I_{j,m}}
\right\|_{L_{p/2}} \le & \frac{1}{4N^2} \|
\bsone_{\left[0,1\right)} \|_{L_{p/2}}
        +\frac{4}{N^2}\sum_{j=0}^{\lceil \log_{2}{N} \rceil -1}
\left\| \sum_{m \in \mathbb{D}_{j}}  \, {\mathbf 1}_{I_{j,m}} \right\|_{L_{p/2}} \\
       &+\frac{1}{16}\sum_{j=\lceil \log_{2}{N} \rceil}^{\infty}2^{-2j}
\left\| \sum_{m \in \mathbb{D}_{j}}   \, {\mathbf 1}_{I_{j,m}} \right\|_{L_{p/2}} \\
       \le & \frac{1}{4N^2}+\frac{4}{N^2}(\log_{2}{N}+1)+\frac{1}{12}\frac{1}{4^{\lceil \log_{2}{N} \rceil}} \ll \frac{\log{N}}{N^2},
\end{align*}
where we regarded the fact that $\sum_{m \in \mathbb{D}_{j}}  \,
{\mathbf 1}_{I_{j,m}}=1$ for a fixed $j \in \NN_0$. $\qed$

\medskip

A general approach to obtain optimal upper bounds for the
$L_2$-discrepancy of generalized van der Corput sequences in any
base $b$ comes from an inequality between $L_{2,N}(\widetilde{X})$
and the diaphony $F_N(X)$, see \cite[Theorem A]{progro87} and also
\cite[Section 4.2.3]{chafa}:
\begin{theorem}[Proinov and Grozdanov]\label{progro}
There exist two constants $a_1$ and $a_2$ such that
\[
(NL_{2,N}(\widetilde{X}))^2 \le \left ( \frac{1}{\pi} \left
\lfloor \frac{N}{2} \right \rfloor F_{\left \lfloor N/2 \right
\rfloor}(X)+a_1 \right )^2+a_2^2.
\]
\end{theorem}
\begin{corollary} \label{corprogro}
The asymptotic behavior of $L_{2,N}(\widetilde{Y}_b^\Sigma)$ is
related to the asymptotic behavior of $F_N(Y_b^\Sigma)$ by the
inequality
\[
t_2(\widetilde{Y}_b^\Sigma)=\limsup_{N \rightarrow \infty}
\frac{(N L_{2,N}(\widetilde{Y}_b^\Sigma))^2}{\log N} \le
\frac{1}{\pi} f(Y_b^\Sigma).
\]
\end{corollary}

The state of the art regarding the optimal constant $t_2=\inf_X(t_2(X))$
is the following (see the end of
Section \ref{secdiscgenvdC} for the definition). From Corollary
\ref{corprogro}, we have the constants $t_2=f / \pi^2$ resulting
from state of the art for the constants $f$ hence, for short,
$t_2<0.134$ in base 19 by Chaix and Faure in 1993 and $t_2<0.116$
in base 57 by Pausinger and Schmid in 2010. However, these values
remain larger than the bound obtained by Faure in 1990 for base 2
with an exact formula for $t_2(\widetilde{Y}_2)$, see Corollary
\ref{corL2symvdc2}, namely $t_2<0.103$. This means that 
$$
0.051599\ldots \le \inf_X \limsup_{N\To\infty} \frac{N L_{2,N} (X)}{\sqrt{\log N}}\le 0.32093\ldots,
$$
where the infimum is extended over all sequences $X$ in $[0,1)$. Note that the lower and upper bounds in the latter equation roughly 
differ by a factor of 6. 


\subsection{$(0,1)$-sequences and NUT sequences}\label{secdig01seq}

\paragraph{General remarks on $(0,1)$-sequences.}

We already observed in Section~\ref{sec_udt} that for every $m \in
\NN_0$ and $k \in \{0,1,\ldots,b^m-1\}$ exactly one of $b^m$
consecutive elements of the van der Corput sequence in base $b$
belongs to the so-called {\it $b$-adic elementary interval}
$[\frac{k}{b^m},\frac{k+1}{b^m})$. This fundamental property of
the van der Corput sequence is the basis of the following
definition which goes back to Niederreiter \cite{niesiam}; see
also \cite{DP10}.

\begin{definition}\rm\label{def01seq}
A sequence $(x_n)_{n\geq 0}$ in the unit interval $[0,1)$ is
called a $(0,1)$-{\it sequence in base $b$} if every  elementary
interval $E$ of the form $[\frac{k}{b^m},\frac{k+1}{b^m})$, with
$m\in \NN_0$ and $k \in \{0,1,\ldots,b^m-1\}$, contains exactly
one element of the point set $$\{x_n\, : \, lb^m \leq n \leq
(l+1)b^m-1\} \ \ \ \mbox{ for every $l \in \NN_0$.}$$
\end{definition}

Of course, for every $b \ge 2$ the van der Corput sequence $Y_b$
is a $(0,1)$-sequence in base~$b$.

To take into account important special cases, it is convenient to
work with a slightly modified definition. To this end define for
every $x\in [0,1]$ and for every $m\in \NN$ the $m$-{\it truncated
expansion} of $x$ by $[x]_{b,m}=\sum_{i=1}^m \xi_i b^{-i}$, where
$x=\sum_{i=1}^\infty \xi_i b^{-i}$ is the $b$-adic expansion of
$x$, with the possibility that $\xi_i=b-1$ for all but finitely
many $i$. Then a sequence $(x_n)_{n\geq 0}$ in the unit interval $[0,1)$ is
called a $(0,1)$-{\it sequence in base $b$ in the broad sense} if
every elementary interval $E$ of the form
$[\frac{k}{b^m},\frac{k+1}{b^m})$, with $m\in \NN_0$ and $k \in
\{0,1,\ldots,b^m-1\}$, contains exactly one element of the point
set 
$$
\{[x_n]_{b,m}\, : \, lb^m \leq n \leq (l+1)b^m-1\} \ \ \
\mbox{ for every $l \in \NN_0$.}
$$
Faure~\cite[Proposition~3.1]{Fau07} showed that every generalized
van der Corput sequence in base $b$ is a $(0,1)$-sequence in base
$b$ in the broad sense. The truncation is required for sequences
$Y_b^\Sigma$ where $\sigma_r(0)=b-1$ for all sufficiently large $r$. \\

It is known that among all $(0,1)$-sequences in base $b$ the
classical van der Corput sequence in base $b$ has the worst star
discrepancy. This was first shown by Pillichshammer~\cite{pil04}
for $b=2$ and later generalized by Kritzer~\cite{kri05} to general
$b$ and next by Faure \cite{Fau07} to $(0,1)$-sequences in the
broad sense (a further extension to $(t,1)$-sequences in the broad
sense is given in \cite{FL12}, see Section \ref{sec_tssequ} for
the meaning of $t$).

\begin{theorem}[Faure, Kritzer, and Pillichshammer]\label{thmKP01}
Let $X_b$ be an arbitrary $(0,1)$-sequence in base $b$ in the
broad sense and let $Y_b$ be the van der Corput sequence in base
$b$. Then we have
$$D_N^*(X_b) \le D_N^*(Y_b).$$ In particular, $$\limsup_{N
\rightarrow \infty} \sup_{X_b} \frac{N D_N^{\ast}(X_b)}{\log
N}=\frac{\alpha_{b,{\rm id}}}{\log b},$$ where $\alpha_{b,{\rm
id}}$ is defined as in \eqref{defalphabid} and where the supremum
is extended over all $(0,1)$-sequences $X_b$ in base $b$.
\end{theorem}

For the construction of a $(0,1)$-sequence in
base $b$ one frequently uses the digital method which is based on the
action of infinite $\NN \times \NN$ matrices, over a finite
commutative ring with identity and with $b$ elements, on the
$b$-adic digits of $n \in \NN_0$. For simplicity, here we only
present the case where $b$ is a prime power and where $R=\FF_b$,
the finite field with $b$ elements. For the general case ($b \ge
2$ and $R$ an arbitrary commutative ring with identity and with
$b$ elements) we refer to \cite{N87,niesiam}. We also require a
bijection $\varphi:\ZZ_b \rightarrow \FF_b$ with the property that
$\varphi(0)=0$.

\begin{definition}\label{Defdig01}\rm
Let $b$ be a prime power. Choose an $\NN \times \NN$ matrix
$C=(c_{r,k})_{r,k\geq 0}$ with entries in $\FF_b$ such that for
any $m\in \NN$ every left upper $m\times m$ sub-matrix of $C$ is
non-singular. Then the sequence $X_b^C=(x_n)_{n \ge 0}$ defined by
\begin{equation}\label{DvdC}
x_n=\sum_{r=0}^\infty \frac{\varphi^{-1}(\xi_{n,r})}{b^{r+1}},
\quad {\rm where } \quad \xi_{n,r}=\sum_{k=0}^\infty c_{r,k}
\varphi(n_k),
\end{equation}
and where the $n_k \in \ZZ_b$ are the $b$-adic digits of $n$, is
called a {\it digital $(0,1)$-sequence over $\FF_b$}. Note that
the second summation in \eqref{DvdC} is finite and performed in
$\FF_b$, but the first one can be infinite with the possibility
that $\varphi^{-1}(\xi_{n,r})=b-1$ for all but finitely many $r$.
\end{definition}

If $b$ is a prime number, then we identify $\FF_b$ with $\ZZ_b$
equipped with addition and multiplication modulo $b$. In this case
we choose for $\varphi$ the identity. We then obtain
the original van der Corput sequence $Y_b$ with the $\NN \times
\NN$ identity matrix $I$, i.e.,  we have $Y_b=X_b^I$.

It is well known that a digital $(0,1)$-sequence over $\FF_b$ is a
$(0,1)$-sequence in base $b$ if $C$ is such that for all $k \ge 0$
we have $c_{r,k}=0$ for $r$ sufficiently large (see
\cite{DP10,LP14,niesiam}), but in any case it is a
$(0,1)$-sequence in base $b$ in the broad sense (see
\cite{Fau07}). The truncation is required for sequences $X_b^C$ in
the case when the matrix $C$ yields digits
$\varphi^{-1}(x_{n,r})=b-1$ for all sufficiently large $r$.

\paragraph{NUT digital $(0,1)$-sequences.}
An important special case of digital $(0,1)$-sequences is obtained
when the associated matrix $C$ is a non-singular upper triangular
(NUT) matrix. In this case the first summation in \eqref{DvdC} is
finite as well. These sequences are called {\it NUT digital
$(0,1)$-sequences over $\FF_b$}. In \cite{Fau05} Faure proved
formulas in the vein of Theorems~\ref{thmF81D} and \ref{thmCF93}
for the different notions of discrepancies of NUT digital
$(0,1)$-sequences over $\FF_b$ (where $b$ is a prime number).

To recall these formulas we first need to introduce some
definitions to deal with NUT digital $(0,1)$-sequences. Let again
$b$ be a prime number, identify $\FF_b$ with $\ZZ_b$, and let $\varphi={\rm
id}$. The symbol $\uplus$ is used to denote the translated (or
shifted) permutation corresponding to a given permutation $\sigma
\in \Sy_b$ by an element $v\in {\mathbb F}_b$ in the following
sense:
$$
(\sigma \uplus v)(i):=\sigma (i)+v \pmod{b} \quad {\rm for \
all}\quad i \in {\mathbb F}_b,
$$
and for any $r\in \NN_0$ we introduce the quantity
$$
\theta_r(N):=\displaystyle\sum_{k=r+1}^\infty c_{r,k} N_k\pmod{b},
$$
where the $N_k$'s are the $b$-adic digits of $N$. Note that
$N_k=0$ for all $k\geq m$ if $0\leq N < b^m$, thus $\theta_r(N)=0$
for all $r\geq m-1$ in this case. Finally, for $r \ge 0$, we
define the permutation $\delta_r$ by $\delta_r(i):=c_{r,r} i
\pmod{b}$ for $i \in \FF_b$ and for future use we set
$\Delta:=(\delta_r)_{r \ge 0}$.
We can now formulate the following result.
\begin{theorem}[Faure]\label{thmfau05}
Let $X_b^C$ be a NUT digital $(0,1)$-sequence over $\FF_b$ with
prime $b$. Then, using the notation introduced in Section
\ref{secdiscgenvdC},  for any $N \in \NN$ we have
$$
D_N^+(X_b^C)=\frac{1}{N}\sum_{j=1}^\infty
\psi_b^{\delta_{j-1}\uplus\theta_{j-1}(N),+} \left ( \frac{N}{b^j}
\right ),
$$
$$
D^-_N(X_b^C)=\frac{1}{N}\sum_{j=1}^\infty
\psi_b^{\delta_{j-1}\uplus \theta_{j-1}(N),-}\left ( \frac{N}{
b^j} \right ),
$$
$$
D_N(X_b^C)=\frac{1}{N}\sum_{j=1}^\infty \psi_b^{\delta_{j-1}}
\left ( \frac{N}{b^j} \right ),
$$
$$
D_N^{\ast}(X_b^{C})=\max(D^+_N(X_b^{C}),D^-_N(X_b^{C})),
$$
\smallskip
$$
(NF_N(X_b^C))^2=\frac{4 \pi^2}{b^2}\sum_{j=1}^\infty
\chi_b^{\sigma_{j-1}} \left ( \frac{N}{b^j} \right),
$$
$$
(N L_{2,N}(X_b^C))^2=\\
\frac{1}{b}\sum_{j=1}^\infty
\phi_b^{\delta_{j-1}\uplus\theta_{j-1}(N)} \left (\frac{N}{b^j}
\right )+ \frac{1}{b^2} \sum_{i \neq j}
\varphi_b^{\delta_{i-1}\uplus\theta_{i-1}(N)}\left (\frac{N}{b^i}
\right ) \varphi_b^{\delta_{j-1}\uplus\theta_{j-1}(N)}\left
(\frac{N}{b^j} \right ).
$$
\end{theorem}

The formulas in Theorem \ref{thmfau05} for $D_N$ and $F_N$ show
that digital $(0,1)$-sequences generated by NUT matrices having
the same diagonal entries have the same extreme discrepancy and
the same diaphony. Therefore, for
 $D_N$ and $F_N$ generalized van der Corput sequences and NUT digital 
$(0,1)$-sequences have the same behavior, i.e., 
$D_N(X_b^C)=D_N(Y_b^\Delta)$ and $F_N(X_b^C)=F_N(Y_b^\Delta)$. 
This is a remarkable feature of NUT digital $(0,1)$-sequences that 
yields the asymptotic behavior of $D_N(X_b^C)$ and $F_N(X_b^C)$, 
directly using the results in Section \ref{secdiscgenvdC}.

For $D_N^+$, $D_N^-$, $D_N^*$, and $L_{2,N}$, the formulas are
similar to those for generalized van der Corput sequences, but the
situation is more complicated because of the quantity
$\theta_{j-1}(N)$ which depends on $N$, via the $b$-adic expansion
of $N-1$, and on the NUT generator matrix $C$, via its entries
strictly above the diagonal. Very few results are available,
for example Theorem~\ref{thmupone} below for $b=2$, or loose
bounds for $L_{2,N}$ in a very special case in base $b$ by
Grozdanov~\cite{Gro96}.

\paragraph{Generalization of NUT digital $(0,1)$-sequences.}
A further generalization of NUT digital $(0,1)$-sequences, in
relation with the search for best possible lower bounds for the
star discrepancy, has recently been given by Faure and
Pillichshammer in \cite{FauPi13}. The resulting family in {\it
arbitrary} integer bases---named {\it NUT $(0,1)$-sequences over
$\ZZ_b$}---include both, generalized van der Corput sequences
$Y_b^\Sigma$, and NUT digital $(0,1)$-sequences $X_b^C$. So far,
this is the most general family of van der Corput type sequences.

The idea is to put arbitrary permutations in place of the diagonal
entries of the NUT matrix $C$ defining a NUT digital
$(0,1)$-sequence:
\begin{definition}\label{defXSigC}\rm
For any integer $b\ge 2$, let $\Sigma=(\sigma_r)_{r \ge 0} \in
\Sy_b^\NN$ and let $C=(c_r^k)_{r \ge 0, k \ge r+1}$ be a strict
upper triangular matrix with entries in $\ZZ_b$.
Then, for $n\in \NN_0$, the $n$-th element of the sequence
$X_b^{\Sigma,C}=(x_n)_{n \ge 0}$ is defined by
\begin{equation}\label{GNUT}
x_n=\sum_{r=0}^\infty \frac{x_{n,r}}{b^{r+1}} \quad {\rm with}
\quad x_{n,r}=\sigma_r(n_r)+\sum_{k=r+1}^\infty c_r^kn_k \pmod{b},
\end{equation}
where the $n_k$'s are the $b$-adic digits of $n$. We call
$X_b^{\Sigma,C}$ a {\it NUT $(0,1)$-sequence over $\ZZ_b$}.
\end{definition}

If all entries of $C$ are zero, then we have
$X_b^{\Sigma,C}=Y_b^\Sigma$, the generalized van der Corput
sequences.  If $\sigma_r=\delta_r$ for all $r \ge 0$ with
$\delta_r$ as in Theorem \ref{thmfau05}, then we have
$X_b^{\Sigma,C}=X_b^{C'}$, the NUT digital $(0,1)$-sequence where
$C'=(c_r^k)_{r \ge 0, k \ge r}$ is obtained from $C$ by putting
{\em invertible} entries $c_r^r \in \ZZ_b$ on the diagonal.

Every NUT $(0,1)$-sequence over $\ZZ_b$ is a $(0,1)$-sequence in
the broad sense (the proof works like for generalized van der Corput
sequences in \cite{Fau07}).

It is quite remarkable that Theorem \ref{thmfau05} for NUT digital
$(0,1)$-sequences in prime base $b$ is still valid for this
generalization (the basic trick is to keep bijections of $\mathbb
{Z}_b$ ``on the diagonal'', which is the only place where
computing inverses is required, all other operations concerning
entries above the diagonal are performed in the ring $\mathbb
{Z}_b$):
\begin{theorem}[Faure and Pillichshammer]\label{thmGNUT}
Let $X_b^{\Sigma,C}$ be a NUT $(0,1)$-sequence over $\ZZ_b$ where
$b \ge 2$ is an integer. Then the results of Theorem
\ref{thmfau05} are still valid for $X_b^{\Sigma,C}$ with
permutations $\sigma_r$ replacing the permutations $\delta_r$.
\end{theorem}

The proof of Theorem \ref{thmGNUT} is rather complex and requires a lot of prerequisites,
but it follows closely the proofs of \cite[Theorems~1--4]{Fau05}.
See \cite{FauPi13} for a sketch where differences between both
proofs are pointed out.

Some applications of Theorem \ref{thmGNUT} will be given later. We
now turn to some special digital $(0,1)$-sequences over $\FF_2$.

\paragraph{Special digital $(0,1)$-sequences over $\FF_2$.}
In base $b=2$ one knows a NUT digital $(0,1)$-sequence over
$\FF_2$ with smaller star discrepancy than the van der Corput
sequence in base $2$. This sequence was given in
\cite[Theorem~3]{pil04}.

\begin{theorem}[Pillichshammer]\label{thmupone}
For the star discrepancy of the NUT digital $(0,1)$-sequence over
$\FF_2$ generated by the matrix $$C_1:=\left(
\begin{array}{lllll}
1 & 1 & 1 & 1 & \ldots \\
0 & 1 & 1 & 1 & \ldots \\
0 & 0 & 1 & 1 & \ldots \\
0 & 0 & 0 & 1 & \ldots \\
\vdots & \vdots & \vdots & & \ddots
\end{array}
\right)$$ we have $$0.2885\ldots=\frac{1}{5 \log 2}\le \limsup_{N
\rightarrow \infty}\frac{N D_N^*(X_2^{C_1})}{\log N} \le
\frac{5099}{22528 \log 2} =0.3265\ldots.$$
\end{theorem}

The next example from \cite{Fau07} shows that the van der Corput
sequence in base 2, is not the worst distributed with respect to
the extreme discrepancy $D_N$ among digital $(0,1)$-sequences in
base $2$ (but $Y_2$ is the worst distributed with respect to
$D_N$, $D_N^*$, $F_N$ and $L_{2,N}$ among {\it NUT digital
$(0,1)$-sequences in base $b$}, see Theorem \ref{Y2worst} below).

\begin{theorem}[Faure]\label{D2worst}
Let $C_2$ be the matrix
$$
C_2:= \def\vspace[#1]{\noalign{\vskip #1}}
\begin{pmatrix}
1 & 0 & 0 & 0 &\cdots\cr \vspace[3pt] 1 & 1 & 0 & 0 &\cdots\cr
\vspace[3pt] 1 & 0 & 1 & 0 &\cdots\cr \vspace[3pt] 1 & 0 & 0 & 1
&\cdots\cr \vspace[3pt] \vdots &\vdots &\vdots &\vdots &\ddots\cr
\end{pmatrix}.
$$
Then, for any $N\in \NN$ we have
$$2D_N(Y_2)-\frac{5}{2N} \leq D_N(X_2^{C_2}) \leq 2D_N(Y_2) \quad{\it and}$$
$$D^*_N(Y_2)-\frac{3}{2N} \leq D^*_N(X_2^{C_2}) \leq D^*_N(Y_2)=D_N(Y_2).$$
Moreover, the sequence $X_2^{C_2}$ is the worst distributed among
all $(0,1)$-sequences in base $2$ (in the broad sense) with
respect to the extreme discrepancy $D_N$.
\end{theorem}

The sequence $X_2^{C_2}$ was already considered by Larcher
and Pillichshammer \cite{lp} to show that there exists a symmetrized
version of a digital $(0,1)$-sequence in base 2, namely
$\widetilde{X}_2^{C_2}$, that does not have optimal order of
$L_2$-discrepancy. It should be remarked that
$X_2^{C_2}=\widetilde{Y}_2$. We see here that $X_2^{C_2}$ has
about the same star discrepancy as
 $Y_2$, but its extreme
discrepancy $D_N$ is about twice $D_N(Y_2)$ and it is the worst
among all $(0,1)$-sequences in base 2. Thus, this sequence appears
to be one with bad distribution properties (recall that the symmetrized sequence of
$Y_2$ has optimal order of $L_2$-discrepancy).

\medskip

The $L_2$-discrepancy of NUT digital $(0,1)$-sequences over
$\FF_2$ is studied in \cite{DLP05}.

\begin{theorem}[Drmota, Larcher, and Pillichshammer]\label{th1}
Let $X_2^C$ be a NUT digital $(0,1)$-sequence over $\FF_2$
generated by a NUT matrix $C$ and let $Y_2$ be the van der Corput
sequence in base 2. Then we have
\begin{eqnarray}\label{ineq1}
(N L_{2,N}(X_2^C))^2 \le (N L_{2,N}(Y_2))^2 \le \left(\frac{\log
N}{6 \log 2}\right)^2 + O(\log N)
\end{eqnarray}
and
\begin{eqnarray}\label{eq2}
\limsup_{N \rightarrow \infty}\sup_{X_2^C}\frac{N
L_{2,N}(X_2^C)}{\log N} = \frac{1}{6 \log 2}
\end{eqnarray}
where the supremum is extended over all NUT digital
$(0,1)$-sequences $X_2^C$ over $\FF_2$. Furthermore, for the
matrix $C_1$ defined in Theorem~\ref{thmupone} there exists a real $c>0$ such that 
$$L_{2,N}(X_2^{C_1}) \ge c \frac{\log N}{N}\ \ \ \mbox{ for infinitely many $N \in \NN$.}$$
\end{theorem}

For central limit theorems for the discrepancy of NUT digital
$(0,1)$-sequences over $\FF_2$ we refer to \cite[Theorems~5 and
6]{DLP05}.

\paragraph{The worst distributed sequence among NUT $(0,1)$-sequences $X_b^{\Sigma,C}$.}
Due to the theory of $\varphi_b$-functions for the study of van
der Corput sequences and their generalizations, it is possible to
bound the discrepancies and the diaphony of sequences obtained.

In \cite[Lemmas~1--3]{fau05}, the following bounds were obtained:
$$
\psi_b^\sigma \le \psi_b^\id,\ \ \ \vert \varphi_b^\sigma\vert
\le \varphi_b^\id \ \mbox{ and }\ \chi_b^\sigma \le \chi_b^\id,
$$
where $\sigma$ is an arbitrary permutation in $\Sy_b$ (the first
estimate actually stems from \cite[Section 5.5.4]{Fau1981}).
According to Theorems \ref{thmfau05} and \ref{thmGNUT} (where the
$L_2$-discrepancy is expressed by means of Koksma's formula
\eqref{FoKoks}), we obtain the following theorem (not published
until now):
\begin{theorem}[Faure and Pillichshammer]\label{Y2worst}
For any $N \in \NN$ we have
$$
D_N^*(X_b^{\Sigma,C})\leq D_N(X_b^{\Sigma,C})\leq
D_N^*(Y_b)=D_N(Y_b),
$$
$$L_{2,N}(X_b^{\Sigma,C})\leq L_{2,N}(Y_b) \quad
\mbox{ and } \quad F_N(X_b^{\Sigma,C})\leq F_N(Y_b).
$$
In other words, the van der Corput sequence in base $b$ is the
worst distributed among NUT $(0,1)$-sequences $X_b^{\Sigma,C}$,
especially among generalized van der Corput sequences $Y_b^\Sigma$
and among NUT digital $(0,1)$-sequences $X_b^C$, with respect to
$D_N^*, D_N, L_{2,N}$ and $F_N$.
\end{theorem}

\begin{remark}\rm
In \cite[Theorem 2 and the subsequent Remark]{fau05} Theorem \ref{Y2worst} was
already stated for $Y_b^\Sigma$ and $X_b^C$. The first part of
Theorem \ref{th1} for $L_{2,N}(X_2^C)$ was obtained independently
of \cite[Theorem 2]{fau05} by using Walsh series analysis. For
$D_N^*$, the result of Theorem \ref{Y2worst} generalizes the bound
from Theorem \ref{thmKP01}. This is not true for $D_N$, according
to Theorem \ref{D2worst}. The latter theorem, for base 2, has an
analog in base $b$ but with a non-digital $(0,1)$-sequence in the
broad sense, see \cite[Theorem~5.3]{Fau07}.
\end{remark}

\paragraph{Some applications of Theorems \ref{thmfau05} and \ref{thmGNUT}.}
We present two applications, the first one concerning best possible lower bounds 
for the star discrepancy and the second one concerning linear digit 
scramblings, which are also relevant for applications in quasi-Monte Carlo rules, 
and scramblings of (digital) NUT $(0,1)$-sequences.

\medskip

\begin{itemize}
\item Best possible lower bounds for the star discrepancy are
dealt with in \cite[Section 5]{FauPi13} for the (currently)
largest family of van der Corput type sequences which was
introduced in Definition~\ref{defXSigC}. We only give an excerpt
of the results.
\begin{theorem}[Faure and Pillichshammer]\label{BestLowBd}
Let $\mathcal{C}_{SUT}$ be the class of all strict upper
triangular $\NN \times \NN$ matrices, let $\sigma \in \Sy_b
\mbox{ such that } D_N^*(Y_b^\sigma)=D_N(Y_b^\sigma)$ and let $\Sigma=(\sigma,\sigma,\ldots)$. Then (see
Theorem \ref{asymptD*} for notation) we have
$$
\inf_{\substack{\Sigma \in \{\sigma,\tau_b\circ\sigma\}^{\NN} \\ C
\in \mathcal{C}_{SUT}}} d^*(X_b^{\Sigma,C}) =
\frac{\alpha_{b,\sigma}}{2 \log b}\cdot
$$
\end{theorem}

A proof can be found in \cite[Corollary~1]{FauPi13}. Besides the
identity $\id$, it is not difficult to find permutations
satisfying the condition $D_N^*(Y_b^\sigma)=D_N(Y_b^\sigma)$ from
Theorem~\ref{BestLowBd}. This can be done, for example, by using
intricate permutations (see for instance \cite[Section 2.3]{Fau92}
for the definition of intrication of two permutations). Moreover,
a systematic computer search performed by Pausinger returned 26, 58, 340, and 1496 such
permutations in bases 6, 7, 8, and 9, respectively. Further, the best sequence currently known
with respect to the star discrepancy found by
Ostromoukhov (see the end of Section \ref{secdiscgenvdC}), namely
$Y_{60}^{{\Sigma^{\sigma_0}_\mathcal{A}}}$, satisfies this
condition. Hence in base $b=60$ we have
$$
\inf_{\substack{\Sigma \in
\{\sigma_0,\tau_{60}\circ\sigma_0\}^{\NN} \\ C \in
\mathcal{C}_{SUT}}} d^*(X_{60}^{\Sigma,C}) = 0.2222\ldots.
$$
The case of identity in Theorem~\ref{BestLowBd} is of special
interest because $\alpha_{b,\id}$ is explicitly known for any
integer $b \ge 2$, see Equations \eqref{defalphabid} and
\eqref{d*alter}.

\begin{corollary}\label{infbX}
With the notation of Theorem \ref{BestLowBd}, we obtain
$$
\inf_{b \ge 2} \inf_{\substack{\Sigma \in \{\id,\tau_b\}^{\NN} \\
C \in \mathcal{C}_{SUT}}} d^* (X_b^{\Sigma,C}) =
d^*(Y_3^{\Sigma^\id_\mathcal{A}})=\frac{1}{4 \log 3}=0.2275\ldots.
$$
\end{corollary}

This result can be seen as an analog for van der Corput sequences
and NUT $(0,1)$-sequences of the best lower bound for
$(n\alpha)$-sequences from Dupain and S\'{o}s~\cite{DS84}
mentioned on page~\pageref{dupSos}. So, compared to the best $(n
\alpha)$-sequence, we obtain a smaller value for $d^*$ with the
generalized van der Corput sequence in base 3 obtained by alternating 
the permutations $\mathrm{id}$ and $\tau_3$.

Finally we show that the constant $1/(6 \log 2)$, which is best
possible for the star discrepancy of any generalized van der
Corput sequence according to \cite[Corollary~4]{KLP07} (cf.
Remark~\ref{re5bejKrLaPi}), is even best possible for the star
discrepancy of any  generalized NUT digital sequence in base 2.
The following result is \cite[Corollary 3]{FauPi13}:

\begin{corollary}\label{dsdNUTseq}
Let $X_2^{\Sigma, C}$ be a generalized NUT digital sequence in
base 2 generated by the NUT matrix $C$, with permutations
$\Sigma \in \{\id,\tau_2\}^{\NN}$. Then, we have
$$
\inf_{\substack{\Sigma \in \{\id,\tau_2\}^{\NN} \\ C \in
\mathcal{C}_{NUT}}} d^*(X_2^{\Sigma,C})=\frac{1}{6 \log
2}=0.2404\ldots,
$$
where $\mathcal{C}_{NUT}$ is the class of all NUT $\NN \times \NN$
matrices over $\FF_b$.
\end{corollary}

\item The notion of linear digit scramblings and scramblings of
(digital) NUT $(0,1)$-sequences
 stems from Matou\v{s}ek \cite{Mat98} in his attempt to classify the very 
general scramblings introduced earlier by Owen. A \emph{linear digit scrambling} 
is a permutation of the set ${\mathbb Z_b}$  of the form
$$
\pi(k)=hk+g \pmod {b},
$$
where $h\neq 0$ and $g$ are given in ${\mathbb Z_b}$ with $b$
${\emph prime}$. The definition also works for any base $b$,
provided that the multiplication by $h$ remains a bijection. If
$g=0$, we obtain the so-called \emph {multipliers} (or
multiplicative factors)  $h$ of preceding papers on NUT digital
$(0,1)$-sequences (see \cite{fau05,Fau08}). The additive factor
$g$ is a translation also called \emph{digital shift} (see
\cite{KrLaPi}).

\medskip

It is quite remarkable that swapping permutations $\tau_b$ are
linear digit scramblings for any base, since
$\tau_b(k)=b-1-k=(b-1)k+b-1 \pmod{b}$.

\medskip
In the following we give some examples showing the effect of
linear digit scramblings with regard to the asymptotic behavior of
the star discrepancy:
\begin{itemize}
\item[--] In base 2 $\tau_2$ and the digital shift $k\mapsto k+1$
over $\mathbb{F}_2$ are the same permutations. \item[--] In base
3, all 6 permutations are linear digit scramblings, namely
$\id$, $\id+1$, $\id+2$, $2\id$, $2\id+1$ and $2\id+2=\tau_3$. All
of them have the same asymptotic behavior with respect to $D_N^*$
and there is no improvement with regard to
$d^*(Y_3^{{\Sigma^I_\mathcal{A}}})$, see Theorem \ref{asymptD*}
and Equation~\eqref{d*alter}. \item[--] In base 4 there is no
improvement  with linear digit scramblings, but the permutation
$\sigma=3(12)+1$ (where $(12)$ denotes the transposition which
exchanges 1 and 2) produces the sequence $Y_4^\sigma$  with
$d^*(Y_4^\sigma)=1/(5 \log 2)$, which is the same as for the
sequence associated with the two-dimensional Halton-Zaremba point
set from \cite{HaZa}. Moreover, $Y_4^{\Sigma^\id_\mathcal{A}}$ in
\eqref{d*alter} has also the same asymptotic constant $1/(5 \log
2)$, which means that the single permutation $\sigma$ is also
equivalent to the sequence of permutations
$\Sigma^\id_\mathcal{A}$ in base 4. \item[--] In base 233, which
is our best prime base for $D_N^*$ up to base 1301 (see
\cite{fau05}), the phenomenon becomes really apparent: with the
two linear digit scramblings $\rho(k)=89k$ and $\pi(k)=89k+44$ we
have
$$\limsup_{N\rightarrow\infty}\frac{ND_N^*(Y_{233}^{\Sigma^\rho_\mathcal{A}})}{\log N}<\frac{469+365}{466\log 233}=0.328\ldots,$$
$$\limsup_{N\rightarrow\infty} \frac{ND_N^*(Y_{233}^\pi)}{\log N}<\frac{368}{233\log 233}=0.289\ldots,$$
$$\limsup_{N\rightarrow\infty} \frac{ND_N^*(Y_{233}^{\Sigma^\pi_\mathcal{A}})}{\log N}<\frac{368}{233\log 233}=0.289\ldots, \ \ {\rm  while \ }$$
$$\limsup_{N\rightarrow\infty} \frac{ND_N^*(Y_{233}^{\Sigma^{\id}_\mathcal{A}})}{\log N}=\frac{232}{8\log 233}=5.32\ldots.$$
Here, we observe even more: adding the digital shift 44 still
improves the behavior of the sequence even without using the
sophisticated sequence $\Sigma_\mathcal{A}$. This fact is in
accordance  with a remark of Matou\v sek in \cite[p. 540]{Mat98}:
``Introducing additive terms makes the situation much simpler and
more regular.'' Such computations with very large bases have been
performed for some scrambling techniques applied to Halton sequences and
Faure sequences (which are examples of $(0,s)$-sequences), see Section
\ref{sec_mult}.
\end{itemize}

\end{itemize}


\subsection{Van der Corput sequences with respect to different digital expansions}\label{cantor}

Van der Corput sequences have not only been introduced with
respect to the classical $b$-adic digit expansion. Here we present
two further prominent examples.

\paragraph{Cantor expansions.} Van der Corput sequences with
respect to Cantor expansions have been studied by Faure
\cite{Fau1981} and Chaix and Faure \cite{chafa}. We call
$B=(b_j)_{j \ge 0}$ with $b_0=1$ and integers $b_i \ge 2$ for all
$i \ge 2$ a {\it Cantor base} and we set $B_k:=b_0\cdots b_k$ for
$k \in \NN_0$.

The special case of ordinary $b$-adic expansions, $b \ge 2$ an
integer, is contained if we choose $b_1=b_2=\cdots =b$ and hence
$B_k=b^k$. The main difference between $B$-adic and ordinary
$b$-adic expansions is that in the general case the $i$-th digit
can take values in $\{0,\ldots,b_{i+1}-1\}$, which may vary for
each $i$ and even become arbitrarily large. Each integer $n$
possesses a unique finite representation
\[ n =  n_0 b_0+n_1 b_0b_1+n_2b_0b_1b_2+ \cdots=\sum_{i\ge 0}n_iB_i,\]
with $n_i\in\{0,\ldots,b_{i+1}-1\}$ for $i \in \NN_0$. We will
call this the $B$-{\it adic expansion} or the {\it Cantor
expansion} of $n$. Additionally, each real number $x\in[0,1)$ has
a representation of the form \[ x =
\frac{x_0}{b_0b_1}+\frac{x_1}{b_0b_1b_2}+\frac{x_2}{b_0b_1b_2b_3}+\dots
=
 \sum_{i\ge 0} \frac{x_i}{
B_{i+1}},\] with $x_i\in\{0,\ldots,b_{i+1}-1\}$ for $i \in \NN_0$.

\begin{definition}\label{vdc2Qant}\rm
Let $B$ be a Cantor base as described above and let
$\Sigma=(\sigma_r)_{r\geq 0}$ be a sequence of corresponding
permutations $\sigma_r$ of $\{0,1,\ldots,b_{r+1}-1\}$ for $r \ge
0$.
\begin{itemize}
\item The {\it $B$-adic radical inverse function with respect to
$\Sigma$} is defined as $Y_B^{\Sigma}: \NN_0 \rightarrow [0,1)$,
$$Y_B^{\Sigma}(n)=\frac{\sigma_0(n_0)}{B_1}+\frac{\sigma_1(n_1)}{B_2}+\frac{\sigma_2(n_2)}{B_3}+\cdots,$$
for $n \in \NN_0$ with Cantor expansion $n=n_0B_0+n_1 B_1+n_2
B_2+\cdots$, where $n_i \in \{0,1,\ldots,b_{i+1}-1\}$. \item The
{\it generalized van der Corput sequence in Cantor base $B$
associated with $\Sigma$} is defined as $Y_B^{\Sigma}:=(y_n)_{n
\ge 0}$ with $y_n=Y_B^{\Sigma}(n)$.
\end{itemize}
\end{definition}


Theorems \ref{thmF81D} and \ref{thmCF93} extend in a natural way
to Cantor bases, see \cite[Section 3.4.2, Th\'eor\`eme 1
bis]{Fau1981} and \cite[Th\'eor\`emes 4.1 to 4.3]{chafa} for their
respective statements. In the following, we first give a general
asymptotic estimate related to the bases $b_j$ of the Cantor
base $B$, and then we deal with a special construction
 associated with two pairs $(b,\sigma)$ and $(c,\rho)$ which is useful 
for building new low discrepancy sequences from other ones already known.

\begin{theorem}[Chaix and Faure]
For every Cantor base $B$ for which $\sum_{j=1}^n b_j \ll n$, and
every sequence $\Sigma$ of corresponding permutations we have
$$ND_N^{\ast}(Y_B^{\Sigma})\ll \log N.$$ If the sequence
$\Sigma$ consists only of the identities, then the condition
$\sum_{j=1}^n b_j \ll n$ is even necessary.
\end{theorem}

The sufficiency of the condition on $B$ is
\cite[Th\'{e}or\`{e}me~4.5]{chafa} whereas the necessity in case
of identities is \cite[Th\'{e}or\`{e}me~4.7]{chafa}.

\medskip

We now define the notion of intrication of two permutations in two
(different) bases introduced in \cite[Section 3.4.3]{Fau1981}:
\begin{definition}\rm \label{defIntr}
The {\it intrication} of two pairs $(b,\sigma)$ and $(c,\rho)$ is
the pair $(bc, \sigma\cdot\rho)$ defined by
$\sigma\cdot\rho(l)=c\sigma(h)+\rho(k)$ with $l=bk+h,\ 0\leq h <b$
and $0\leq k <c$ (i.e., we consider Euclidean division of $l$ by $b$, $0\leq l<bc$).
\end{definition}

\begin{lemma}\label{PropIntr}
\begin{enumerate}
\item The function $\psi_{bc}^{\sigma\cdot\rho,+}$ satifies the
relation
$$\psi_{bc}^{\sigma\cdot\rho,+}(x)=\psi_b^{\sigma,+}(cx)+\psi_c^{\rho,+}(x) \mbox{ for any } x\in \RR $$
and an analogous relation is also valid for
$\psi_{bc}^{\sigma\cdot\rho,-}$ and
$\psi_{bc}^{\sigma\cdot\rho}$.

\item Let $Y_B^\Sigma$ be the sequence defined by $B=(b_j)_{j \ge
0}$ and $\Sigma=(\sigma_j)_{j \ge 0}$  with $b_{j+1}=b$ and
$\sigma_{j}=\sigma$ if $j$ is even, and $b_{j+1}=c$ and
$\sigma_j=\rho$ if $j$ is odd. Then
$Y_{bc}^{\sigma\cdot\rho}=Y_B^\Sigma$.
\end{enumerate}
\end{lemma}

If the Cantor base $B$ and the sequence of permutations $\Sigma$
are periodic with the same period $J$, we get
$Y_B^\Sigma=Y_{B_J}^{\Sigma_J}$ with $B_J=\prod_{j=0}^J b_j$ and
$\Sigma_J=\prod_{j=0}^{J-1} \sigma_j$ (with the product $\cdot$ of
Definition \ref{defIntr} which is associative). In particular, if
$B=(1,b,b,b,\ldots)$ and $\Sigma=(\sigma, \sigma, \ldots)$ are
constant, we obtain $Y_b^\sigma=Y_{b^n}^{\sigma^n}$ for any $n \in
\NN$.

\medskip

Definition \ref{defIntr} and Lemma \ref{PropIntr} allow to build
by a simple algorithm a sequence of permutations $\sigma_b$  and
to show the following theorem, hence positively answering a
question set in Remark \ref{alphabid} (see \cite[Theorem 1.1 and
Section 3.1]{Fau92}):
\begin{theorem}[Faure]
For any base $b \ge 2$, there exists a permutation $\sigma_b \in
\Sy_b$ such that
$$
d(Y_b^{\sigma_b})=\limsup_{N \rightarrow \infty} \frac{N
D_N(Y_b^{\sigma_b})}{\log N} \le \frac{1}{\log 2}\cdot
$$
\end{theorem}
\paragraph{$\beta$-adic expansions.} Van der Corput sequences with respect to $\beta$-expansions were introduced and/or studied by Barat and
Grabner~\cite{BarGrab},  Ninomiya~\cite{nino98a,nino98b} and
Steiner~\cite{steiner2006}. There are several equivalent ways
to introduce such van der Corput sequences. Here we follow the
presentation in \cite{steiner2006}.

Given a real number $\beta>1$, the expansion of 1 with respect to
$\beta$ is the sequence of non-negative integers $(a_j)_{j \ge 1}$
satisfying
\begin{equation}\label{expof1}
1=0.a_1a_2\ldots=\frac{a_1}{\beta}+\frac{a_2}{\beta^2}+\cdots \ \
\ \mbox{ with }\ \ \ a_ja_{j+1}\ldots \prec a_1a_2\ldots \ \mbox{
for all $j \ge 2$.}
\end{equation}
The symbol $\prec$ denotes the lexicographic order of words. For
$x \in [0,1)$ the {\it $\beta$-expansion of $x$} is given by
$$x=0.\epsilon_1\epsilon_2\ldots=\frac{\epsilon_1}{\beta}+\frac{\epsilon_2}{\beta^2}+\cdots
\ \ \ \mbox{ with }\ \ \ \epsilon_j\epsilon_{j+1}\ldots \prec
a_1a_2\ldots \ \mbox{ for all $j \ge 1$.}$$

\begin{definition}\rm
The elements of the {\it $\beta$-adic van der Corput sequence}
$Y_{\beta}=(y_n)_{n \ge 0}$ are the real numbers in $[0,1)$ with
finite $\beta$-expansion, i.e.,
\begin{eqnarray*}
\{y_n \ : \ n \in \NN_0\} & = & \{0.\epsilon_1\epsilon_2\ldots \ : \   \epsilon_j\epsilon_{j+1}\ldots \prec a_1a_2\ldots \mbox{ for all } j \ge 1 \mbox{ and } \\
& & \hspace{2.4cm}
\epsilon_{\ell}\epsilon_{\ell+1}\epsilon_{\ell+2}\ldots=000\ldots
\mbox{ for some } \ell \ge 1\},
\end{eqnarray*}
ordered lexicographically with respect to the (inversed) word
$\ldots \epsilon_2\epsilon_1$, i.e., for
$y_n=0.\epsilon_1\epsilon_2\ldots$ and
$y_m=0.\epsilon_1'\epsilon_2'\ldots$ we have $n < m$ if there
exists some $k \ge 1$ such that $\epsilon_k < \epsilon_k'$ and
$\epsilon_j=\epsilon_j'$ for all $j \ge k$.
\end{definition}

For example, if $\beta_0=\frac{1+\sqrt{5}}{2}$, then the expansion
of 1 is $1=0.11000\ldots$. In this case the initial elements of
the $\beta_0$-adic van der Corput sequence are $$y_0=0, y_1=0.1,
y_2=0.01, y_3=0.001, y_4=0.101, y_5=0.0001, y_6=0.1001,\ldots.$$

This special instance of a $\beta$-adic van der Corput sequence
can also be introduced with the help of the Zeckendorf expansion:
let $(F_n)_{n \ge -1}$ be the sequence of Fibonacci numbers, i.e.,
$F_{-1}=F_0=1$ and $F_n=F_{n-1}+F_{n-2}$ for $n \in \NN$. Every
positive integer $n$ can be uniquely written in the form
$$n=\sum_{j=0}^m a_j F_j\ \ \ \mbox{ where $a_j \in \{0,1\}$ and
$a_j a_{j+1}=0$.}$$ This expansion is called the {\it Zeckendorf
expansion} of $n$. Then the elements of $Y_{\beta_0}$ with
$\beta_0=\frac{1+\sqrt{5}}{2}$ are given by 
$$y_n=\sum_{i=1}^m
\frac{1}{\beta_0^{r_i+1}}\ \ \ \mbox{ if } n=\sum_{i=1}^m F_{r_i}
\mbox{ with $0 \le r_1 < r_2 < \cdots < r_m$ and $r_{i+1} >
r_i+1$.}
$$ 
This definition can also be found in \cite{BPV2010}.

Ninomiya \cite{nino98a} showed that under certain conditions on
$\beta$ the sequence $Y_{\beta}$ has optimal order of the star
discrepancy: A {\it Pisot number} is an algebraic integer for
which all algebraic conjugates have modulus less than 1.
Bertrand~\cite{bertr1977} and K. Schmidt~\cite{schmidt1980} proved
that all Pisot numbers are Parry numbers which means that the
expansion \eqref{expof1} of $1$ is finite or eventually periodic,
i.e., $a_1a_2\ldots=a_1\ldots a_d00\ldots$ or
$a_1a_2\ldots=a_1\ldots a_{d-p}\overline{a_{d-p+1}\ldots a_d}$,
respectively. In this case $\beta$ is the dominant root of the
so-called $\beta$-polynomial $x^d-a_1 x^{d-1}-\cdots-a_d$ (with
$a_d>0$) or $x^d-a_1 x^{d-1}-\cdots-a_d-(x^{d-p}-a_1
x^{d-p-1}-\cdots -a_{d-p})$ (where $p$ is assumed to be minimal),
respectively.

The following theorem is \cite[Theorem~3.1]{nino98a}.

\begin{theorem}[Ninomiya]
If $\beta$ is a Pisot number with irreducible $\beta$-polynomial,
then we have $D_N^{\ast}(Y_{\beta}) \ll (\log N)/N$.
\end{theorem}

Bounded remainder sets for the $\beta$-adic van der Corput
sequence were studied in \cite{steiner2006}. Van der 
Corput sequences with respect to abstract numeration
systems were introduced and studied by
Steiner~\cite{steiner2009}. Under some assumptions, these
sequences also have optimal order of the star discrepancy with
respect to W.M. Schmidt's lower bound \eqref{lowschmid}.


\subsection{Polynomial van der Corput sequences}\label{secpolyvdc}

A polynomial version of the radical inverse function was
introduced by Tezuka~\cite{tez1993}. Let $b$ be a prime power
and let $\FF_b$ be the finite field of order $b$. Furthermore let
$\FF_b((x^{-1}))$ be the field of formal Laurent series
$$g=\sum_{k=w}^{\infty} a_k x^{-k}\ \ \ \mbox{ with } a_k \in
\FF_b \ \mbox{ and } \ w \in \ZZ \ \mbox{ with }\ a_w \not=0.$$
For $g \in \FF_b((x^{-1}))$ let $\lfloor g\rfloor$ be the
polynomial part of $g$.

First we introduce a polynomial version of the radical inverse
function. Choose a non-constant polynomial $p \in \FF_b[x]$ and a
bijection $\psi:\ZZ_b \rightarrow \FF_b$ with $\psi(0)=0$. Let
$q\in \FF_b[x]$ and expand this polynomial in base $p(x)$ rather
than in $x$. I.e., find $r_0,r_1,\ldots,r_h \in \FF_b[x]$ such
that $$q(x)=r_{h}(x) p(x)^h+\cdots +r_{1}(x) p(x)+r_{0}(x).$$ Here
$h=\lfloor \deg(q)/\deg(p)\rfloor$ and $r_{j}=\left\lfloor
q/p^j\right\rfloor \pmod{p}$ for $j=0,\ldots,h$. Now the radical
inverse $Y_{b,p}:\FF_b[x] \rightarrow \FF_b((x^{-1}))$ is defined
as
$$Y_{b,p}(q)=\frac{r_0(x)}{p(x)}+\cdots+\frac{r_h(x)}{p(x)^{h+1}}.$$

Let $\mu_b:\FF_b((x^{-1})) \rightarrow [0,1)$ be given as
$\mu_b(\sum_{k=w}^{\infty} a_k x^{-k})=\sum_{k=\max(1,w)}^{\infty}
\psi^{-1}(a_k) b^{-k}$. We identify $n \in \NN_0$ with $b$-adic
expansion $n=n_0+n_1 b+\cdots +n_{m-1} b^{m-1}$ with the
polynomial $n(x)=\psi(n_0) + \psi(n_1) x+ \cdots +\psi(n_{m-1})
x^{m-1} \in \FF_b[x]$ and vice versa.

\begin{definition}\label{def_poly_vdc}\rm
The {\it polynomial van der Corput sequence in base $p$ over
$\FF_b$} is given by $Y_{b,p}^{{\rm poly}}=(y_n)_{n \ge 0}$ with
$y_n=\mu_b(Y_{b,p}(n))$.
\end{definition}

It should be mentioned that this polynomial version of the van der
Corput sequence over $\FF_b$ is a special univariate instance of a
so-called {\it generalized Niederreiter sequence over $\FF_b$}
(see~\cite[Proof of Theorem~1]{tez1993}). Generalized Niederreiter
sequences are a very powerful construction of low discrepancy
sequences and frequently used as sample nodes in modern
quasi-Monte Carlo rules. For their definition we refer to
\cite[Section~8.1]{DP10}; see also Section~\ref{sec_tssequ} of
this article.

If we choose in Definition~\ref{def_poly_vdc} the polynomial $p$
as $p(x)=x+a$ with $a \in \FF_b$, then $Y_{b,p}^{{\rm poly}}$ is a
digital $(0,1)$-sequence over $\FF_b$, see \cite{tez1993}.

If $b$ is a prime number, we have $\FF_b=\ZZ_b$ and we can choose
$\psi=\id$. If we take now $p(x)=x$, then we obtain that
$Y_{b,p}^{{\rm poly}}=Y_b$, the classical van der Corput sequence
in base~$b$.


\section{Multi-dimensional generalizations of the van der Corput sequence}\label{sec_mult}

The final section of this survey is dedicated to multi-dimensional
variants and generalizations of van der Corput sequences. These
comprise generalized or scrambled Halton sequences as well as
Niederreiter's $(t,s)$-sequences which are the bases of modern
quasi-Monte Carlo rules for numerical integration, even in very
high dimensions (dimensions in the hundreds are not uncommon, for
example in applications from financial mathematics).

\subsection{Generalized two-dimensional Hammersley point sets}\label{sec_ham}
Strictly speaking, these point sets are not multi-dimensional
generalizations of the van der Corput sequence. However, they are
intimately related to generalized van der Corput sequences and
there is abundant literature devoted to these point sets.

\begin{definition}[generalized Hammersley point set]\label{defgenHam}\rm
Let $b \ge 2$ be an integer, let $Y_b^{\Sigma}$ be a generalized
van der Corput sequence in base $b$ as defined in
Definition~\ref{vdc2} and let $m \in \NN_0$. Then the {\it
generalized two-dimensional Hammersley point set in base $b$},
consisting of $b^m$ points in $[0,1)^2$ and associated to $\Sigma$,
is defined by
$$
\cH_{b,m}^\Sigma:=\left \{\left (Y_b^{\Sigma}(n), \frac{n}{b^m}
\right )\ : \ n \in \{0,1,\ldots, b^m-1\} \right \}.
$$
\end{definition}

In order to match the traditional definition of Hammersley
point sets, which consists of $m$-bit numbers (see page~\pageref{nbit}),
the infinite sequence of permutations $\Sigma$ is restricted to
permutations such that $\sigma_r(0)=0$ for all $r\ge m$, for
instance $\Sigma=(\sigma_0,\ldots,\sigma_{m-1},\id ,\id, \id,
\ldots)$. Hence, the behavior of $\cH_{b,m}^\Sigma$ only depends
on the finite sequence of $m$ permutations
$\bssigma=(\sigma_0,\ldots,\sigma_{m-1})$. From now on, we only
consider such generalized Hammersley point sets denoted by
$\cH_{b,m}^{\bssigma}$. If $\bssigma=(\sigma,\ldots,\sigma)$, then
we simply write $\cH_{b,m}^{\sigma}$.

If we choose in the above definition $\sigma_j=\id$ for all
permutations, then we obtain the classical two-dimensional
Hammersley point set in base $b$ denoted by
$\cH_{b,m}:=\cH_{b,m}^{\id}$. Two pictures of classical
two-dimensional Hammersley point sets are shown in
Figure~\ref{f18} (notice that Figures~\ref{fboard} and \ref{fboard2} in Section \ref{sec_Faure_meth} are also representations of the two-dimensional Hammersley point sets $\cH_{3,3}$ and $\cH_{3,2}$).

\begin{figure}[htp]
     \centering
     {\includegraphics[width=50mm]{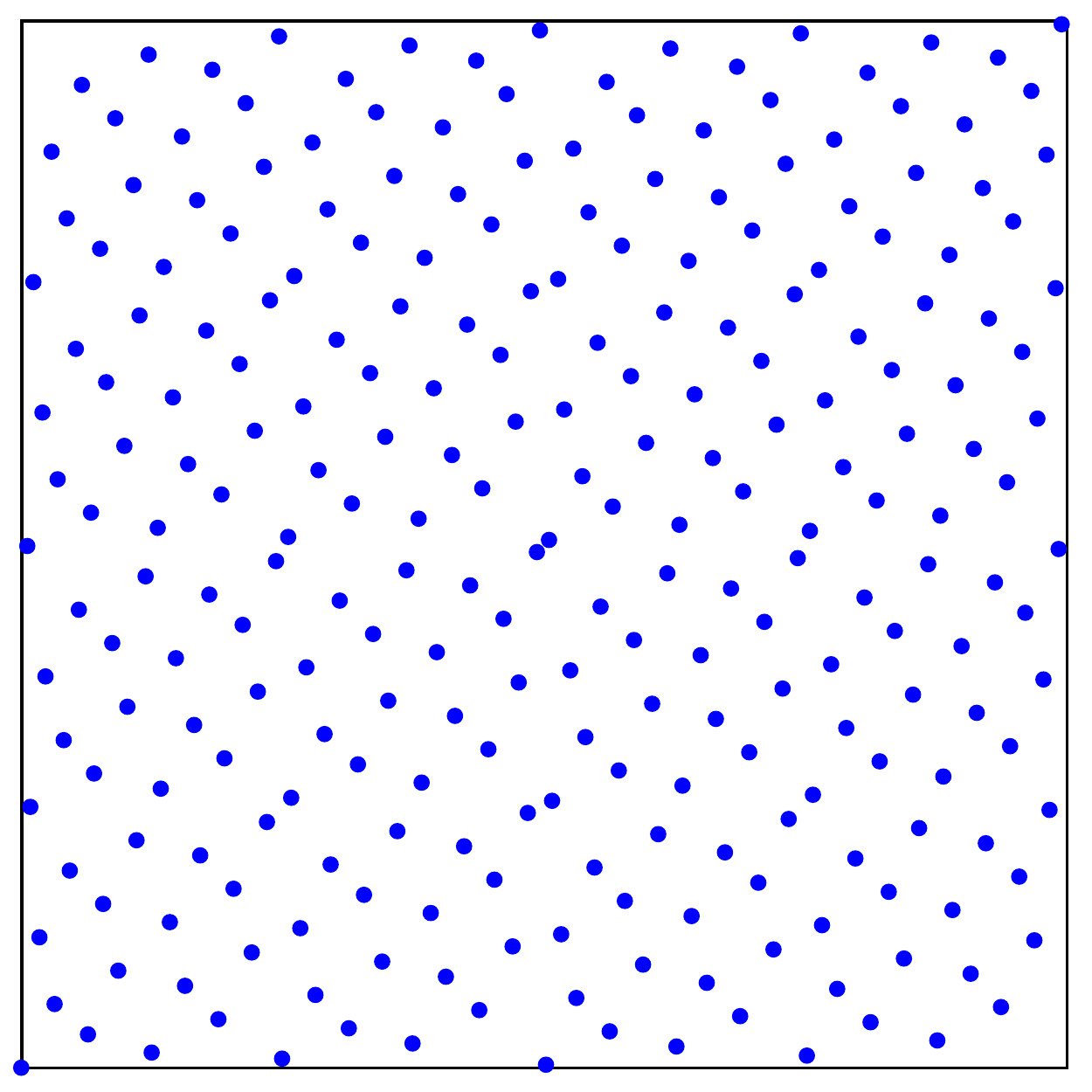}}\hspace{.2in}
     {\includegraphics[width=50mm]{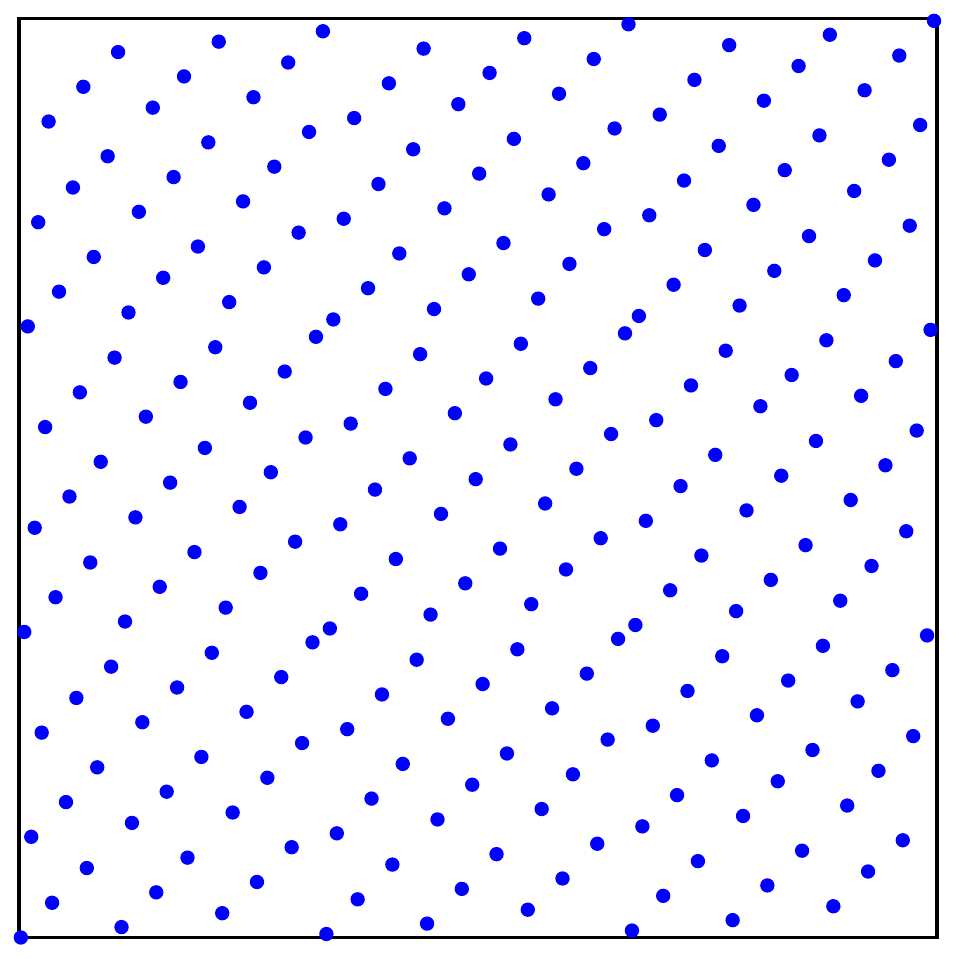}}
     \caption{Two-dimensional Hammersley point sets: $\cH_{2,8}$ in base $b=2$ with $2^8=256$ 
     points (left) and $\cH_{3,5}$ in base $b=3$ with $3^5=243$ points (right).}
     \label{f18}
\end{figure}

Now we need a definition of the local discrepancy in dimension
2. For short, for a two-dimensional point set $\cP=\{\bsx_0,
\ldots, \bsx_{N-1}\}$ of $N$ points in $[0,1)^2$ we define the
local discrepancy (i.e., the deviation from ideal distribution) as
$$\Delta_{\cP}(x,y):=\frac{A([0,x)\times [0,y);\cP)}{N}-xy,$$ where  
$A([0,x)\times [0,y);\cP):=\#\{n \in \{0,1,\ldots,N-1\} \ : \ \bsx_n \in [0,x)\times [0,y)\}$.

\medskip

The link between the local discrepancies of Hammersley point sets
(in dimension 2) and van der Corput sequences (in dimension 1)
is almost trivial and has already been used in
Section~\ref{sec_Faure_meth}. We state it again in a lemma as it is the fundamental key that
permits transferring results for van der Corput sequences to
results for Hammersley point sets \cite[Lemma 3]{fau08} .
\begin{lemma}\label{HamvdC}
Let $m,N$, and $\lambda$ be positive integers with $1 \leq N \leq
b^m$ and $1\leq \lambda < b^m$. Then,
$$\Delta_{\cH_{b,m}^{\bssigma}} \left(\frac{\lambda}{b^m},\frac{N}{b^m} \right) = \Delta_{Y_b^\Sigma,N}\left (\frac{\lambda}{b^m} \right ).$$
\end{lemma}
\begin{proof}
The proof is almost immediate: from the definition of
$\cH_{b,m}^{\bssigma}$ the counting functions
$A(\tfrac{\lambda}{b^m},\tfrac{N}{b^m},\cH_{b,m}^{\bssigma})$ and
$A(\tfrac{\lambda}{b^m},N,Y_b^\Sigma)$ attain the same value, and on
the other hand we have
$b^m\tfrac{\lambda}{b^m}\tfrac{N}{b^m}=N\tfrac{\lambda}{b^m}$
(note that $ \#\cH_{b,m}^{\bssigma}=b^m$), so that the local
discrepancies are equal.
\end{proof}

From Lemmas \ref{DesLem} and \ref{HamvdC}, it is possible to
obtain (see \cite{fau08} for the details) the following result.
\begin{theorem}[Faure]\label{thmD*Ham}
For any $m \in \NN$ and any
$\bssigma=(\sigma_0,\ldots,\sigma_{m-1})\in \Sy_b^m$ we have, with
some $c_m \in [0,2]$,
$$b^m D_{b^m}^*(\cH_{b,m}^{\bssigma})= \max \left (\max_{1 \le N \le b^m}\sum_{j=1}^m \psi_b^{\sigma_{j-1},+} \bigg( \frac{N}{b^j} \bigg)\ , \ \max_{1 \le N \le b^m} \sum_{j=1}^m \psi_b^{\sigma_{j-1},-} \bigg( \frac{N}{b^j} \bigg)\right)+c_m.$$
\end{theorem}

\begin{corollary}\label{corgenhamdisc}
For any $m\in \NN$ and any
$\bssigma=(\sigma_0,\ldots,\sigma_{m-1})\in \Sy_b^m$ we have, with
some $c_m \in [-2,2]$,
$$D_{b^m}^*(\cH_{b,m}^{\bssigma}) \le D_{b^m}^*(\cH_{b,m})+\frac{c_m}{b^m}.$$
\end{corollary}

For the star discrepancy of the classical Hammersley point set in
dimension 2 we have the following sharp result according to 
De Clerck~\cite{DeCl} (see also \cite{HaZa} and \cite{LP2003} for
the case $b=2$):

\begin{theorem}[De Clerck]
For $m \ge 2$ we have:
\begin{itemize}
\item if $m \ge 2$ is odd, then $$b^m D_{b^m}^*( \cH_{b,m})=m
\frac{b-1}{4} +\left(\frac{5}{4}+\frac{1}{b}\right)-\frac{1}{4
b^m};$$ \item if $m \ge 2$ is even, then $$b^m D_{b^m}^*(
\cH_{b,m})=m \frac{b^2}{4(b+1)}
+\left(\frac{5}{4}+\frac{2b+3}{4(b+1)^2}\right)-\frac{1}{4
b^m}\left(1+\frac{2b+3}{(b+1)^2}\right).$$
\end{itemize}
\end{theorem}

Note that $m=(\log b^m)/\log b$ and $b^m=\#\cH_{b,m}$. Hence De Clerck's result 
(together with Corollary~\ref{corgenhamdisc}) shows that the star discrepancy 
of the two-dimensional (generalized) Hammersley point set is of order 
$D_{b^m}^*(\cH_{b,m})=O((\log b^m)/b^m)=O((\log N)/N)$, 
which is optimal with respect to the order of magnitude in 
$N=b^m$ according to W.M. Schmidt's general lower bound on the 
star discrepancy of two-dimensional point sets (cf. Remark~\ref{genfactsdisc}). 
We mention that De Clerck~\cite{DeCl} also provides the ``bad'' intervals for 
which the star discrepancy is attained. Further results for the star discrepancy 
of generalized Hammersley point sets can be found in \cite{Fau08,K06a,KLP07}. 
The permutations $\bssigma$ can help to minimize the constant which is associated 
to the leading term $m$ in the star discrepancy formulas, see \cite{fau08} where several results are obtained with the help of the swapping permutation $\tau_b$ applied in various situations, getting analogs in arbitrary bases of results from \cite{K06a,KLP07} in base 2.\\

For the $L_p$-discrepancies the situation is quite different. We
have:

\begin{theorem}[Faure and Pillichshammer]\label{thm_lpham}
 For any integer base $b \ge 2$ and any $m \in \NN_0$ we have $$b^m L_{1,b^m}(\cH_{b,m})=m \frac{b^2-1}{12 b}+\frac{1}{2}+\frac{1}{4 b^m}$$ and
\begin{eqnarray*}
(b^m L_{2,b^m}(\cH_{b,m}))^2 & = & m^2 \left(\frac{b^2-1}{12 b}\right)^2+m \left(\frac{3 b^4+10 b^2-13}{720 b^2}+  \frac{b^2-1}{12 b}\left(1-\frac{1}{2 b^m}\right)\right)\\
&&+\frac{3}{8}+\frac{1}{4 b^m}-\frac{1}{72  b^{2 m}}.
\end{eqnarray*}
Furthermore, for arbitrary $p \in [1,\infty)$ we have
$$(b^mL_{p,b^m}(\cH_{b,m}))^p=m^p \left(\frac{b^2-1}{12
b}\right)^p+O_{p,b}(m^{p-1}),$$ where the constant in the
$O$-notation only depends on $p$ and $b$. In particular, $$\lim_{m
\rightarrow \infty}\frac{b^m L_{p,b^m}(\cH_{b,m})}{\log
b^m}=\frac{b^2-1}{12 b \log b}.$$
\end{theorem}

A proof of this result can be found in \cite{FauPi08}.
Theorem~\ref{thm_lpham} implies that the $L_p$-discrepancy for
finite $p$ is not of optimal order according to the lower bounds
of Hal\'{a}sz, Roth, and Schmidt for arbitrary $N$-element point
sets (which is $L_{p,N}(\cP) \gg \sqrt{\log N}/N$ in dimension 2,
cf. Remark~\ref{genfactsdisc}). This defect of the classical
Hammersley point sets can be overcome when we consider generalized
Hammersley point sets. There are various papers which provide
sequences of permutations which lead to the optimal order of
$L_2$-discrepancy (see, for example,
\cite{FauPi08,FauPi09,FauPi11,FauPi13,FauPiPi13,FauPiPiSch,HaZa,
KP06,white}). We provide an exemplary result from \cite{FauPi08}.

\begin{theorem}[Faure and Pillichshammer]\label{optshiftL2Ham}
Let $\bssigma \in \{\id,\tau_b\}^m$, where $\tau_b(k):=b-k-1$ for
$k \in \ZZ_b$, and let $l(m)$ denote the number of components of
$\bssigma$ which are equal to $\id$. Then we have
\begin{eqnarray*}
(b^m L_{2,b^m}(\cH_{b,m}^{\bssigma}))^2 & = &
\displaystyle\left(\frac{b^2-1}{12 b}\right)^2 ((m-2 l(m))^2-m) +\frac{b^2-1}{12 b} \left(1-\frac{1}{2 b^m}\right)(2l(m)-m)\\
&& +m \frac{b^4-1}{90 b^2}+\frac{3}{8}+\frac{1}{4
b^m}-\frac{1}{72b^{2m}}.
\end{eqnarray*}
In particular, the $L_2$-discrepancy of $\cH_{b,m}^{\bssigma}$
becomes minimal if we choose a sequence of permutations $\bssigma$
in which exactly $$l_{\min}(m):=\left\{
\begin{array}{ll}
m/2-2 & \mbox{ if } b=2,\\
m/2-1 & \mbox{ if } 3 \le b \le 6,\\
m/2 & \mbox{ if } b \ge 7,
\end{array}\right.$$ elements are the identity in the case of even $m$, and $$l_{\min}(m):=\left\{
\begin{array}{ll}
(m-3)/2 & \mbox{ if } 2 \le b \le 3,\\
(m-1)/2 & \mbox{ if } b \ge 4,
\end{array}\right.$$ elements are the identity in the case of odd $m$. In any case we have

\begin{equation}\label{ungen}
\min_{\bssigma \in \{\id,\tau_b\}^m} (b^m
L_{2,b^m}(\cH_{b,m}^{\bssigma}))^2 = m \frac{(b^2-1) (3
b^2+13)}{720 b^2}+O(1).
\end{equation}
\end{theorem}

A generalization of these results and a detailed discussion can be
found in \cite{FauPiPiSch}. Equation~\eqref{ungen} shows that the optimal permutations lead to
an $L_2$-discrepancy of order $O(\sqrt{m}/b^m)=O(\sqrt{\log N}/N)$
where $N=b^m$ and this is optimal according to the general lower
bound of Roth~\cite{Roth}. It is interesting that the minimal
$L_2$-discrepancy only depends on the number of $\id$-permutations
and not on their position within $\bssigma$. Not quite exact as
the result for $L_{2,N}$, but with the same essence, is the next
result for the $L_p$-discrepancy (with $p \in [1,\infty)$) of
$\cH_{2,m}^{\bssigma}$.

\begin{theorem}[Markhasin]\label{thmLev}
Let $p \in [1,\infty)$. Let $\bssigma \in \{\id,\tau_2\}^m$ and
let $l(m)$ be the number of components of $\bssigma$ which are
equal to $\id$. Then we have $$L_{p,2^m}(\cH_{2,m}^{\bssigma})
\ll_p \frac{\sqrt{m}}{2^m} \ll \frac{\sqrt{\log N}}{N}\ \ \ \mbox{
where $N=2^m$} $$ if and only if  $|2 l(m)-m| \ll_p \sqrt{m}$.
\end{theorem}

The proof of this result according to Markhasin
\cite{Lev13d,Lev13} is indirect via optimality of the norm of the
discrepancy function in Besov spaces with dominating mixed
smoothness together with embedding theorems between Besov spaces
and Triebel-Lizorkin spaces which contain $L_p$-spaces as special
cases. The main tool there is the computation of Haar coefficients
of the discrepancy function, which, for the generalized Hammersley
point sets, were already computed in \cite{hin2010}. A direct
proof of Theorem~\ref{thmLev} via Littlewood-Paley theory, which
is accessible without knowledge of function space theory, can be
found in \cite{HKP14}. Although Theorem~\ref{thmLev} was proved
only in base $b=2$ a similar result should hold for general bases
$b$. This is work in progress.

We close this section by mentioning that besides permutations also
a symmetrization of the classical Hammersley point set leads to
the optimal order of $L_p$-discrepancy with respect to the order of magnitude
in $N$ (see, e.g., \cite[Theorems~2 and 3]{HKP14}). Furthermore, 
also geometric shifts of the Hammersley point set and
their influence on $L_2$-discrepancy were studied by Bilyk~\cite{B09}.


\subsection{The Halton sequence}\label{sec_halt}

We now come to a very natural extension of van der Corput
sequences to arbitrary dimensions $s \in \NN$.

\paragraph{The classical Halton sequence.} Let $s \in \NN$. A natural generalization of the one-dimensional
van der Corput sequence to dimension $s$ is to choose $s$
different bases $b_1,b_2,\ldots,b_s \ge 2$ and to concatenate
component-wise the van der Corput sequences with bases
$b_1,b_2,\ldots,b_s$, respectively, to a sequence in the
$s$-dimensional unit cube. Such a sequence is called a Halton
sequence.

\begin{definition}\rm
Let $s \ge 1$ and let $\bsb=(b_1,b_2,\ldots,b_s)$ be a vector of
integers greater than 1. The {\it $s$-dimensional Halton sequence
in bases $b_1,b_2,\ldots,b_s$} is given as $H_{\bsb}=(\bsy_n)_{n
\ge 0}$ where $$\bsy_n=(Y_{b_1}(n),Y_{b_2}(n),\ldots ,Y_{b_s}(n))
\ \ \ \mbox{ for $n \in \NN_0$.}$$
\end{definition}

A minimal requirement on the Halton sequence is that it should be
uniformly distributed modulo 1. This depends on the choice of the
bases $b_1,b_2,\ldots,b_s$, as can be seen in the following theorem.

\begin{theorem}\label{thmHalton}
The Halton sequence $H_{\bsb}$ is uniformly distributed modulo 1
if and only if $b_1,b_2,\ldots,b_s$ are pairwise coprime.
\end{theorem}

The result in Theorem \ref{thmHalton} follows by similar arguments
as used in the proof of uniform distribution modulo 1 of the van der
Corput sequence together with the Chinese Remainder Theorem.

From now on we assume that $\bsb$ is a vector of pairwise coprime
integers. In this case the Halton sequence even has a low
discrepancy of order of magnitude $O((\log N)^s/N)$. This was first
shown by Halton~\cite{halton}. The following result is due to
Niederreiter~\cite{niesiam} and can be shown by a tricky
induction argument. Further discrepancy estimates for the Halton
sequence can be found in \cite{DP10,fau1980,HuaWang,LP14,Meij}.

\begin{theorem}[Niederreiter]\label{diha}
If $H_{\bsb}$ is the Halton sequence in pairwise coprime bases
$b_1,b_2,\ldots,b_s$, then $$D_N^*(H_{\bsb}) <
\frac{s}{N}+\frac{1}{N}\prod_{j=1}^s\left(\frac{b_j -1}{2 \log
b_j} \log N +\frac{b_j+1}{2}\right) \ \ \ \mbox{ for all }\ N \in
\NN.$$
\end{theorem}

It was a long standing open question if the discrepancy estimate
for Halton sequences is best possible with respect to the order of magnitude in
$N$. Until recently one only had the general lower bound due to 
Kuipers and Niederreiter~\cite{kuinie}  (which is based on the
result of Roth~\cite{Roth} for finite point sets) which
states that there exists some $c_s >0$ such that for every
sequence $X$ in $[0,1)^s$ we have
\begin{equation}\label{eqproinovdims}
 D_N^*(X) \ge c_s \frac{(\log N)^{s/2}}{N} \ \ \ \mbox{ for infinitely many $N\in \NN$}.
\end{equation}
We add that there is a slight but remarkable improvement of the
exponent of the $\log N$-term. It follows from a result of
Bilyk, Lacey, and Vagharshakyan~\cite{BLV08} for finite
point sets that there exists some $c_s >0$ and $\eta_s\in
(0,\tfrac{1}{2})$ such that for every sequence $X$ in $[0,1)^s$ we
have
\begin{equation}\label{lbdbillayvar}
D_N^*(X) \ge c_s \frac{(\log N)^{s/2+\eta_s}}{N} \ \ \ \mbox{for
infinitely many $N \in \mathbb{N}$.}
\end{equation}
For $s \ge 2$ the sharp lower bound for the star discrepancy of
infinite sequences is still unknown. It is conjectured that it is
significantly larger than the one given in \eqref{lbdbillayvar}.
In general, finding a sharp lower bound seems to be
a very difficult problem. We also refer to \cite{BL13} for a
discussion of these questions.

In 2015 Levin~\cite{levin15} obtained a solution of this
question for the important sub-class of Halton sequences. 
He could show that the existing upper bounds on the star
discrepancy of the Halton sequence are optimal with respect to the order of
magnitude in $N$.

\begin{theorem}[Levin]
Let $s \ge 2$ and let $B:=b_1 b_2 \cdots b_s$ and $m_0=\lfloor2 B
\log_2 b_0\rfloor +2$. Then $$\sup_{1 \le N \le 2^{m m_0}} N
D_N^*(H_{\bsb}) \ge \frac{m^s}{8 B} \ \ \ \mbox{ for all $m \ge
B$.}
$$ 
In particular, there exists some $c_{\bsb,s}>0$ such that
$$
D_N^*(H_{\bsb}) \ge c_{\bsb,s} \frac{(\log N)^s}{N} \ \ \ \mbox{
for infinitely many $N \in \NN$.}
$$
\end{theorem}

\paragraph{Atanassov's results.} The bound in Theorem~\ref{diha} has a slightly defective
appearance when one considers the constant in the leading $\log
N$-term. Consider the quantity
\begin{equation}\label{dast}
d_s^{\ast}(X):= \limsup_{N \rightarrow \infty} \frac{N
D_N^*(X)}{(\log N)^s}
\end{equation}
for Halton sequences $X=H_{\bsb}$. From Theorem~\ref{diha} it
follows that $$d_s^{\ast}(H_{\bsb}) \le \prod_{j=1}^s
\frac{b_j-1}{2 \log b_j}=: c_s.$$ This bound $c_s$ is large and
grows very fast to infinity for growing dimensions $s$. For
example, if $b_1,\ldots ,b_s$ are the first $s$ prime numbers,
then the Prime Number Theorem implies that $b_j \approx j \log j$
for large $j$ and hence
\begin{eqnarray*}
c_s = \prod_{j=1}^s\frac{b_j-1}{2\log b_j} \approx
\prod_{j=1}^s\frac{j \log j}{2 \log(j \log j)} \approx
\prod_{j=1}^s\frac{j \log j}{2 \log j}=\frac{s!}{2^s}.
\end{eqnarray*}
Also the bounds from \cite{fau1980,HuaWang,LP14,Meij} have the
same disadvantage. In 2004, Atanassov \cite{ata} was able to
overcome this particular disadvantage (see also the recent survey
\cite{fau14}, where Atanassov's method is outlined).

\begin{theorem}[Atanassov]\label{thmata}
If $H_{\bsb}$ is the Halton sequence in pairwise coprime bases
$b_1,b_2,\ldots,b_s$, then $$D_N^{\ast}(H_{\bsb})\le \left[
\frac{1}{s!}\prod_{j=1}^s\left(\frac{\lfloor b_j/2\rfloor \log
N}{\log b_j} +s\right)+
\sum_{k=0}^{s-1}\frac{b_{k+1}}{k!}\prod_{j=1}^k\left(\frac{\lfloor
b_j/2 \rfloor \log N}{\log b_j}+k\right)\right]\frac{1}{N}.$$
\end{theorem}

This result implies that $$d_s^{\ast}(H_{\bsb})\le \frac{1}{s!}
\prod_{j=1}^s\frac{\lfloor b_j/2\rfloor}{\log b_j}.$$ If again
$b_1,\ldots ,b_s$ are the first $s$ prime numbers, then this
implies that $d_s^{\ast}(H_{\bsb}) \ll  \frac{1}{2^s s}$ and hence
$d_s^{\ast}(H_{\bsb})$ tends to zero at an exponential rate as $s
\rightarrow \infty$. More exact, we have
\begin{equation}\label{asydsternhalt}
\limsup_{s \rightarrow \infty}\frac{\log d_s^*(H_{\bsb})}{s}\le
-\log 2.
\end{equation}

\paragraph{Generalized Halton sequences.}

A disadvantage of the classical Halton sequence is that there
seem to exist quite many correlations between the components from
the higher dimensions, see e.g. \cite{lem_book,LP14}. To overcome
this problem one often studies generalized Halton sequences
$H_{\bsb}^{\Sigma_1,\ldots,\Sigma_s}$ which are given by
$\bsy_n=(Y_{b_1}^{\Sigma_1}(n),Y_{b_2}^{\Sigma_2}(n),\ldots,Y_{b_s}^{\Sigma_s}(n))$
for $n \in \NN_0$, where $\Sigma_j=(\sigma_{j,r})_{r \ge 0}$ are
sequences of permutations of $\{0,1,\ldots,b_j\}$ for $1 \le j \le
s$; see again \cite{lem_book} and the references therein. The
discrepancy estimate from Theorem~\ref{thmata} remains valid also
for generalized Halton sequences. However, it is important to note
that the permutations can even lead to a drastic reduction of the
star discrepancy of $H_{\bsb}$. For example, in \cite{ata}
Atanassov constructed specific sequences of permutations
$\Sigma_1,\ldots,\Sigma_s$ which lead to exactly such a result:

Let $b_1,\ldots,b_s$ be distinct prime numbers. Then integers
$k_1,\ldots,k_s$ are called ``admissible'' to these primes, if
$b_j \nmid k_j$ for $1 \le j \le s$ and if for each set of
integers $a_1,\ldots,a_s$ with $b_j \nmid a_j$ for $1 \le j \le
s$, there exist integers $\alpha_1,\ldots,\alpha_s$ such that
$$k_j^{\alpha_j} \prod_{i=1 \atop i \not=j}^s b_i^{\alpha_i}
\equiv a_j \pmod{b_j}\ \ \ \mbox{ for all $1 \le j \le s$.}$$

Now let $b_1,\ldots,b_s$ be distinct prime numbers and let
$k_1,\ldots,k_s$ be ``admissible'' to these primes. For each $j
\in \{1,\ldots,s\}$ define a sequence of permutations
$\Sigma_j=(\tau_{j,r})_{r \ge 0}$ of $\ZZ_{b_j}$ by $
\tau_{j,r}(x):=x k_j^r \pmod{b_j}$. For these so-called
scramblings we have (\cite[Theorem~2.3]{ata})

\begin{theorem}[Atanassov]
Let $b_1,\ldots,b_s$ be distinct prime numbers and let
$k_1,\ldots,k_s$ be ``admissible'' to these primes. For the
sequences of permutations constructed above we have
$$D_N^*(H_{\bsb}^{\Sigma_1,\ldots,\Sigma_s}) \le \left(\frac{1}{s!} \sum_{j=1}^s \log b_j\right) \left(\prod_{i=1}^s \frac{b_i (1+\log b_i)}{(b_i-1) \log b_i}\right) \frac{(\log N)^s}{N}+O\left(\frac{(\log N)^{s-1}}{N}\right).$$
\end{theorem}

From this discrepancy bound it follows easily (see, for example,
\cite{FL09}) that
\begin{equation}\label{asydsternhalt1}
\limsup_{s \rightarrow \infty}\frac{\log
d_s^*(H_{\bsb}^{\Sigma_1,\ldots,\Sigma_s})}{s\log s}=-1,
\end{equation}
and this is a drastic improvement compared to
\eqref{asydsternhalt}, since now
$d_s^*(H_{\bsb}^{\Sigma_1,\ldots,\Sigma_s})$ tends to zero at a
superexponential rate as $s \rightarrow \infty$.

\paragraph{Randomized and deterministic scrambled Halton sequences.}
Another possibility to overcome the problem of correlations
between the components from the higher dimensions of Halton
sequences is to use so-called ``random-start'' Halton sequences
which are based on generalizing the Halton sequence in terms of
von Neumann-Kakutani transformations, i.e., the sequence
$((T_{b_1}^n x_1,T_{b_2}^n x_2,\ldots, T_{b_s}^n x_s))_{n \ge 0}$
with $(x_1,x_2,\ldots,x_s)\in [0,1)^s$; see
\cite{oekt,str95,WH2000}.

One could also consider linear digit scramblings: Various types of
linear digit scramblings as introduced at the end of Section
\ref{secdig01seq} have been used by many authors with the aim to
improve the behavior of Halton sequences for multi-dimensional
integration based on quasi-Monte Carlo methods (note that also
Atanassov's permutations $\tau_{j,r}$ introduced above are
specific examples of linear digit scramblings). An extensive
comparative study of such methods is given in \cite{FL09} where a
new effective way of scrambling is also proposed. This method
consists of the search for a list of good multipliers by means of
the diaphony of one-dimensional projections of Halton sequences,
coupled with a criterion based on two-dimensional projections that
selects the best multipliers in the preceding list (the diaphony
is considered, rather than the discrepancy,  because it is easier and
faster to compute, see \cite{Fau06} for more details). Several
numerical experiments show a regular and better behavior of these
multipliers in comparison with other ones (including multipliers
from the modified Halton sequences of Atanassov). The study also
compares linear digit scramblings with ``random-start'' Halton
sequences and suggests that these sequences have almost the same
behavior as original Halton sequences.

\paragraph{Further remarks.}
Bounded remainder intervals for Halton sequences are well
understood and it is known that they are of the form $[0,a_1
b_1^{-m_1})\times \cdots \times [0,a_s b_s^{-m_s})$. We refer to
\cite[Theorem~70]{GHL} and the references therein. Subsequences of
the Halton sequence are studied in \cite{HelNie2011,HLKP,KrLaPi}.
Ergodic properties of $\beta$-adic Halton sequences are studied in
\cite{HIT}.

Also the polynomial version of the van der Corput sequence as
presented in Section~\ref{secpolyvdc} can be generalized to
polynomial Halton sequences which are a special sub-class of
$(t,s)$-sequences then; see \cite{tez1993}.

\paragraph{The Hammersley point set.}
A finite version of the Halton sequence in dimension $s-1$ is the
Hammersley point set in dimension $s$ (cf.
Definition~\ref{defgenHam} for $s=2$).

\begin{definition}[Hammersley point set]\rm
Let $s, N \in \NN$, $s \ge 2$, and let
$\bsb=(b_1,b_2,\ldots,b_{s-1})$ be a vector of integers with $b_j
\ge 2$. The {\it $s$-dimensional Hammersley point set in bases
$b_1,b_2,\ldots,b_{s-1}$} is given as
$\cH_{\bsb,N}=\{\bsy_0,\bsy_1,\ldots,\bsy_{N-1}\}$ where
$$\bsy_n=\left(\frac{n}{N},Y_{b_1}(n),Y_{b_2}(n),\ldots
,Y_{b_{s-1}}(n)\right) \ \ \ \mbox{ for $n=0,1,\ldots,N-1$.}$$
\end{definition}

According to a well-known relation between the star discrepancy of
infinite sequences in dimension $s-1$ and finite point sets in
dimension $s$, whose elements have first coordinate of the form
$n/N$ (see, e.g., \cite[Lemma~3.45]{DP10} or
\cite[Lemma~3.7]{niesiam}) it follows from Theorem~\ref{diha} that
if $b_1,b_2,\ldots,b_{s-1}$ are pairwise coprime integers, then
$$D_N^*(\cH_{\bsb,N}) \ll_{s,\bsb} \frac{(\log N)^{s-1}}{N}.$$
See, e.g., \cite[Chapter~3.4]{DP10} or \cite[Chapter~3]{niesiam}
for more information in this direction. For generalized Hammersley
point sets we refer to \cite{Fau1986a}.


\subsection{$(t,s)$-sequences}\label{sec_tssequ}

In Section~\ref{sec_halt} we generalized the van der Corput sequence 
to arbitrary dimension by choosing different bases, one per coordinate 
direction, and concatenating component-wise van der Corput sequences 
in these bases. We now introduce a generalization to arbitrary dimension 
where the base $b$ is kept fixed for all dimensions. 

\paragraph{Definition of $(t,s)$-sequences.} In the one-dimensional case 
(see Definition \ref{def01seq} in Section~\ref{secdig01seq}), we defined a 
$(0,1)$-sequence in base $b$ as a sequence of points such that each elementary
interval of a certain size contains exactly one element of the
corresponding segment of the sequence. In higher dimensions, the
idea of a $(t,s)$-sequence in base $b$ generalizes this concept.
However, the condition that then higher-dimensional elementary
intervals all contain exactly one point of the segments of the
sequence is, for combinatorial reasons, no longer sustainable
in general. Therefore, one needs additional parameters to specify
how many points lie in the elementary intervals. To be more
precise, we have the following definition that generalizes
Definition \ref{def01seq}. The general idea of this definition is
due to Niederreiter \cite{N87}, but there also were earlier
definitions of a similar flavor for special cases by Sobol'
\cite{S67} and Faure \cite{Fau1982}.

\begin{definition}\label{deftsseq}\rm
Let $b\ge 2$, $t\ge 0$, and $s\ge 1$ be integers. A sequence
$(\bsx_n)_{n\ge 0}$ in $[0,1)^s$ is called a {\it $(t,s)$-sequence
in base $b$ (in the broad sense)} if for all $l,m\in\NN_0$ with
$m\ge t$, every elementary interval of volume $b^{t-m}$ of the
form
$$\left[\frac{a_1}{b^{d_1}},\frac{a_1+1}{b^{d_1}}\right)\times\cdots\times\left[\frac{a_s}{b^{d_s}},\frac{a_s+1}{b^{d_s}}\right),$$
with $a_j,d_j\in\NN_0$ such that $0\le a_j < b^{d_j}$ for $1\le
j\le s$, contains exactly $b^t$ elements of the point set
$$\{[\bsx_n]_{b,m}:\ lb^m\le n\le (l+1)b^m-1\}\ \ \ \mbox{ for all $l \in \NN_0$.}$$ Here the truncation operator $[\cdot]_{b,m}$ is applied component-wise.
\end{definition}

Some remarks on this definition are in order:
\begin{remark}\rm
\begin{itemize}
\item The original definition of $(t,s)$-sequences in base $b$ due
to Niederreiter~\cite{N87,niesiam} is the same but with
$[\bsx_n]_{b,m}$ replaced by $\bsx_n$. The definition presented
here goes back to Niederreiter and Xing (see, e.g.,
\cite[Chapter~8]{NX_book}). \item It is well known that every
$(t,s)$-sequence in base $b$ is uniformly distributed modulo 1
(see, e.g., \cite{DP10,LP14}). \item If one would like to have
greater flexibility in the definition of a $(t,s)$-sequence, it is
also possible to allow the parameter $t$ to depend on the
parameter $m\in\NN_0$ in the above definition. In this case, the
value of $t$ can vary with $m$, and instead of a $(t,s)$-sequence,
one then has a $({\bf T},s)$-sequence with ${\bf T}=(t(m))_{m\ge
0}$. For the sake of simplicity, we concentrate on discussing only
$(t,s)$-sequences with a fixed $t$ in the present paper. For
further information on $({\bf T},s)$-sequences, we refer the
interested reader to \cite{LN95} and \cite[Chapters 4 and
5]{DP10}.

\item Note that in Definition \ref{deftsseq} lower values of $t$
mean better distribution properties, with $t=0$ being the optimal
case. For this reason $t$ is called the {\it quality parameter} of
the sequence. For the one-dimensional sequences presented in
Section \ref{secdig01seq}, a value of $t=0$ can be reached,
however for general sequences in $[0,1)^s$ this is not always
achievable. For instance, a $(0,s)$-sequence in base $b$ can only
exist if $s \le b$ (if $b$ is a prime power, then the ``only if''
can be replaced by ``if and only if''). The question which $t$ for
which parameter configuration $(s,b)$ is optimal has gained much
interest. In \cite{NX1996} it was shown that for the minimal
achievable $t$ we have 
$$t
\ge \frac{s}{b}-\log_b\left( \frac{(b-1)s+b+1}{2}\right),
$$ 
where $\log_b$ is the logarithm to the base $b$. For
further information in this direction we also refer to
\cite{mint}.
\end{itemize}
\end{remark}

As it was the case for $(0,1)$-sequences, also for
$(t,s)$-sequences in base $b$ there exists a digital version
generated by $\NN \times \NN$ matrices over a finite ring $R$. As
in the one-dimensional case we restrict ourselves to the case
where $b$ is a prime power and $R$ is the finite field $\FF_b$ of
order $b$. Indeed, $s$-dimensional digital sequences over $\FF_b$
are obtained by choosing, for each $j\in\{1,\ldots,s\}$, an
$\NN\times\NN$ matrix $C^{(j)}=(c_{r,k}^{(j)})_{r,k\ge 0}$ over
$\FF_b$. Using the construction method outlined in
Definition~\ref{Defdig01}, one obtains a sequence
$X_b^{C^{(j)}}=(x_n^{(j)})_{n\ge 0}$ for each
$j\in\{1,\ldots,s\}$. The $s$-dimensional sequence then consists
of the points
\begin{equation}\label{DigSeq}
\bsx_n:=(x_n^{(1)},\ldots,x_n^{(s)}) \ \ \ \ \mbox{ for $n\in
\NN_0$.}
\end{equation}
The sequence $(\bsx_n)_{n \ge 0}$ is often called a {\it digital
sequence over $\FF_b$}.

The matrices $C^{(j)}$ also determine the $t$-value of the
corresponding digital sequence. Indeed, let $C^{(j)}(m \times m)$
be the left upper $m \times m$ submatrix of $C^{(j)}$. If there
exists a $t \in \NN_0$ such that for every choice of nonnegative
integers $m>t$ and $d_1,\ldots,d_s\in \NN_0$ with $d_1+\cdots +
d_s=m-t$,
\begin{center}
\begin{itemize}
\setlength{\itemsep}{-2pt} \item[] the first $d_1$ rows of
$C^{(1)}(m\times m)$, together with \item[] the first $d_2$ rows
of $C^{(2)}(m \times m)$, together with
\item[] etc., up to\\
\item[] the first $d_s$ rows of $C^{(s)}(m \times m)$,
\end{itemize}
\end{center}
(these are $m-t$ vectors in $\FF_b^m$) is a linearly independent set over $\FF_b$, 
then the digital sequence \eqref{DigSeq} is a $(t,s)$-sequence in base 
$b$ which is then called a {\it digital $(t,s)$-sequence over $\FF_b$} 
(see \cite{DP10,LP14,N87,niesiam}). Hence, finding a $(t,s)$-sequence in base 
$b$ with a low value of $t$ boils down to the problem of finding generating matrices 
$C^{(1)},\ldots, C^{(s)}$ whose row vectors satisfy the necessary linear independence conditions. \\

The dynamics behind digital $(t,s)$-sequences was studied by 
Grabner, Hellekalek, and Liardet. We do not go into details
here and just refer to the papers \cite{GHL} and \cite{hellia}.

\paragraph{Generalized Niederreiter sequences.} A particular kind of digital $(t,s)$-sequences are so-called 
Niederreiter sequences, whose points are constructed by means of polynomials over finite fields. 
These sequences
offer an explicit way of constructing the generating matrices of
powerful digital $(t,s)$-sequences over $\FF_b$. They were
introduced in \cite{N88}, and their definition was further
extended to generalized Niederreiter sequences by Tezuka in
\cite{tez1995}. Following the presentation in \cite{DP10}, a
generalized Niederreiter sequence is obtained as follows.

Let $s\in\NN$, let $b$ be a prime power, and let $p_1,\ldots,p_s$
be distinct monic irreducible polynomials over $\FF_b$. Let
$e_j:=\deg (p_j)$ for $1\le j\le s$. Moreover, choose polynomials
$y_{j,i,k}$ for $1\le j\le s$, $i\ge 1$, and $0\le k< e_j$ with
the property that the set $\{y_{j,i,k}: 0\le k< e_j\}$ is
linearly independent modulo $p_j$ over $\FF_b$. Then consider
the expansions
$$\frac{y_{j,i,k}(x)}{p_j(x)^i}=\sum_{r=0}^\infty a^{(j)}(i,k,r)x^{-r-1}$$
over the field $\FF_b ((x^{-1}))$ for $1\le j\le s$, $i\ge 1$, and
$0\le k< e_j$. We use these expansions to define the generating
matrices of a digital sequence; for $1\le j\le s$ let
$C^{(j)}=(c_{i,r}^{(j)})_{i\ge 1, r\ge 0}$ be given by
\begin{equation}\label{defmatnies}
c_{i,r}^{(j)}=a^{(j)} (Q+1,k,r)\in\FF_b
\end{equation}
for $1\le j\le s$, $i\ge 1$, and $0\le k< e_j$, where $i-1=Qe_j +
k$ with suitable integers $Q=Q(j,i)$ and $k=k(j,i)$ with $0\le k<
e_j$.

\begin{definition}\rm
A digital sequence over $\FF_b$ generated by $C^{(1)},\ldots,
C^{(s)}$ as given by \eqref{defmatnies} is called {\it generalized
Niederreiter sequence}.
\end{definition}

Every generalized Niederreiter sequence is a digital
$(t,s)$-sequence over $\FF_b$ with $$t=\sum_{j=1}^s (e_j-1),$$
see, e.g., \cite{DP10,LP14,niesiam}.

The concept of generalized Niederreiter sequences comprises
earlier constructions due to Sobol'~\cite{S67},
Faure~\cite{Fau1982}, and Niederreiter~\cite{N88}, which are
nowadays very popular as underlying sample points for quasi-Monte
Carlo integration rules (cf. \cite{DP10,niesiam}).
\begin{itemize}
\item A {\it Sobol' sequence} is the generalized Niederreiter
sequence where $b = 2$, $p_1(x) = x$, and, for $j=2,\ldots,s$,
$p_j(x)$ is the $(j-1)$-st primitive polynomial in a sequence of
all primitive polynomials over $\FF_2$ arranged according to
nondecreasing degrees. Further, there are polynomials
$g_{j,0},\ldots, g_{j,e_j-1}$ with $\deg(g_{j,h}) = e_j-h+1$ such
that $y_{j,i,k} = g_{j,k}$ for all $i \in \NN$, $k \in
\{0,\ldots,e_j-1\}$, and $j\in \{1,\ldots,s\}$, see
\cite[Remark~2]{tez1995}. \item A {\it Faure sequence} corresponds
to the case where the base $b$ is a prime number such that $b \ge
s$, $p_j(x) = x-j+1$ for $j=1,2,\ldots, s$, and all $y_{i,j,k} \equiv
1$, see \cite[Remark~3]{tez1995}. In particular, every Faure
sequence is a $(0,s)$-sequence. The matrices of a Faure sequence
can be given explicitly; there exist $s$ pairwise
different elements of $\FF_b$, $\beta_1,\ldots,\beta_s$, which are determined 
by the choice of the $p_j$ such that the $i$th matrix is given as
$$C_i:=\left(
\begin{array}{cccccc}
1 & {1 \choose 0} \beta_i^1 & {2 \choose 0} \beta_i^2 & {3 \choose 0} \beta_i^3 & {4 \choose 0} \beta_i^4 &  \ldots\\ \\
0 & 1 & {2 \choose 1} \beta_i^1 & {3 \choose 1} \beta_i^2 & {4 \choose 1} \beta_i^3 & \ldots \\ \\
0 & 0 & 1 & {3 \choose 2}\beta_i^1 & {4 \choose 2}\beta_i^2 & \ldots \\ \\
\vdots & \ddots  & \ddots & 1 & {4 \choose 3} \beta_i^1& \ldots \\ \\
\vdots & \ddots & \ddots & \ddots & \ddots & \ddots
\end{array}
\right) \in \FF_b^{\NN \times \NN},$$ where the binomial
coefficients have to be taken modulo $b$. \item A {\it
Niederreiter sequence} corresponds to the case where $b$ is a
prime power and the $y_{j,i,k}$ are of the form $y_{j,i,k}(x)=x^k
g_{j,i}(x)$ with polynomials $g_{j,i}\in \FF_b[x]$ such that
$\gcd(g_{j,i},p_j)=1$, see \cite[Remark~4]{tez1995}.
\end{itemize}

The currently best constructions for digital sequences with
respect to the quality parameter $t$ are based on methods from
algebraic geometry. These constructions were developed in a series
of papers by Niederreiter and Xing. For an introduction
to this subject and an overview we refer to the book by
Niederreiter and Xing \cite[Chapter~8]{NX_book}. A survey can also
be found in the book \cite[Chapter~8]{DP10}.

\medskip

Linear digit scramblings have also been used to improve the
quality of $(0,s)$-sequences for quasi-Monte Carlo methods, with
the same approach as for Halton sequences in
Section~\ref{sec_halt}, see \cite{LeFe}.

\paragraph{The star discrepancy of $(t,s)$-sequences.} Regarding the discrepancy of $(t,s)$-sequences, 
it is, not surprisingly, more difficult than in the one-dimensional case to come 
up with good discrepancy bounds. Let us start with discussing the star discrepancy,
for which there has been much progress over the past years.
Regarding lower discrepancy bounds, the best results for dimensions 
$s\ge 2$ are the general bounds \eqref{eqproinovdims} and
\eqref{lbdbillayvar}, excepted in the special case of the two-dimensional $(0,2)$-sequence of Sobol' (also a Faure sequence), say $S_2$, obtained with the van der Corput sequence $Y_2$  on the $x$-axis and its image by the Pascal matrix modulo 2 on the $y$-axis. For this sequence, Faure and Chaix \cite{fauch} were able to obtain the exact order $(\log N)^2$ for $ND^*_N(S_2)$ with a tight lower bound. Together with an upper bound from Pillichshammer in \cite{Pil03}, we have the following estimate:
\begin{theorem}[Faure and Chaix, Pillichshammer]\label{estim S2}
The sequence $S_2$ satisfies the inequalities
\[
\frac{1}{24(\log 2)^2} \le \limsup_{N \rightarrow \infty} \frac{ND^*_N(S_2)}{(\log N)^2} \le \frac{1}{12(\log 2)^2}\cdot
\]
\end{theorem}

Until now, this is the only case of a $(t,s)$-sequence in dimension $s \ge 2$ for which the exact order of star discrepancy is known. Based on thorough computations, Faure and Chaix conjectured that the left inequality should be an equality.

\medskip

However, for any dimension, there is a metric result of Larcher
and Pillichshammer~\cite{lp14} for general digital sequences over
prime fields.

\begin{theorem}[Larcher and Pillichshammer]\label{metrbddisc}
Let $s \in \NN$ and let $b$ be a prime number. Then for almost
all\footnote{Here ``almost all'' is meant with respect to a
certain probability measure $\mu_s$ on the class of all $s$-tuples
of $\NN \times \NN$-matrices over $\FF_b$; see \cite{lp14}.}
$s$-tuples $(C_1,\ldots,C_s) \in (\FF_b^{\NN \times \NN})^s$ of
generating matrices the corresponding digital sequences have a star
discrepancy satisfying
$$D_N^\ast \ge c(b,s) \frac{(\log N)^s \log \log
N}{N} \ \ \ \mbox{ for infinitely many } N \in \NN$$ with some
$c(b,s)>0$ not depending on $N$.
\end{theorem}

The situation is different for upper discrepancy bounds, where one
can use the structure of $(t,s)$-sequences to show strong results
which are better than the metric lower bound from
Theorem~\ref{metrbddisc}. Upper bounds for the star discrepancy
that hold for arbitrary $(t,s)$-sequences are due to Niederreiter
\cite{N87}, which are based on a clever recursion argument. The
bounds in \cite{N87} were slightly refined in \cite{K06}, and in a
series of papers \cite{FK13, FL12, FL14}, further improvements
were shown. We shortly summarize some of the most up-to-date
results. Star discrepancy bounds for $(t,s)$-sequences can be seen
from two different perspectives. One may either try to optimize
the term involving the highest power of $\log N$, including the
constants, as this yields the asymptotically best discrepancy
bounds. On the other hand, one may try to optimize the bounds
globally such that for the non-asymptotic setting the bounds might
be better than those that are optimized with respect to the
leading term.

Regarding the former approach, Faure and Kritzer showed in
\cite{FK13} the currently best general star discrepancy bounds for
general $(t,s)$-sequences, where the constant in the leading term
is the lowest currently known. The following result was shown in
\cite{FK13} using a recursion on $s$ and a discrepancy bound for
finite point sets known as $(t,m,s)$-nets.
\begin{theorem}[Faure and Kritzer]\label{thmfkdiscbound}
 Let $s\ge 2$ and let $S_s$ be a $(t,s)$-sequence in base $b$. Then it is true that
$$D_N^\ast (S_s)\le c_{s,b}\, b^t \, \frac{(\log N)^s}{N} + O\left(b^t\,\frac{(\log N)^{s-1}}{N}\right),$$
where
$$c_{s,b}=\begin{cases}
       \frac{1}{s!}\frac{b^2}{2(b^2-1)}\left(\frac{b-1}{2\log b}\right)^s & \mbox{if $b$ is even,}\\
       \frac{1}{s!}\frac{1}{2}\left(\frac{b-1}{2\log b}\right)^s & \mbox{if $b$ is odd.}
      \end{cases}
 $$
\end{theorem}

If we consider the quantity $d_s^*$ as defined in \eqref{dast} for $(t,s)$-sequences in base $b$, 
then we obtain from Theorem~\ref{thmfkdiscbound} that $$d_s^*(S_s) \le c_{s,b} b^{t}.$$ 
It is known that for any $s,b \in \NN$, $b \ge 2$, one can construct a 
$(t_b(s),s)$-sequence $S_s^{\ast}$ in base $b$ with $t_b(s) \asymp_b s$. 
These constructions are based on the method of Niederreiter and Xing (see \cite[Chapter~8]{NX_book}). 
For such sequences we obtain 
$$
\limsup_{s \rightarrow \infty} \frac{\log  d_s^{\ast}(S_s^{\ast})}{s \log s} \le -1
$$
and hence $d_s^{\ast}(S_s^{\ast})$ tends to zero at a superexponential rate as 
$s \rightarrow \infty$. This is the same rate as for the special generalized 
Halton sequence given by Atanassov~\cite{ata}, cf. \eqref{asydsternhalt1} in Section~\ref{sec_halt}. 
Remember that the analog result for Halton sequences given in \eqref{asydsternhalt} 
is much weaker. A comparison  is also discussed in \cite{FL09}.\\

The paper \cite{FK13} contains also a result where all constants
in the star discrepancy bounds for $(t,s)$-sequences are made
explicit (cf. \cite[Theorem 2]{FK13}). However, the focus of
\cite{FK13} lies on optimizing the leading constant $c_s$. Faure
and Lemieux showed in \cite{FL12, FL14} further upper
discrepancy bounds which are not optimized regarding $c_s$ but
which in the non-asymptotic setting (i.e., for smaller values of
$N$) yield better discrepancy bounds than \cite[Theorem 2]{FK13}.
The bounds of Faure and Lemieux are as follows.
\begin{theorem}[Faure and Lemieux]\label{thmFauLe1}
 Let $S_s$ be a $(t,s)$-sequence in base $b$. For any $N\in \NN$ it is true that
$$D_N^\ast (S_s)\le\frac{b^t}{s!}\left(\left\lfloor\frac{b}{2}\right\rfloor\frac{\log N}{\log b}+s\right)^s\frac{1}{N} + \frac{b^t}{N}\sum_{k=0}^{s-1}\frac{b}{k!}
\left(\left\lfloor\frac{b}{2}\right\rfloor\frac{\log N}{\log
b}+k\right)^k.$$
\end{theorem}

Theorem~\ref{thmFauLe1} was shown in \cite{FL12} by an adaption of
Atanassov's method for Halton sequences to $(t,s)$-sequences.
Finally, in \cite[Theorem~1]{FL14}, the following
improvement of Theorem~\ref{thmFauLe1} was shown.
\begin{theorem}[Faure and Lemieux]\label{thmFauLe2}
Let $S_s$ be a $(t,s)$-sequence in base $b$. For any $N\in \NN$ it
is true that
$$D_N^\ast (S_s)\le\frac{b^t}{s!}\left(\frac{b-1}{2}\right)^s \frac{1}{N} \prod_{k=1}^s \left(\frac{\log N}{\log b} + \gamma_b k\right)
+ \frac{b^t}{N} \sum_{k=0}^{s-1}
\frac{b}{k!}\left(\frac{b-1}{2}\right)^k \prod_{\ell=1}^k
\left(\frac{\log N}{\log b} + \gamma_b \ell\right),$$ where
$\gamma_b:=2-(b \pmod 2)$.
\end{theorem}

Numerical experiments in \cite{FL15} for different values of $N$
and the other parameters show that the bounds in Theorems
\ref{thmFauLe1} and \ref{thmFauLe2} may have an advantage over
those in Theorem \ref{thmfkdiscbound} in the
non-asymptotic setting.

Finally, we mention that special discrepancy estimates for
generalized Niederreiter sequences were recently shown by
Tezuka in \cite{T13} and then improved for even bases by Faure and
Lemieux \cite[Theorem~2]{FL14}.
\begin{theorem}[Tezuka, Faure, and Lemieux]\label{thmFauLe3}
Let $S_s$ be a generalized Niederreiter sequence in dimension $s$.
Then it is true that
\begin{eqnarray*}
D_N^\ast (S_s)&\le&\frac{1}{s!} \frac{1}{N} \prod_{j=1}^s \left(\frac{b^{e_j}-1}{2e_j}\left(\frac{\log N}{\log b}+\sum_{i=1}^s e_i\right)+s\right)\\
&&+\frac{1}{N}\sum_{k=0}^{s-1} \frac{b^{e_{k+1}}}{k!}\prod_{i=1}^k
\left(\frac{b^{e_i}-1}{2e_i}\left(\frac{\log N}{\log
b}+\sum_{r=1}^s e_r\right)+k\right),
\end{eqnarray*}
where the $e_j$, $1\le j\le s$, are the degrees of the polynomials
$p_j$ needed in the construction of generalized Niederreiter
sequences as outlined above.
\end{theorem}

The bound in Theorem \ref{thmFauLe3}, as well as the earlier
result of Tezuka in \cite{T13}, were obtained by relating
generalized Niederreiter sequences to the new concept of so-called
$(t,\boldsymbol{e},s)$-sequences as introduced in \cite{T13}.
Indeed, generalized Niederreiter sequences can be viewed as
$(0,\boldsymbol{e},s)$-sequences, where
$\boldsymbol{e}=(e_1,\ldots,e_s)$, with the $e_j$ exactly
corresponding to the degrees of the polynomials $p_j$.

\paragraph{The $L_p$-discrepancy of $(t,s)$-sequences.} Regarding the 
$L_p$-discrepancy with $p \in [1,\infty)$, less is known for general 
$(t,s)$-sequences than in the case of the star discrepancy. 
In general, it is clear from the monotonicity of the $L_p$-discrepancies in $p$ that the upper bounds
on the star discrepancy of $(t,s)$-sequences are also valid for
the $L_p$-discrepancy. However, in view of a general lower bound
due to Proinov~\cite{pro86} for $p\in (1,\infty)$, which states
that there exists some $c_{s,p}>0$ such that for every infinite
sequence $X$ in $[0,1)^s$ we have 
$$
L_{p,N}(X) \ge c_{s,p}
\frac{(\log N)^{s/2}}{N}\ \ \ \mbox{ for infinitely many $N \in
\NN$,}
$$ 
the question remains whether there exist
$(t,s)$-sequences for which the $L_p$-discrepancy matches this
lower bound with respect to the order of magnitude in $N$. This question was
answered  in the affirmative by Dick and Pillichshammer for $p=2$
(see \cite{DP14a,DP14b}) and by Dick, Hinrichs, Markhasin, and
Pillichshammer~\cite{DHMP} for general $p$ recently.

\begin{theorem}[Dick, Hinrichs, Markhasin, and Pillichshammer]
For $p \in [1,\infty)$ and every dimension $s$ one can explicitly
construct a digital sequences $S_s$ over $\FF_2$ for which we have
$$L_{p,N}(S_s) \ll_{s,p} \frac{(\log N)^{s/2}}{N}.$$
\end{theorem}

The construction mentioned in the above theorem is based on
so-called higher order sequences which were introduced by
Dick~\cite{D07,D08}, and which are especially useful to obtain
optimal convergence rates for the integration of smooth functions
(see also \cite[Chapter~15]{DP10}). One particular instance of
such a construction are so-called digit-interlaced Niederreiter
sequences.

\paragraph{Conclusion.} Nowadays, digital $(t,s)$-sequences, 
which have their roots in the simple digital construction due to J.G. van der Corput, 
are among the most powerful constructions of sequences which are used as the
sample points for quasi-Monte Carlo integration rules 
$$\frac{1}{N}\sum_{n=0}^{N-1}f(\bsx_n) \approx \int_{[0,1]^s}f(\bsx)\rd \bsx.$$ 
The general link between the distribution properties of the underlying sample 
nodes and the absolute integration error of a quasi-Monte Carlo rule is provided 
by the famous Koksma-Hlawka inequality which was already stated in 
Remark~\ref{genfactsdisc}. Introductory texts on the quasi-Monte Carlo method and 
digital sequences are the books \cite{DP10,LP14,niesiam} and the survey \cite{DKS}.


\begin{small}
\noindent\textbf{Authors' addresses:}\\
\noindent Henri Faure, Aix-Marseille Universit\'{e}, CNRS, Centrale Marseille, I2M, UMR 7373, 13453 Marseille, France, E-mail: \texttt{henri.faure@univ-amu.fr}\\ \\
\noindent Peter Kritzer, Friedrich Pillichshammer, Department of
Financial Mathematics and Applied Number Theory, Johannes Kepler
University Linz, Altenbergerstr.~69, 4040 Linz, Austria, E-mail:
\texttt{peter.kritzer@jku.at},
\texttt{friedrich.pillichshammer@jku.at}
\end{small}
\end{document}